\documentclass[11pt,draft,twoside]{article}

\usepackage{amsmath,bm}
\usepackage{amssymb}
\usepackage{fancyhdr}
\usepackage{latexsym}
\usepackage{bbding}
\usepackage{mathrsfs}
\usepackage{exscale}
\usepackage{relsize}
\usepackage{wasysym}
\usepackage{cite}
\usepackage{multicol,graphics}
\usepackage{color}

\makeatletter
\@addtoreset{equation}{section}
\makeatother

\allowdisplaybreaks[4]
\def\to{\rightarrow}
\def\e{\left}
\def\r{\right}
\def\l{\leq}
\def\g{\geq}

\def\i{\infty}

\def\R{\mathbb{R}}

\def\p{\partial}
\def\P{\Phi}
\def\dl{\delta}
\def\vp{\varepsilon}

\def\bar{\overline}

\renewcommand\tilde{\widetilde}
\newcommand{\be}{\begin{equation}}
\newcommand{\ee}{\end{equation}}
\newcommand{\bc}{\begin{cases}}
\newcommand{\ec}{\end{cases}}
\newcommand{\bes}{\begin{equation*}}
\newcommand{\ees}{\end{equation*}}
\newcommand{\bls}{\begin{align*}}
\newcommand{\els}{\end{align*}}
\newcommand{\baa}{\begin{array}}
\newcommand{\eaa}{\end{array}}
\newcommand{\ba}{\begin{eqnarray}}
\newcommand{\ea}{\end{eqnarray}}
\newcommand{\bas}{\begin{eqnarray*}}
\newcommand{\eas}{\end{eqnarray*}}
\newcommand{\bd}{\begin{description}}
\newcommand{\ed}{\end{description}}

\newtheorem{theo}{\bf Theorem}[section]
\newtheorem{lem}[theo]{\bf Lemma}

\newtheorem{defi}[theo]{\bf Definition}
\newtheorem{rem}[theo]{\bf Remark}

\newenvironment{pr}[1][Proof]{\noindent\textbf{#1.} }{\hfill $\Box$}
\allowdisplaybreaks
\begin{document}
\date{}
\title{\bf{Transition fronts of monotone bistable reaction-diffusion systems around an obstacle}}
\author{Yang-Yang Yan$^{a,b}$ \quad  Wei-Jie Sheng$^{a,}$\thanks{Corresponding author
 (E-mail address: shengwj09@hit.edu.cn).}
\\
\footnotesize{$^a$ School of Mathematics, Harbin Institute of Technology}, \\
\footnotesize{Harbin, Heilongjiang, 150001, People's Republic of China}\\
\footnotesize{$^b$ Aix Marseille Univ, CNRS, I2M, Marseille, France}
}
\maketitle

\begin{abstract}

This paper is concerned with the interaction between a planar traveling front and a compact obstacle for monotone bistable reaction-diffusion systems in exterior domains. By constructing appropriate sub- and supersolutions, we first establish the existence, uniqueness and monotonicity of the entire solution emanating from a planar traveling front. In particular, we verify that regardless of the shape of the obstacle, the entire solution locally converges to a stationary solution as time tends to infinity. Under the complete propagation assumption, we further show that the entire solution recovers to the same planar traveling front as time tends to infinity after passing the obstacle, and it constitutes a transition front. In addition, we provide some geometric conditions on the obstacle to ensure that the complete propagation assumption is nonempty. Finally, we apply our theoretical results to the Lotka-Volterra competition-diffusion system.


\end{abstract}

\section{Introduction}
In this paper, we study the propagation phenomenon of the following reaction-diffusion system
\begin{align}\label{1.1}
\begin{cases}
{ u}_t=   D\Delta u+   F( u),& t\in\mathbb R,\ x\in\Omega=\mathbb R^N\backslash K,\\
{\nu}\cdot\nabla u=\bm 0, &t\in\mathbb R,\ x\in\p\Omega=\p K,
\end{cases}
\end{align}
where $x=(x_1,\dots,x_N)$, $ u=(u_1,\cdots,u_m)$,
$F(u)=(F_1(u),\cdots,F_m(u))$, $\bm0=(0,\cdots,0)\in \R^m$, $N,m\geq2$, ${\nu}=\nu(x)=(\nu_1(x),\cdots,\nu_N(x))$ is the outward unit normal for the point $x$ on the boundary $\partial\Omega$, and the obstacle $K$ is a compact set in $\mathbb R^N$ with smooth boundary.

For the sake of convenience, we first introduce some notations. Let $\bm1= (1,\cdots,1)\in\R^m$. For any vectors $A=(a_1,\cdots,a_m)$ and $B=(b_1,\cdots,b_m)$ in $\R^m$, the symbol $A\ll B$ means $a_i<b_i$ for each $i=1,\cdots,m$, and  $A\leq B$ means $a_i\leq b_i$ for each $i=1,\cdots,m$.
The interval $[A,B]$ denotes the set of $V\in\mathbb R^m$ such that $A\leq V\leq B$,
and $(A,B)$ denotes the set of $V\in\mathbb R^m$ such that $A\ll V\ll B$.

Throughout the paper, we assume that the following conditions hold.
\begin{itemize}
  \item [(A1)] $ D$ is a diagonal matrix of order $m$ with entries $D_i>0$ on the main diagonal for $i=1,2,\cdots,m$.

\item[(A2)] $F$ has two stable equilibria $\bm 0$ and $\bm 1$, that is, $F(\bm0)= F(\bm1)=\bm0$, and all eigenvalues of Jacobian matrices $ F'(\bm 0)$ and $ F'(\bm 1)$ lie in the open left-half complex plane.

\item[(A3)] There exist two positive vectors $R_0=(r_{01},\cdots,r_{0m})$, $R_1=(r_{11},r_{12},\cdots,r_{1m})$ and two positive numbers $\lambda_0$, $\lambda_1$ such that $  F'(\bm0)R_0\leq-\lambda_0 R_0$ and $ F'(\bm1) R_1\leq-\lambda_1 R_1$.

\item[(A4)] The nonlinearity $F(u)$ is defined on an open domain $\mathcal A\supset[\bm0,\bm1]$. Moreover, $F\in C^1([\bm0,\bm1],\R^m)$
satisfies $\frac{\partial{F_i}}{\partial{u_j}}({u})\geq0$ for all $u\in[\bm0,\bm1]$ and $1\leq i\neq j\leq m$.
\end{itemize}

The condition {{(A1)}} implies that all components of $u$ diffuse in $\Omega$ with a positive rate and
the system \eqref{1.1} is not degenerate.
The condition {(A2)} indicates that \eqref{1.1} is a bistable system.
The condition {(A3)} plays a crucial role in the construction of sub- and supersolutions. It implies that there are two Perron-Frobenius directions $R_0$ and $R_1$ at stable equilibria $\bm0$ and $\bm1$.
 In fact, by the Perron-Frobenius theorem, ${(A3)}$ holds if $F'(\bm0)$ and $F'(\bm1)$ are further irreducible.
The condition {(A4)} ensures that the system \eqref{1.1} is  cooperative
and allow us to apply the comparison principle within the order interval $[\bm0,\bm1]$.
 These conditions were utilized to study the stability of planar traveling fronts \cite{OT, Tsai},
the existence and stability of V-shaped traveling fronts \cite{wangzhicheng2012},
as well as the existence of nonstandard transition fronts \cite{SW18} of bistable reaction-diffusion systems.

In general, conditions {(A1)}-{(A4)} are not sufficient to guarantee the existence of planar traveling front connecting $\bm 0$ and $\bm1$ for system \eqref{1.1} in $\R^N$. Therefore, we assume in this paper that when $\Omega=\mathbb R^N$ (with no boundary conditions),
 system \eqref{1.1} admits a unique (up to shifts) planar traveling front of the type
$$
 u(t,x)=\Phi(x\cdot{e}+c t)=\left(\Phi_1(x\cdot{e}+c t),\cdots,\Phi_m(x\cdot{e}+c t)\right),
$$
where $ e\in\mathbb S^{N-1}$ is the propagation direction ($\mathbb S^{N-1}$ denotes the unit sphere of $\mathbb R^N$), $c\in\mathbb R$ is the wave speed and $\Phi:\mathbb R\to[\bm0,\bm1]$ is the propagation profile satisfying
\begin{equation*}
\begin{cases}
 D_i\Phi_i{''}(\xi)-c \Phi_i{'}(\xi)+F_i(\P(\xi))=0,&\xi\in\mathbb R,\\
\P_i{'}(\xi)>0,&\xi\in\mathbb R,\\
\P_i(-\infty)=0,\ \P_i(+\i)=1
\end{cases}
\end{equation*}
 for each $i=1,\cdots,m$.
 By virtue of \cite[Chapter 3]{VVV}, there exist two  positive constants $a$ and $b$ such that for each $i=1,\cdots,m$,
\begin{align}\label{1.a}
\bc
\Phi_{i}(\xi),\ |\Phi'(\xi)|,\ |\Phi''(\xi)|\leq ae^{-b\xi}, &\xi\geq0,\\
1-\Phi_{i}(\xi),\ |\Phi'(\xi)|, \ |\Phi''(\xi)|\leq  ae^{b\xi}, & \xi<0.
\ec
\end{align}
We further assume that there exist constants $\bar K_1>0$ and $\mathcal C>0$ such that, for each $i=1,\cdots,m$,
\begin{align}\label{impo}
-\bar K_1\Phi'_i(\xi)\leq\Phi_i''(\xi)\leq\bar K_1\Phi'_i(\xi)\ \ \text{ for }\xi\in\mathbb R
\end{align}
and
\begin{align}\label{impo1}
\Phi''_i(\xi)\leq0 \ \text{ for }\xi\in(\mathcal C,+\i].
\end{align}
In fact, if $F$ is of class $C^{1+\theta}$ on $\mathcal A$ for some $\theta\in(0,1)$, and the off-diagonal elements of Jacobian matrices $F'(\bm0)$ and $F'(\bm1)$ are positive,
then \eqref{impo} and \eqref{impo1} hold by \cite[Theorem 5.5]{SW18}.
It is well known that the profile of planar traveling front is invariant in the moving
frame with speed $c$, and its level sets are parallel hyperplanes orthogonal to the propagation direction $ e$. For more results on the existence, stability and other qualitative properties of planar traveling fronts, we refer to \cite{aronson,conley,fife,Gardner,GJS,Kan-on,Kan-on2,LX,OT,Tsai,VVV} and references therein.
In the remainder of this paper, we always assume that  $e=(1,0,\cdots,0)$ and
 $c > 0$ (otherwise, we can make $c>0$ by replacing the roles of $\bm0$ and $\bm1$).

In addition to planar traveling fronts, many types of nonplanar traveling fronts are also known to exist in  $\R^N$,
such as V-shaped fronts,
conical shaped fronts,  pyramidal fronts and so on.
The level sets of these nonplanar traveling fronts is no longer hyperplanes.
For more results on  nonplanar traveling fronts,
see \cite{H2,H3,NT,NT1,T1,T2,T3,wangzhicheng2012,wlr} and references therein. 
Both planar and nonplanar traveling fronts share some common properties.
For instance, they converge to one of steady states
 uniformly in time far away from their level sets,
 their profiles keep invariant in a moving frame, as well as their level sets
  move with a constant speed.
  Based on these fundamental observations, a fully general notion of traveling front, known as  transition front, was proposed by
   Berestycki and Hamel \cite{BH1,BH2} (see also \cite{shen} in the one-dimensional setting).
    In particular, they showed that this new definition covers classical cases
    so that we can consider many new situations with a unified framework.
For scalar reaction-diffusion equations (that is, $m = 1$) of bistable type in $\mathbb R^N$,
Hamel\cite{H1} constructed a nonstandard transition front that behaves as three moving planar
 traveling fronts as time goes to $-\i$ and as a V-shaped traveling front as time goes to $+\i$.
 Moreover, he showed that any transition front connecting $0$ and $1$ has a global mean speed
  coinciding with the absolute value of the unique planar wave speed. Sheng and Guo \cite{sg} and Sheng and
   Wang \cite{SW18} generalized the results of \cite{H1} to the time-period bistable
    reaction-diffusion equations and  the time-period bistable reaction-diffusion
    systems, respectively.
 For more results on the existence and stability of transition fronts, we refer to \cite{ads,bgw,dhz,g,H4,hr,hr1,NR,NR1,S2,S3,S4,sww}
 and references therein.

Recently, significant attention has been devoted to studying the propagation phenomenon when an obstacle $K$ appears in the whole space,
particularly in exploring how the geometric properties of the obstacle influence propagation.
In \cite{BHM2009}, Berestycki, Hamel and Matano thoroughly studied the interaction between 
the obstacle $K$ and
a planar traveling front for the bistable reaction-diffusion equation
\begin{align}\label{scalar}
\begin{cases}
u_t=\Delta u+f(u),& t\in\R, \ x\in\Omega,\\
\nu\cdot\nabla u=0,& t\in\R, \ x\in\p\Omega,
\end{cases}
\end{align}
where the nonlinearity $f$ satisfies
\begin{align*}
&\text{$\exists\ \vartheta\in(0,1)$, $f(0)=f(\vartheta)=f(1)=0$, $ f'(0)<0$, $ f'(1)<0$,}\\
&~~~~\text{$f<0$  on $(0, \vartheta)$, $f>0$ on $(\vartheta, 1)$ and $ \int_{0}^{1} f(s) d s>0$.}
\end{align*}
More precisely,
they observed that whatever the shape of $K$ may be, there is  an entire solution $u(t,x)$ emanating from the planar traveling front,
and $u(t,x)\to u_{\i}(x)$ as $t\to+\infty$ locally uniformly in $x\in\overline\Omega$,
where $u_{\i}:\overline\Omega\to(0,1]$ is a $C^2(\overline\Omega)$ solution of the elliptic problem
\begin{equation*}
 \Delta u_{\i}+f(u_{\i})=0\mbox{ in }\Omega, \ \nu\cdot\nabla u_{\i}=0\mbox{ on }\partial\Omega
 \ \text{ and }\ u_{\i}(x)\to1\text{ as }|x|\to+\infty,
\end{equation*}
and $|\cdot|$ denotes the Euclidean norm in $\mathbb R^N$.
In particular, the entire solution $u(t,x)$ is verified to be  a transition front connecting $0$ and $u_{\i}(x)$.
When the obstacle $K$ is {\it star-shaped}
\footnote{We say $K$ is star-shaped if either $K=\emptyset$ or there is $x\in \rm{Int}(K)$ such that, the point $x+t(y-x)\in K$ and $\nu_{K}(y)(y-x)\geq0$ for all $y\in\partial K$ and $t\in[0,1]$, where $\nu_{K}(y)$ denotes the outward unit normal to $K$ at $y$.}
or {\it directionally convex with respect to a hyperplane}
\footnote{We say $ K$ is directionally convex with respect to a hyperplane $P$ if there exists a hyperplane $P=\{x\in\R^N: x\cdot\tilde e=l\}$ with some $\tilde e\in \mathbb S^{N-1}$ and $l\in\mathbb R$, such that for every line $\Sigma$ parallel to $\tilde e$, the set $K\cap\Sigma$ is either a single line segment or empty, and $K\cap P$ is equal to the orthogonal projection of $K$ onto $P$.},
it was proved that $u_\i\equiv1$ and the entire solution $u(t,x)$
will recover to the same planar traveling front after passing the obstacle.
In these cases, the entire solution $u(t,x)$ propagates completely in the sense of
\begin{align*}
u(t,x)\to1\ \text{ locally uniformly in }x\in\bar\Omega \text{ as }t\to+\i.
\end{align*}
They further showed that $u_\i<1$ when $K$ is an almost annular region with a small channel. In this case,
the  entire solution $u(t,x)$ is partly blocked in the sense of
\begin{equation*}
u(t,x)<1\text{ for some points }x\in\overline\Omega \text{ and any }t\in\R.
\end{equation*}
Besides, Guo, Hamel and Sheng \cite{ghs} proved that all transition fronts of \eqref{scalar} connecting $0$ and $u_\i(x)$
propagate with the same global mean speed equal to the unique
planar wave speed.
In exterior domains, transition fronts connecting $0$ and $1$ with other shapes are also known to exist.
For example, Guo and Monobe \cite{Guomonobe} proved that problem \eqref{scalar} admits an entire solution
originating from any homogeneous transition front (which means that the transition front defined in $\R\times\R^N$) in exterior domains.
In addition, by providing complete propagation condition, they indicated that the entire solution originating from a V-shaped traveling front is a transition front connecting $0$ and $1$,
and it converges to the same V-shaped traveling front as time tends to infinity.
More results on the propagation phenomenon of transition fronts in exterior domains can be found in \cite{LiLinlin,jwz,qls,ys} and references therein.

In the present paper, we aim to study the interaction between an obstacle and a planar traveling front in exterior domains for the monotone bistable reaction-diffusion system \eqref{1.1}.
Utilizing sub- and supersolution methods and the comparison principle,
we first establish the existence, uniqueness and monotonicity of the entire solution
emanating from a planar traveling front, and show that the entire solution always converges
to a stationary solution as time tends to $+\i$.
Under the complete propagation assumption, we then
verify that the entire solution can recover to the same planar traveling front after passing the obstacle.
Furthermore, by using the sliding method, we prove that the entire solution initially obtained  propagates completely
when the obstacle is star-shaped or directionally convex with respect to
a hyperplane.

Unlike the scalar case, the sign of nonlinearity $F$  in the neighborhoods of the stable equilibria $\bm0$ and $\bm1$ is unknown,
which brings some difficulties in the construction of  sub- and supersolutions.
 Fortunately, the conditions {(A3)} and {(A4)} imply that
  $F$ has similar behavior to a scalar bistable nonlinearity near $\bm0$ and $\bm1$ along two Perron-Frobenius directions $R_0$ and $R_1$.
  Consequently, we are able to construct sub- and supersolutions by appealing to these Perron-Frobenius directions and the planar traveling fronts.

Before stating our main results, we introduce the definition of transition fronts and their global mean speed for the system \eqref{1.1}
in exterior domain $\Omega$.
For any two subsets $A$ and $B$ in $\Omega$, we set
\begin{equation*}
d_\Omega(A,B)=\inf\{d_\Omega(x,y):(x,y)\in A\times B\}\ \text{ and }\ d_\Omega(x,A)=d_\Omega(\{x\},A),
\end{equation*}
where $d_{\Omega}$ is the geodesic distance in ${\Omega}$. Consider two families $(\Omega^-_t)_{t\in\mathbb R}$ and $(\Omega^+_t)_{t\in\mathbb R}$ of open nonempty subsets of $\Omega$ such that
\begin{equation}\label{1.c}
\forall\ t\in\mathbb R, \
\begin{cases}
\Omega_t^-\cap\Omega_t^+=\emptyset,\\
\partial\Omega_t^-\cap\Omega=\partial\Omega_t^+\cap\Omega=:\Gamma_t,\\
\Omega_t^-\cup\Gamma_t\cup\Omega_t^+=\Omega,\\
\sup\{d_\Omega(x,\Gamma_t):x\in\Omega_t^+\}=\sup\{d_\Omega(x,\Gamma_t):x\in\Omega_t^-\}=+\infty
\end{cases}
\end{equation}
and
\begin{equation}\label{1.d}
\bc
\inf\{\sup\{d_{\Omega}(y,\Gamma_t): y\in \Omega_t^+, d_{\Omega}(y,x)\leq r\big\}: t\in \mathbb{R},\ x\in \Gamma_t\} \rightarrow  +\infty,\\
\inf\{\sup\{d_{\Omega}(y,\Gamma_t): y\in \Omega_t^-, d_{\Omega}(y,x)\leq r\big\}: t\in \mathbb{R},\ x\in \Gamma_t\} \rightarrow+\infty
\ec
\ \ \text{as $r\to+\infty$.}
\end{equation}
Notice that \eqref{1.c} in particular implies that  $\Gamma_t$ divides $\Omega$ into two disjoint unbounded open sets $\Omega_t^-$ and $\Omega_t^+$, which also ensure that the interface $\Gamma_t$ is nonempty for each $t\in\mathbb R$.
As far as (\ref{1.d}) concerned, for each $x\in\Gamma_t$ and $t\in\mathbb R$,
the point $x\in\Gamma_t$ is not too far away from the centers of two large balls included in $\Omega_t^+$ and $\Omega_t^-$.
Moreover, the sets $\Gamma_t$ are assumed to be made of a finite number of graphs. Namely, there is an integer $n_0\geq1$ such that, for each $t\in\mathbb R$, there are $n_0$ open subsets $\omega_{i,t}\subset\mathbb R^{N-1}$ ($1\leq i\leq n_0$), $n_0$ continuous maps $\psi_{i,t}:\omega_{i,t}\to\mathbb R$ and $n_0$ rotations $R_{i,t}$ of $\mathbb R^N$ such that
\begin{equation}\label{1.e}
\Gamma_t\subset\bigcup\limits_{1\leq i\leq n_0}R_{i,t}\left(\{x\in\mathbb R^N: x'=(x_1,\cdots,x_{N-1})\in\omega_{i,t}, \ x_N=\psi_{i,t}(x')\}\right).
\end{equation}

\begin{defi}[\cite{BH1,BH2}]\label{da}
For system \eqref{1.1}, a transition front connecting $\bm0$ and $\bm1$ is a classical
solution $u:\mathbb{R}\times\overline{\Omega} \rightarrow\R^m$ for which there exist
some sets $(\Omega_t^{\pm})_{t\in \mathbb{R}}$ and $(\Gamma_t)_{t\in \mathbb{R}}$
satisfying \eqref{1.c}, \eqref{1.d}, \eqref{1.e}, and for any $\varepsilon>0$, there exists
a constant $M_{\varepsilon}>0$ such that
\begin{equation*}
\begin{cases}
\forall\,t\in \mathbb{R},\ \ \forall\,x\in \overline{\Omega_t^+}, \ \ \left(d_{\Omega}(x,\Gamma_t)\geq M_{\varepsilon}\right)\Rightarrow
\left( u(t,x)\geq (1-\varepsilon)\cdot\bm1\right) ,\vspace{3pt}\\
\forall\,t\in \mathbb{R},\ \ \forall\,x\in \overline{\Omega_t^-}, \ \ \left(d_{\Omega}(x,\Gamma_t)\geq M_{\varepsilon}\right)\Rightarrow
\left( u(t,x)\leq \varepsilon\cdot\bm1\right) .\end{cases}
\end{equation*}
Moreover, $ u$ is said to have a global mean speed $\gamma(\geq0)$ if
\begin{equation*}
\frac{d(\Gamma_t,\Gamma_s)}{|t-s|}\to\gamma~~\text{as}~~|t-s|\to+\infty.
\end{equation*}
\end{defi}

We now state our main results. The first result addresses that  system \eqref{1.1} admits a unique time-increasing entire solution
originating from a planar traveling front.
In particular, the entire solution locally converges to a stationary solution of  \eqref{1.1} as time goes to $+\infty$.
Note that no additional requirements are placed on the geometry of obstacle $K$ in the following theorem.
\begin{theo}{\label{t1}}
There exists a unique entire solution $u(t,x)$ of \eqref{1.1} such that $\bm0\ll u(t,x)\ll\bm1$, $ u_t(t, x)\gg\bm0
$ for  $t\in\mathbb R$ and $x\in\bar\Omega$, and
\begin{equation}\label{1.5}
 u(t,x)- \P(x_1+c t) \to \bm 0\ \text{ uniformly in } x\in\bar\Omega\text{ as }t\to-\infty.
\end{equation}
Moreover, there exists a classical solution $ u_\i(x)$ of elliptic system
\begin{align}\label{ui}
\bc
D\Delta u_\i(x)+ F(u_\i(x))=\bm0,&x\in\Omega,\\
 \nu\cdot \nabla u_\i(x)=\bm0,&x\in\p\Omega,\\
\bm0\ll u_\i\leq\bm1,&x\in\bar\Omega,\\
u_\i(x)\to1\text{ as }|x|\to+\i
\ec
\end{align}
such that $ u(t,x)\to u_\i(x)$ locally uniformly in $x\in\bar\Omega$ as $t\to+\i$.
\end{theo}

On the other hand, by providing complete propagation condition, we prove that the entire solution obtained in Theorem \ref{t1} will recover to the same planar traveling front as time goes to $+\i$.
\begin{theo}\label{t2}
Let $u$ be the entire solution obtained in Theorem \ref{t1}. If $ u$ propagates completely in the sense of
\begin{align}\label{comp}
u(t,x)\to\bm1\ \text{ locally uniformly in }x\in\bar\Omega\text{ as } t\to+\i,
\end{align}
then
\begin{equation}\label{l+}
u(t,x)-\P(x_1+c t) \to \bm 0\ \text{ uniformly in } x\in\bar\Omega \text{ as } t\to+\infty .
\end{equation}
Furthermore, there holds
\begin{align}\label{UP}
u(t,x)-\P(x_1+c t) \to \bm 0\ \text{ uniformly in }t\in\R\text{ as } |x|\to+\infty.
\end{align}
\end{theo}

\begin{rem}
 {\rm By virtue of \eqref{1.5}, \eqref{l+} and \eqref{UP}, one concludes that
  $u(t,x)\to\bm0$ as $x_1+c t\to-\i$ and
  $u(t,x)\to\bm1$ as $x_1+c t\to+\i$.
  The entire solution $u(t,x)$ is thus a transition front connecting $\bm0$ and $\bm1$ in the sense of Definition \ref{da},
  with global mean speed $c$, and with
$\Gamma_t=\{x\in\bar\Omega:x_1+ct=0\}$ and
$\Omega_t^\pm=\{x\in\bar\Omega:\pm(x_1+ct)>0\}$ for each $t\in\R$.
In particular, $u$ is an {\it almost-planar front}
\footnote{A transition front $u$, as defined in Definition \ref{da}, is called an almost-planar front if,
for every $t\in\R$, the set $\Gamma_t$ can be chosen  as the hyperplane
$\Gamma_t=\{x\in\Omega:x\cdot e_t=\xi_t\}$
for some $e_t\in\mathbb S^{N-1}$ and some real number $\xi_t$ (see \cite{BH2,H1}).}.}
 \end{rem}

According to \cite{BHM2009}, propagation may be blocked for scalar bistable reaction-diffusion equations in exterior domains.
In the following theorem, we prove two Liouville-type results by sliding methods, and then we
obtain some geometrical conditions of the obstacle to guarantee that
the complete propagation assumption  made in Theorem \ref{t2} is nonempty.
\begin{theo}\label{lcomp}
If the obstacle $K$ is star-shaped or directionally convex with respect to a hyperplane,
then the solution $u_\i$ of \eqref{ui} satisfies $u_\i(x)\equiv\bm1$
for $x\in\bar\Omega $.
Furthermore, the entire solution $u$ obtained
in Theorem \ref{t1} propagates completely in the sense of \eqref{comp}.
\end{theo}

The paper is organized as follows. In section \ref{s2}, we present  auxiliary notations and lemmas.
Section \ref{s3} is devoted to the existence, uniqueness, monotonicity and the large time behavior of
the entire solutions emanating from a planar traveling front,  in other words, we prove Theorem \ref{t1}.
In Section \ref{s4}, we study the large time behavior of the entire solution given in Theorem \ref{t1} under complete
propagation assumption, namely,
we prove Theorem \ref{t2}.
In Section \ref{s5}, we focus on two Liouville-type results by providing star-shaped or directionally convex obstacle,
that is, we prove Theorem \ref{lcomp}.
Finally, we apply our results to a Lotka-Volterra competition-diffusion system in Section \ref{s6}.

\section{Preliminaries}\label{s2}
In this section, we introduce some preliminaries.
Let us first state the definitions of sub- and supersolutions.

\begin{defi}\label{dss}
If a vector-valued function $u=(u_1,\cdots,u_m)$ satisfies $u_i\in C^{0,1}(\R\times\bar\Omega)\cap C^{1,2}(\R\times\Omega)$ for each $i=1,\cdots,m$
and such that
\begin{align*}
\bc
u_t-D\Delta u-F(u)\leq\bm0\ ({resp.} \geq\bm0)&\ \text{ in }\ \R\times\Omega,\\
u_\nu\leq\bm0\ (resp.\geq\bm0)&\ \text{ on }\ \R\times\p\Omega,
\ec
\end{align*}
then $u$ is called a subsolution $($resp. supersolution$)$ of \eqref{1.1} in $\R\times\bar\Omega$.
Furthermore, if both $ u$ and $v$ are subsolutions $($resp. supersolutions$)$ of \eqref{1.1}, then
$\max(u,v)$ $(resp. \min(u,v))$
 are also called a subsolution $($resp. supersolution$)$ of \eqref{1.1},
 where max $($resp. min$)$ is to be understood componentwise.
\end{defi}

The following comparison principle is directly derived from \cite{FT,PM}.
\begin{lem}\label{com}
If the functions $\underline u,\bar u\in[\bm0,\bm1]$ are sub- and supersolutions of \eqref{1.1} respectively,
and it holds  $ \underline u(0,\cdot)\leq\bar u(0,\cdot)$ in $\Omega$, then
$\underline u\leq\bar u$ in $[0,+\i)\times\bar\Omega$.
\end{lem}

Next, we construct two vector-valued functions $P$ and $Q$ composed of two Perron-Frobenius directions $ R_0$ and $ R_1$, such that the sign of each component of $F$ is clear in the neighborhoods of the equilibria $\bm0$ and $\bm{1} $ along the directions $ P $ and $Q $.
Choose two positive constants
$\eta_0$ and $\eta_1$ such that $\eta_0 R_0\ll R_1$ and $\eta_1R_1\ll R_0$.
Let $(p_1^0,\cdots,p_m^0)= P_0:=\eta_0R_0$ and $(p_1^1,\cdots,p_m^1)= P_1:=R_1$,
$(q_1^0,\cdots,q_m^0)=Q_0:=R_0$ and $(q_1^1,\cdots,q_m^1)= Q_1:=\eta_1R_1$.
Making use of conditions {(A2)} and {(A3)}, there are irreducible constant matrices
$A_0=(\mu_{ij}^0)_{m\times m}$ and $ A_1=(\mu_{ij}^1)_{m\times m}$ such that
$\frac{\partial F_i}{\partial u_j}(\bm0)<\mu_{ij}^0$ and
$\frac{\partial F_i}{\partial u_j}(\bm1)<\mu_{ij}^1$
for each $i,j=1,\cdots,m$, $A_0 P_0\ll-\frac{1}{2}\lambda_0P_0$,
$ A_0 Q_0\ll-\frac{1}{2}\lambda_0 Q_0$, $ A_1 P_1\ll-\frac{1}{2}\lambda_1 P_1$,
$ A_1Q_1\ll-\frac{1}{2}\lambda_1Q_1$,
and the principle eigenvalues of $ A_0$ and $ A_1$ are negative.

Choose a $C^\i$ function $\chi:\R\to[0,1]$ satisfying
\begin{equation*}
 \chi\equiv0 \text{ in } (-\i,0],\ \chi\equiv1 \text{ in } [1,+\i),  0<\chi<1  \text{ and }\chi'>0\text{ in } (0,1).
\end{equation*}
For all $s\in\R$, define vector-valued functions $P(s)=(p_1(s),\cdots,p_m(s)) $ and $Q(s)=(q_1(s),\cdots,q_m(s))$ by
\begin{equation}{\label{2C.a}}
p_i(s)=\chi(s)p_i^1+(1-\chi(s))p_i^0\ \text{ and }\ q_i(s)=\chi(s)q_i^1+(1-\chi(s))q_i^0
\end{equation}
 for each $i=1,\cdots,m$.
It is easy to see that
\begin{align}{\label{2C.2}}
&p_i(s)\in[p_i^0,p_i^1],\ q_i(s) \in[q_i^1,q_i^0],\ p_i'(s)\geq0,\ q_i'(s)\leq0
\ \text{ for }s\in \mathbb R,\notag\\
&\min_{ i=1,\cdots, m}\left(\inf\limits_{s\in\mathbb R}p_i(s)\right)=\min\limits_{ i=1,\cdots, m} p_i^0=:p_*>0,\
\min\limits_{ i=1,\cdots, m}\left(\inf\limits_{s\in\mathbb R}q_i(s)\right)=\min\limits_{ i=1,\cdots, m} q_i^1=:q_*>0,\notag\\
&\max_{ i=1,\cdots, m}\left(\sup\limits_{s\in\mathbb R}p_i(s)\right)=\max\limits_{ i=1,\cdots, m} p_i^1=:p^*>0,\
\max\limits_{ i=1,\cdots, m}\left(\sup\limits_{s\in\mathbb R}q_i(s)\right)=\max\limits_{ i=1,\cdots, m} q_i^0=:q^*>0,\notag\\
&\max_{ i=1,\cdots, m}\left(\sup_{s\in\R}\e(\left| p_i'(s)\right|+\left| p_i''(s)\right|+\left| q_i'(s)\right|+\left| q_i''(s)\right|\r)\right)\leq M\
\text{ for some positive constant } M.
\end{align}

For  any point $v\in\mathbb R^m$ and any constant $r>0$,
let $B_m(v,r)$ be the open Euclidean ball with center $v$ and radius $r$ in $\R^m$.
By the definition of $\mu_{ij}^0$ and $\mu_{ij}^1$,
there exist constants $0<\varepsilon_0<\min\left({p_*}/{4},{q_*}/{4}\right)$ and
 $\varpi>0$ such that
\begin{align}{\label{2C.3}}
&\frac{\partial F_i}{\partial u_j}( u)\leq\mu_{ij}^0\ \text{ for }\
 u\in B_m(\bm0,{4\varepsilon_0})\cap[\bm0,\bm1]\ \text{ and }\ i,j=1,\cdots,m,\notag\\
&\frac{\partial F_i}{\partial u_j}(u)\leq\mu_{ij}^1\ \text{ for }\
 u\in B_m(\bm1,{4\varepsilon_0})\cap[\bm0,\bm1]\ \text{ and }\ i,j=1,\cdots,m,\notag\\
&\sum\limits_{j=1}^{m}\mu_{ij}^0 w_j\leq-\varpi w_i\ \text{ for }\  w=(w_1,\cdots,w_m)
\in\left(\mathbb R^m_+\cap B_m( P_0,{2\varepsilon_0})\right)\cup\left(\mathbb R^m_+\cap
B_m(Q_0,{2\varepsilon_0})\right),\notag\\
&\sum\limits_{j=1}^{m}\mu_{ij}^1 w_j\leq-\varpi w_i\ \text{ for }\ w=(w_1,\cdots,w_m)
\in\left(\mathbb R^m_+\cap B_m( P_1,{2\varepsilon_0})\right)\cup\left(\mathbb R^m_+\cap
 B_m(Q_1,{2\varepsilon_0})\right).
\end{align}

Since $F\in C^1([\bm0,\bm1],\R^m)$, there exists a constant  $\Lambda>0$ such that
\begin{equation}{\label{2C.7}}
\Lambda=\max\limits_{i=1,\cdots, m}\e(\sup\left\{\sum\limits_{j=1}^{m}\left|\frac{\partial F_i}{\partial u_j}(u)\right|:u\in[\bm0,\bm1]\right\}\r).
\end{equation}
Denote
\begin{equation}\label{2C.b}
\overline D=\max\limits_{i=1,\cdots, m}D_i>0,\ \underline D=\min\limits_{i=1,\cdots, m}D_i>0
 \text{ and }\eta=\min\left(\frac{b c}{2},\frac{\varpi}{2}\right),
\end{equation}
where $b>0$ is given by \eqref{1.a}.
For any vector $v\in\R^N$ and any real number $r>0$, let $B(v,r)$ be the open Euclidean ball with center $v$ and radius $r$ in $\R^N$.
Without loss of generality, we assume for the remainder of this paper that the origin $\bm\theta=(0,\cdots,0)$ in $\R^N$ is included in the obstacle $K$.
Since  $ K$ is compact, there exists  $L>0$ large enough such that
\begin{align}\label{ck}
B(\bm\theta,L)\subset K.
\end{align}

We now turn to introduce an auxiliary function that plays a crucial role in the construction of sub- and supersolutions.
Let $\widetilde{\zeta}:\overline\Omega\to\mathbb R$ be a $C^2$ function with compact support in $\overline\Omega$ satisfying $\nu
\cdot\nabla\widetilde{\zeta}=1$ on $\partial \Omega$.  Then, the continuous functions $|\nabla\zeta|$ and $\Delta\widetilde{\zeta}$ are bounded and compact supported. In fact, we can construct a truncated function as a such $
\widetilde{\zeta}$ by applying the classical distance function around the boundary $\partial \Omega$ as in \cite{GT}.  Choose a constant $\widehat C>0$ such that
 \begin{equation}\label{zeta}
\zeta(x)=\widetilde\zeta(x)+\widehat C>1\ \text{ for }\ x\in\overline\Omega
 \end{equation}
and
\begin{equation}\label{2l.a}
  \left\|\frac{\Delta\zeta}{\zeta}\right\|_{L^{\infty}(\overline\Omega)}\leq\frac{\eta}{\overline D}\left(\leq\frac{\varpi}{2\bar D}\right).
\end{equation}
Moreover, it is evident that
$\nu\cdot\nabla\zeta=1$ on $\p\Omega$.
In the case when $\Omega=\R^N$, we define $\zeta=1$ in $\R^N$.

Finally, we establish a key lemma analogous to  \cite[Lemma 5.2]{BHM2009}.
Fix a point $x_0\in\R^N$ and a constant $R>0$ such that $B(x_0, R)\subset\Omega$.
For any $\delta>0$, let  $ v_{x_0,R}^\dl=\e(v_{x_0,R}^{\dl,1},\cdots,v_{x_0,R}^{\dl,m}\r)$ denote the solution of the following
 Cauchy problem
\begin{align}\label{vr}
\bc
\e( v_{x_0,R}^\dl\r)_t(t,x)-D\Delta v_{x_0,R}^\dl(t,x)= F\e(v_{x_0,R}^\dl(t,x)\r),&t>0,\ x\in\Omega,\\
\nu\cdot \nabla v_{x_0,R}^\dl(t,x)=\bm0,&t>0,\ x\in\p\Omega
\ec
\end{align}
with initial value condition
\begin{align}\label{vr0}
v_{x_0,R}^\dl(0,x)=
\bc
(1-\delta q_*)\cdot\bm1,&x\in\bar{B(x_0,R)},\\
\bm0,&x\in\bar{\Omega\backslash B(x_0,R)}.
\ec
\end{align}
It is worth noting that in the proof of \cite[Lemma 5.2]{BHM2009}, the nonlinearity was exactly extended. But this is difficult to achieve for the
vector-valued function \( F \). Consequently, we cannot directly apply \cite[Lemma 5.2]{BHM2009} to the initial-boundary value problem \eqref{vr} and \eqref{vr0}. In fact, by constructing a suitable subsolution close to some radially symmetric expanding fronts propagating with speed less than $c$, we derive the following lemma.

\begin{lem}\label{key}
For any small $\delta>0$, there exist four positive constants  $R_1=R_1(\dl)$,  $R_2=R_2(\dl)$, $R_3=R_3(\dl)$ and $\tilde T=\tilde T(\dl)$ satisfying $R_1<R_2<R_3$ such that, if $B(x_0,R_3)\subset\Omega$, then
$$
v_{x_0,R_1}^\dl(\tilde T,\cdot)\geq(1-\delta q_*)\cdot \bm 1\ \text{ in }\bar{B(x_0,R_2)}(\subset\Omega).
$$
\end{lem}
\begin{pr}
The proof is divided into three steps.

{\it Step 1: some notations and parameters}.
Choose any constant $\dl$ satisfying
\begin{align}\label{bvd}
 0<\dl<\min\left( \frac{2\vp_0}{q^*\|\zeta\|_{L^\i(\bar\Omega)}},\frac{\varpi}{2},1\right)
\left(\leq\frac{2\vp_0}{q^*}\right).
\end{align}
Denote
$\mu=\frac{c}{4\bar D}$.
Define a $C^2$ function $h_\mu:[0,\i)\to\R$ satisfying
\begin{align}\label{hv}
\bc
0\leq h_\mu'\leq1&\text{ in }[0,+\i),\\
h_\mu'=0 &\text{ in a neighborhood of 0},\\
h_\mu(r)=r &\text{ in $[H_\mu,+\i)$ for some $H_\mu>0$},\\
\frac{N-1}{r}h_\mu'(r)+h_\mu''(r)\leq\frac{\mu}{2}&\text { in }[0,+\i).
\ec
\end{align}
Notice in particular that
\begin{align}\label{hr}
r\leq h_\mu(r)\leq r+h_\mu(0)\ \text{ for all }r\geq0.
\end{align}

By \eqref{impo1}, $\Phi(-\i)=\bm0$ and $\Phi(+\i)=\bm1$, there exists a constant $C_1>1$ such that
\begin{align}\label{c}
 \P\leq\delta q_*\cdot\bm1 \text{ in }(-\i,-C_1],\
 \P\geq(1-\delta q_*)\cdot\bm1
\text{ and }\P''\leq\bm0\ \text{ in }[C_1,+\i).
\end{align}
Furthermore, there exists $\kappa>0$ such that
\begin{align}\label{kp}
\kappa=\min_{i=1,\cdots,m}\left(\inf_{\xi\in{[-C_1,C_1]}}\P_i'(\xi)\right).
\end{align}
Similarly, there exists $C_2\g C_1$ such that
\begin{align}\label{hc}
\P\leq\frac{1}{2}\delta q_*\cdot\bm1 \text{ in }(-\i,- C_2]
\ \text{ and }\ \P\geq\e(1-\frac{1}{2}\delta q_*\r)\cdot\bm1 \text{ in }[C_2,+\i).
\end{align}

Let $w>0$ large enough such that
\begin{align}\label{w}
w\geq\kappa^{-1}&\e(\e(\Lambda+ 1\r) q^*\|\zeta\|_{L^\i(\bar\Omega)}+\e(c+\bar D\r)M\|\zeta\|_{L^\i(\bar\Omega)}
+2\bar D M\|\nabla\zeta\|_{L^\i(\bar\Omega)}
+\bar Dq^*\|\Delta\zeta\|_{L^\i(\bar\Omega)}\r).
\end{align}
Pick a positive constant $\tilde T$ such that
\begin{align*}
\tilde T\geq\frac{4}{3 c}
(C_1+C_2+h_\mu(0)+w+1).
\end{align*}
Note that $\tilde T=\tilde T(\dl)$ since $C_1$, $C_2$, $\kappa$ and $w$ depend on the choice of $\dl$.
Since
$$
\frac{3}{4}c \tilde T+H_\mu-h_\mu(0)
\geq 1+w+H_\mu+C_1+C_2 > w+H_\mu+C_1+C_2,
$$
 we can
choose $R_1=R_1(\dl)$,  $R_2=R_2(\dl)$ and  $R_3=R_3(\dl)$ such that
\begin{align}\label{3r}
\begin{cases}
w+H_\mu+C_1+ C_2\leq R_1<R_2\l\frac{3}{4}c \tilde T+H_\mu-h_\mu(0),\\
R_3\geq\max(R_1+L,R_2, w+H_\mu+C_1+ C_2+L+c \tilde T).
\end{cases}\end{align}
It is evident that  $0<R_1<R_2<R_3$.
Take any point $x_0\in \R^N$ such that
$B(x_0,R_3)\subset\Omega$. It then follows from \eqref{ck} that
$$
\bar{B(x_0,R_1)}\subset\bar{B(x_0,R_2)}\subset B(x_0,R_3)\subset\Omega\ \text{ and }\ |x_0|\geq R_3.
$$

{\it Step 2: construction of a subsolution}.
For all $t\in[0,\tilde T]$ and $x\in\bar\Omega$, we define the function $\underline { v}=(\underline v^1,\cdots,\underline v^m)$ as
$$
\underline {v}(t,x)=\max\left(\Phi(\xi(t,x))-\delta Q(\xi(t,x))\zeta(x)e^{-\dl t},\bm 0\right),
$$
where
$$
\xi(t,x)=-h_\mu(|x-x_0|)+\frac{3}{4}c t+ w e^{-\dl t}+H_\mu+C_1.
$$
We shall prove that  the function $\underline { v}(t,x)$ is a subsolution of the problem satisfied by $v_{x_0,R_1}^\dl$ in $[0,\tilde T]\times\bar\Omega$.
Let us first verify the initial and boundary conditions. If $x\in \bar{B(x_0,R_1)}$, one  then infers from $\Phi\leq\bm1$, \eqref{2C.2}
and \eqref{zeta} that $\underline { v}(0,x)\leq(1-\delta q_*)\cdot\bm1= v_{x_0,R_1}^\dl(0,x)$.
If $x\in\bar{\Omega\backslash{B(x_0,R_1)}}$, it then follows from \eqref{hr} and \eqref{3r} that
$$
\xi(0,x)\leq-|x-x_0|+ w+H_\mu+C_1\leq-R_1+ w+H_\mu+C_1\leq-C_2.
$$
Hence,
$\underline{ v}(0,x)\leq\max\left(({\delta q_*}/{2}-\delta q_*)\cdot\bm1,\bm0\right)=\bm0= v_{x_0,R_1}^\dl(0,x)$
for $x\in\bar{\Omega\backslash{B(x_0,R_1)}}$, by \eqref{2C.a}, \eqref{zeta}, \eqref{hc}.
 As a result,  $\underline{ v}(0,x)\leq v_{x_0,R_1}^\dl(0,x)$ for $x\in\bar\Omega$.
Moreover, for all $0\leq t\leq\tilde T$ and  $x\in\bar\Omega$ such that $|x|\leq L$, it follows from
  \eqref{hr} and \eqref{3r} that
$$
\xi(t,x)\leq-|x-x_0|+\frac{3}{4}c\tilde T+H_\mu+ w+C_1
\leq-R_3+L+\frac{3}{4}c\tilde T+H_\mu+ w+C_1\leq-C_2,
$$
then $\underline{v}(t,x)\leq\max(({\delta q_*}/{2}-\delta q_*)\cdot\bm1,\bm0)=\bm0$ because of
\eqref{2C.2}, \eqref{zeta} and \eqref{hc}.
According to \eqref{ck},
one has $\nu\cdot\nabla\underline v(t,x)=0$ for all $0\leq t\leq\tilde T$ and $x\in\p\Omega$.

Let $i=1,\cdots,m$ be any fixed index. It suffices to prove that
$$
\mathscr L_i[\underline{v}](t,x)=\underline v^i_t(t,x)-D_i\Delta\underline v^i(t,x)-F_i(\underline { v}(t,x))\leq0
$$
for $0<t\leq\tilde T$ and $x\in\Omega$ such that $\underline v^i(t,x)>0$.
By direct calculations, one has
\begin{align*}
\mathscr L_i[\underline{ v}](t,x)
=&F_i(\Phi(\xi(t,x)))-F_i(\underline{ v}(t,x))
 +\dl^2  q_i(\xi(t,x))\zeta(x)e^{-\dl t}\\
 & -\dl  we^{-\dl  t}\Phi_i'(\xi(t,x))+D_i\Phi_i''(\xi(t,x))\e(1-\e(h_\mu'(|x-x_0|)\r)^2\r)\\
 &-\Phi_i'(\xi(t,x))\left(\frac{c}{4}-D_ih_\mu''(|x-x_0|)-D_i\frac{N-1}{|x-x_0|}h_\mu'(|x-x_0|)\right)\\
 &-\dl\zeta(x)e^{-\dl t} q_i'(\xi(t,x))\left(\frac{3c}{4}-\dl  we^{- \dl t}\right)
 +\dl D_i\Delta q_i(\xi(t,x))\zeta(x)e^{-\dl t}\\
 &
 +2\dl D_i\nabla q_i(\xi(t,x))\nabla\zeta(x)e^{-\dl t}
 +\dl D_i q_i(\xi(t,x))\Delta\zeta(x)e^{-\dl t}.
 \end{align*}

If $\xi(t,x)\leq-C_1$, then by \eqref{2C.a} and \eqref{hv}, one has that
$Q(\xi(t,x))=Q_0$, $h_\mu(|x-x_0|)\geq H_\mu$ and $h_\mu'(|x-x_0|)=1$.
Note that \eqref{bvd} and \eqref{c} lead to
$\bm0\leq\underline{v}(t,x)\leq\Phi(\xi(t,x))\leq\delta q_*\cdot\bm1\leq4\vp_0\cdot\bm1$.
By \eqref{2C.3}, one has
\begin{align*}
 F_i(\Phi(\xi(t,x)))-F_i(\underline{ v}(t,x))
\leq\sum_{j=1}^m\mu_{ij}^0\dl q_j^0\zeta(x)e^{-\dl t}
\leq-\dl \varpi q_i^0\zeta(x)e^{-\dl t}.
\end{align*}
Since $\P_i'>0$, it follows from \eqref{2C.b}, \eqref{2l.a}, \eqref{bvd}  and \eqref{hv} that
\begin{align*}
\mathscr L_i[\underline{ v}](t,x)
\leq-\Phi'(\xi(t,x))\left(\frac{c}{4}-\frac{\mu \bar D}{2}\right)+
\dl q_i^0\zeta(x)e^{-\dl t}\left(-\varpi+\bar D\left\|\frac{\Delta \zeta}{\zeta}\right\|_{L^\i(\bar\Omega)}+\dl\right)
\leq0.
\end{align*}

 If $\xi(t,x)\geq C_1$, then  $ Q(\xi(t,x))=Q_1$ by \eqref{2C.a}. By \eqref{2C.a}, \eqref{bvd} and \eqref{c}, one has
$$
\bm1\geq\Phi(\xi(t,x))\geq\underline{v}(t,x)\geq\bm1-\delta q_*\cdot\bm1-\delta q_*\|\zeta\|_{L^\i(\bar\Omega)}\cdot\bm1\geq1-4\vp_0\cdot\bm1.
$$
From \eqref{2C.3}, there holds
\begin{align*}
F_i(\Phi(\xi(t,x)))-F_i(\underline{ v}(t,x))
\leq\sum_{j=1}^m\mu_{ij}^1\dl q_j^1\zeta(x)e^{-\dl t}
\leq-\dl \varpi q_i^1\zeta(x)e^{-\dl t}.
\end{align*}
On the other hand, it follows from \eqref{hv} and  \eqref{c}  that $D_i\Phi_i''(\xi(t,x))\e(1-\e(h_\mu'(|x-x_0|)\r)^2\r)\leq0$.
Thus, by similar arguments as above, one obtains that $\mathscr L_i[\underline{ v}](t,x)\leq0$.

If $-C_1\leq\xi(t,x)\leq C_1$, then $h_\mu(|x-x_0|)\geq H_\mu$. Hence $h_\mu'(|x-x_0|)=1$ by \eqref{hv}.
Since $\bm0\leq\underline{ v}(t,x)\leq\Phi(\xi(t,x))\leq\bm 1$, it follows from  \eqref{2C.2} and \eqref{2C.7}
that
\begin{align*}
F_i(\Phi(\xi(t,x)))-F_i(\underline{ v}(t,x))
\leq&\dl\Lambda q^*\zeta(x)e^{-\dl t}.
\end{align*}
By \eqref{2C.2}, \eqref{2C.b}, \eqref{bvd}, \eqref{hv}, \eqref{kp} and \eqref{w}, one gets that
\begin{align*}
\mathscr L_i[\underline{ v}](t,x)
\leq&\dl\Lambda q^*\zeta(x)e^{-\dl t}+\dl  q^*\zeta(x)e^{-\dl t}-\dl  w\kappa
 e^{-\dl t}-\frac{3}{4}\dl c\zeta(x)e^{-\dl t}q_i'(\xi(t,x))\\
 &+\dl\zeta(x)e^{-\dl t}D_iq_i''(\xi(t,x))\e(h_\mu'(|x-x_0|)\r)^2
 +\dl \bar Dq^*\|\Delta\zeta\|_{L^\i(\bar\Omega)}e^{-\dl t}\\
 &-\dl\zeta(x)e^{-\dl t}D_i
 q_i'(\xi(t,x))\left(h_\mu''(|x-x_0|)+\frac{N-1}{|x-x_0|}h_\mu'(|x-x_0|)\right)
\\
&+2\dl D_i|q_i'(\xi(t,x))||h_\mu'(\xi(t,x))|\|\nabla\zeta\|_{L^\i(\bar\Omega)}e^{-\dl t}\\
\leq&\dl e^{-\dl t}\left(-  w\kappa+(\Lambda+ 1) q^*\|\zeta\|_{L^\i(\bar\Omega)}+\e(\frac{3c}{4}+\bar D
+\frac{\mu\bar D}{2}\r)M\|\zeta\|_{L^\i(\bar\Omega)}\right.\\
&\left.+2\bar D M\|\nabla\zeta\|_{L^\i(\bar\Omega)}
+\bar Dq^*\|\Delta\zeta\|_{L^\i(\bar\Omega)}\right)\\
\leq&0.
\end{align*}

{\it Step 3: conclusion.}
By the comparison principle, there holds
$v_{x_0,R_1}^\dl(t,x)\geq\underline{ v}(t,x)$  for
$0\leq t\leq \tilde T$ and $x\in\bar\Omega$.
For $x\in \bar{B(x_0,R_2)}(\subset\Omega)$, it follows from \eqref{hr} and \eqref{3r} that
$$
\xi(\tilde T,x)\geq-R_2-h_\mu(0)+\frac{3}{4}c\tilde T+H_\mu+C_1\geq C_1.
$$
Combining with \eqref{c}, one obtains that
$$
v_{x_0,R_1}^\dl(\tilde T,x)\geq\underline{v}(\tilde T,x)\geq\Phi(\xi(\tilde T,x))\geq(1-\dl q_*)\cdot\bm1 \ \ \text{ for }
0\leq t\leq \tilde T\text{ and }x\in \bar{B(x_0,R_2)}.
$$
The proof is complete.
\end{pr}
\section{Entire solution emanating a planar traveling front}\label{s3}
This section is devoted to the existence, uniqueness, monotonicity and large time behavior of the entire solution emanating from the planar traveling front, that is,  we prove Theorem \ref{t1}.
To this end, we first construct a pair of sub- and supersolutions of \eqref{1.1}.
In the sequel, let $a$ and $b$ be the constants given by \eqref{1.a},
let $q_*$, $q^*$, $M$, $\vp_0$, $\varpi$, $\Lambda$, $\bar D$, $\underline D$, $\eta$ and $L$ be
the positive constants defined as in the
beginning of section \ref{s2}, let $Q$ and $\zeta$ be the functions defined as in \eqref{2C.a} and \eqref{zeta},  respectively.

\begin{lem}{\label{l2.2}}
There exist some constants $w>0$,  $\delta>0$ and $T<0$ such that the functions ${u}^{\mp}=({u}_1^{\mp},\cdots,{u}_m^{\mp})$ defined by
\begin{equation*}
  {u}^{-}(t,x)=\max\left(\P(\xi^-(t,x))-{2a}{q_*^{-1}}\delta Q(\xi^-(t,x))\zeta(x) e^{\eta t},\bm{0}\right)
\end{equation*}
and
\begin{equation*}
  {u}^{+}(t,x)=\min\left(\P(\xi^+(t,x))+{2a}{q_*^{-1}}\delta Q(\xi^+(t,x))\zeta(x) e^{\eta t},\bm{1}\right)
\end{equation*}
 are sub- and supersolutions of \eqref{1.1} for $t\leq T$ and $x\in\overline\Omega$ respectively, where
 $\xi^\mp(t,x)=x_1+c t\mp we^{\eta t}$.
\end{lem}

\begin{pr}
We only prove that the function $u^+$  is a supersolution of \eqref{1.1} in $(-\i,T]\times\bar\Omega$, the proof of subsolution is similar.
We first introduce some notations and parameters. Define
\begin{equation}{\label{2l.1}}
\delta=
   \frac{2\varepsilon_0}{q_*}(<1).
\end{equation}
Since $\Phi'\gg\bm0$, there exists a constant $C>1$ such that
\begin{align}\label{cc}
\P\leq\delta q_*\cdot\bm1\ \text{ in }(-\i,-C]\ \text{ and }\ \P\geq(1-\delta q_*)\cdot\bm1\ \text{ in }[C,+\i).
\end{align}
Furthermore, there exists a constant $\kappa>0$ such that
\begin{equation}{\label{2l.b}}
\min\limits_{i=1,\cdots, m}\left(\inf\limits_{-C\leq\xi\leq C}\Phi_i'(\xi)\right)=\kappa.
\end{equation}
Choose $w>0$ sufficiently large such that
\begin{align}{\label{2l.5}}
w\eta\kappa \geq&{2a}{q_*^{-1}}\left((cM+\eta M+\bar D M+\Lambda q^*) \|\zeta\|_{L^\i(\bar\Omega)}
+2\bar D M\|\nabla\zeta\|_{L^\i(\bar\Omega)}
+\bar Dq^*\|\Delta\zeta\|_{L^\i(\bar\Omega)}
\r).
\end{align}
Take $T<0$ small enough such that
\begin{align}\label{2l.2}
e^{\eta T}<\min\left(\frac{1}{w},\frac{q_*^2}{2aq^*\|\zeta\|_{L^{\infty}(\overline\Omega)}},
\frac{\kappa q_*}{4a\delta M\|\zeta\|_{L^{\infty}(\overline\Omega)}}
\right)\ \text{ and }\
 L+\frac{c t}{2}+1\leq\frac{\ln\delta}{b}(<0).
\end{align}

Next, we verify the boundary conditions. Fix any $i=1,\cdots,m$.
For all $t\leq T$ and $x\in\partial\Omega$, it follows from  \eqref{ck} and \eqref{2l.2} that $\xi^+(t,x)\leq L+c t+1<0$, then $ Q(\xi^+(t,x))= Q_0$ by \eqref{2C.a}.
From \eqref{1.a}, \eqref{2C.b}, \eqref{ck} and \eqref{2l.2}, one has
\begin{align*}
|\nabla \Phi_i(\xi^+(t,x))|
=|\Phi_i'(\xi^+(t,x))|
\leq a e^{b(x_1+c t+we^{\eta t})}
\leq a e^{b\left(L+\frac{c T}{2}+1\right)}e^{\frac{b c}{2}t}
\leq a \delta e^{\eta t}
\end{align*}
for all  $ t\leq T$ and $x\in\p\Omega$. Since $q_i^0\geq q_*$ by \eqref{2C.2} and $\nu\cdot\nabla\zeta=1$ on $\p\Omega$,
one concludes that
\begin{align*}
    \nu\cdot \nabla u_i^+(t,x)&=\nu\cdot\nabla \P_i(\xi^+(t,x))
     +{2a}{q_*^{-1}}\delta q_i^0e^{\eta t}
     \geq-a\delta e^{\eta t}+2a\delta e^{\eta t}
     \geq0
   \end{align*}
for all $ t\leq T$ and $x\in\p\Omega$ such that $u_i^+(t,x)<1$.

It suffices to prove that
$\mathscr L_i[ u^+](t,x)=( u_i^+)_t(t,x)-D_i\Delta u_i^+(t,x)-F_i(u^+(t,x))\geq0
$
 for all $(t,x)\in(-\infty, T]\times \Omega$ such that $u_i^+(t,x)<1$.
 By direct calculations, there holds
 \begin{align*}
\mathscr L_i[ u^+](t,x)
=&w\eta e^{\eta t}\Phi_i'(\xi^+(t,x))+{2a}{q_*^{-1}}\delta(c+w\eta e^{\eta t}) q_i'(\xi^+(t,x))\zeta(x) e^{\eta t}\\
&+{2a}{q_*^{-1}}\delta\eta q_i(\xi^+(t,x))\zeta(x) e^{\eta t}
-2D_i{a}{q_*^{-1}}\delta q_i{''}(\xi^+(t,x))\zeta(x)e^{\eta t}\\
&-4D_i{a}{q_*^{-1}}\delta \nabla q_i(\xi^+(t,x))\cdot\nabla\zeta(x)e^{\eta t}
-2D_i{a}{q_*^{-1}}\delta q_i(\xi^+(t,x))\Delta\zeta(x) e^{\eta t}\\
&+F_i(\P(\xi^+(t,x)) ) -F_i(u^+(t,x)).
\end{align*}

If $\xi^+(t,x)\leq-C$,  then
 $ Q(\xi^+(t,x))= Q_0$
and $\P(\xi^+(t,x))\leq\dl q_*\cdot\bm1$ because of  \eqref{2C.a} and \eqref{cc}.
By \eqref{2C.2}, \eqref{2l.1}  and \eqref{2l.2}, one infers that
$$
\bm0\leq \P(\xi^+(t,x))\leq  u^+(t,x)\leq\left(\delta q_*+{2a}\delta {q_*^{-1}}q^*\|\zeta\|_{L^{\infty}(\overline\Omega)}e^{\eta T}\right)\cdot\bm1\leq2\delta q_*\cdot\bm1\leq4\varepsilon_0\cdot\bm1.
$$
 By \eqref{2C.3}, there holds
\begin{align*}
 F_i(\P(\xi^+(t,x))) -F_i( u^+(t,x))
 \geq- {2a}{q_*^{-1}}\dl\zeta(x) e^{\eta t}\sum\limits_{j=1}^m\mu_{ij}^0q_j^0
\geq{2a}{q_*^{-1}}\dl\zeta(x)\varpi q_i^0 e^{\eta t}.
\end{align*}
Hence, one gets from $\P_i'>0$ and \eqref{2l.a} that
\begin{align*}
\mathscr L_i[ u^+](t,x)
\geq&{2a}{q_*^{-1}}\dl\zeta(x)\varpi q_i^0 e^{\eta t} +{2a}{q_*^{-1}}\dl q_i^0\zeta(x)e^{\eta t}\left(\eta- \bar{D}\left\|\frac{\Delta\zeta}{\zeta}\right\|_{L^{\infty}(\overline\Omega)}\right)
\geq0.
\end{align*}

If $\xi^+(t,x)\geq C$, then $ Q(\xi^+(t,x))= Q_1$ by \eqref{2C.a}.
Similarly,
$\mathscr L_i[u^+](t,x)\geq0$.

If $-C\leq\xi^+(t,x)\leq C$,  it then follows from \eqref{2C.7} that
\begin{align*}
 F_i(\P(\xi^+(t,x))) -F_i( u^+(t,x))
 \geq-{2a}{q_*^{-1}} \Lambda \dl q_i(\xi^+(t,x))\zeta(x) e^{\eta t}
 \geq-{2a}{q_*^{-1}} \Lambda \dl q^*\|\zeta\|_{L^\i(\bar\Omega)} e^{\eta t}.
\end{align*}
By \eqref{2C.2}, \eqref{2C.b},  \eqref{2l.1}, \eqref{2l.b}, \eqref{2l.5} and  \eqref{2l.2}, we have
\begin{align*}
\mathscr L_i[ u^+](t,x)
\geq&w\eta\kappa e^{\eta t}-{2a}{q_*^{-1}}e^{\eta t}\left((cM+\eta M+\bar D M+\Lambda q^*) \|\zeta\|_{L^\i(\bar\Omega)}\r.
\\
&\left.+2\bar D M\|\nabla\zeta\|_{L^\i(\bar\Omega)}+\bar D  q^*\|\Delta\zeta\|_{L^\i(\bar\Omega)} \r)\\
\geq&0.
\end{align*}

As a result, the function $u^+$ is a supersolution of \eqref{1.1}  in $(-\i, T]\times\bar\Omega$. The proof is complete.
\end{pr}
\vspace{0.2cm}

\begin{pr}[{Proof of Theorem \ref{t1}}] We divide the proof into several steps.

{\it Step 1: existence of entire solutions.}
Let $u^\pm$ be the sub- and supersolutions defined as in Lemma \ref{l2.2}.
Let $n\in\mathbb N\cap[|T|+1,+\i)$ and let $u_n(t,x)$ be the solution of \eqref{1.1} for all $t\geq-n$ with the initial condition $u_n(-n,x)=u^+(-n,x)$ for all $x\in\overline\Omega$. By the definition of $ u^\pm$, there holds $u^-(-n,x)\leq u^+(-n,x)$ for all $x\in\overline\Omega$. According to the comparison principle, one has
\begin{equation}\label{t1.1}
  u^-(t,x)\leq u_n(t,x)\leq u^+(t,x)\ \ \text{ for all } t\in[-n, T] \text{ and } x\in\overline\Omega.
\end{equation}
Therefore, $ u_n(-n+1,\cdot)\leq u^+(-n+1,\cdot)= u_{n-1}(-n+1,\cdot)$ in $\overline\Omega$. Applying the comparison principle again, one gets that
$ u_{n}(t,x)\leq u_{n-1}(t,x)$ for all $t\in[-n+1,
+\infty]$ and $x\in\overline\Omega$, which implies that $ u_n$ is decreasing in $n$.
By standard parabolic estimate, there exists an entire solution $ u=(u_1,\cdots,u_m)$ of \eqref{1.1}
defined in $\mathbb R\times\overline\Omega$ such that
$ u_n(t,x)\to u(t,x)$ uniformly in $(t,x)\in\mathbb R\times\overline\Omega$ as $n\to+\i$.
From (\ref{t1.1}), there holds $
  u^-(t,x)\leq u(t,x)\leq u^+(t,x)$ for $t\in[-\infty,T]$ and $x\in\overline\Omega$.
 Therefore, one gets that
\begin{equation*}
 u(t,x)-\P(x_1+c t)\to\bm0\ \ \text{ uniformly in } x\in\overline\Omega \text{ as } {t\to-\infty} ,
\end{equation*}
then the formula \eqref{1.5} follows.

{\it Step 2: monotonicity of entire solutions.}
 If $\xi^+(t,x)<0$ or $\xi^+(t,x)>1$, by \eqref{2C.a} and  $\Phi'\gg\bm0$, one gets that
 $u^+_t(t,x)\gg\bm0$ for all $t\leq T$ and $x\in\bar\Omega$.
If $0\leq\xi^+(t,x)\leq1$, thanks to \eqref{2C.2}, for each $i=1,\cdots,m$,
 \begin{align*}
 \lim_{t\to-\i}(u^+_i)_t(t,x)
 =&
 \lim_{t\to-\i}\e((c+w\eta e^{\eta t})\Phi_i'(\xi(t,x))+{2a}{q_*^{-1}}\dl(c+w\eta e^{\eta t}) q_i'(\xi(t,x))\zeta (x)e^{\eta t}\r)\\
 \geq&\lim_{t\to-\i}\e(
 c\min_{0\leq\xi\leq1}\Phi_i'(\xi)-{2a}{q_*^{-1}}\dl(c+w\eta e^{\eta t})M\zeta (x)e^{\eta t}\r)=
 c\min_{0\leq\xi\leq1}\Phi_i'(\xi)>0
 \end{align*}
uniformly in $x\in\bar\Omega$.
As a result,
 $u^+(t,\cdot )$ is increasing in sufficiently negative $t$.
Hence, ${( u_n)}_t(-n,\cdot)\gg\bm0$ for large $n$.
 Consider the parabolic initial-boundary problem
 \begin{align*}
 \bc
 ( u_n)_{tt}(t,x)- D\Delta( u_n)_t(t,x)- F'( u_n)( u_n(t,x))_t(t,x)=\bm0,&t>-n,\ x\in\Omega,\\
 \nu\cdot\nabla( u_n)_t(t,x)=\bm 0&t>-n,\ x\in\partial\Omega,\\
 ( u_n)_t(-n,x)\gg\bm0,&x\in\Omega.
\ec
\end{align*}
 By the maximum principle and the Hopf lemma, one has
 $(u_n)_{ t}(t,x)\gg\bm0$ for $t\geq-n $ and $x\in\overline\Omega$.
 Passing to the limit $n\to+\infty$, then $u_t(t,x)\geq\bm0$ for $t\in\mathbb R$ and $x\in\overline\Omega$.
 Assume by contradiction that there exists a point $(t_0,x_0)\in\R\times\bar\Omega$
 and an index $i_0$
 such that $(u_{i_0})_t(t_0,x_0)=0$. If $(t_0,x_0)\in\R\times\Omega$,
  then the maximum principle leads to $(u_{i_0})_t(t,x)\equiv0$ for $t\leq t_0$ and $x\in\Omega$, it contradicts with \eqref{1.5}. If $(t_0,x_0)\in\R\times\p\Omega$, then the Hopf boundary lemma yields that $\nu\cdot\nabla (u_{i_0})_t(t_0,x_0)\ll0$, a contradiction. As a conclusion,
$u_t(t,x)\gg\bm0$ for $t\in\mathbb R$ and $x\in\overline\Omega$.

In addition, combining \eqref{1.5} with the Hopf boundary lemma and the  maximum principle, one gets immediately that $\bm 0\ll u(t,x)\ll\bm1$
for  $t\in\R$ and $x\in\bar\Omega$.

{\it Step 3: uniqueness of entire solutions.} For any $0<\delta< {2\vp_0}/{q^*}(<1)$,
by \eqref{1.5} and the properties of $\P$, there exists $T_\dl<0$ negative enough and $C_\dl>1$ large enough such that
\begin{align}\label{td1}
\bc
 u(t,x)\leq\delta q_*\cdot\bm 1& \text{ for all } t\leq T_\dl \text{ and } x\in\bar\Omega\text{ such that }x_1+c t\leq -C_\dl,\\
u(t,x)\geq(1-\delta q_*)\cdot\bm 1&\text{ for all } t\leq T_\dl \text{ and }x\in\bar\Omega \text{ such that } x_1+c t\geq C_\dl.
\ec
\end{align}
Even if it means decreasing $T_\dl$, one can assume that
\begin{align}\label{TD}
L+c T_\dl\leq0.
\end{align}
 Denote
\begin{align}\label{od}
\Omega_\dl(t)=\left\{x=(x_1,x')\in\Omega:|x_1+c t|\leq C_\dl\right\},
\end{align}
where $x'=(x_2,\cdots,x_N)$. We claim that there exists $\hat\kappa>0$ such that for each $i=1,\cdots,m$,
 \begin{align}\label{hk}
 (u_i)_t(t,x)\geq\hat\kappa\ \text{ for all }x\in\Omega_\dl(t)\text{ and }t\leq T_\dl.
\end{align}
Assume that there is $i_0\in\{1,\cdots,m\}$
such that \eqref{hk} does not hold.
Then there is a sequence of real numbers $(t_k)_{k\in\mathbb N}\subset(-\i,T_\dl)$ and
a sequence of points $(x_k)_{k\in\mathbb N}=(( x_{1,k},x_k'))_{k\in \mathbb N}\subset\Omega_\dl(t_k)$ such that
\begin{align}\label{uit}
(u_{i_0})_t(t_k,x_k)\to0 \text{ as }k\to+\i.
\end{align}
Up to extraction of a subsequence, we can assume that either $t_k\to t^*$
for some $t^*\in(-\i,T_\dl]$ or $t_k\to-\i$ as $k\to+\i$.
If $t_k\to t^*$ as $k\to+\i$, by \eqref{od}, one gets that $x_{1,k}$ is bounded for large $k$.
Thus, we can assume that $x_{1,k}\to x_1^*$ as $k\to+\i$.
For each
$k\in\mathbb N$, set
$$
u_k(t,x)=u(t+t_k,x+x_k) \ \text{ for all }t\leq T_\dl \text{ and }x\in\R^N\text{ such that }x_1\geq L+C_\dl+c t_k.
$$
For these $x$, it follows from  \eqref{od} that $x_1+x_{1,k}\geq
L$,
hence the function $u_k $ is well defined
since $x+x_k\in \Omega$ by \eqref{ck}.
By standard parabolic estimates,  up to extraction of a subsequence,
there is a classical solution $u^{\i}=(u_1^{\i},\cdots,u_m^{\i})$ of \eqref{1.1}
 such that $
 u_k(t,x)\to u^\i(t,x)$ locally  uniformly in $t\leq T_\dl$ and $x\in\mathbb R^N $
such that $x_1\geq L+C_\dl+c t^*$ as $k\to+\i$.
By \eqref{uit} and $u_t\gg\bm0$, one has $(u_{i_0}^\i)_t(0,\bm0)=0$
and $(u_{i_0}^\i)_t(t,x)\geq0$ for $t\in\R$ and $x\in\R^N$ such that $x_1\geq L+C_\dl+c t^*$.
Applying the maximum principle to $(u_{i_0}^\i)_t$, we obtain
$(u_{i_0}^\i)_t\equiv0$ for all $t\leq 0$ and   $x\in\R^N$ such that $x_1\geq L+C_\dl+c t^*$,
but it is impossible since \eqref{1.5} implies that
 $u_{i_0}^\i(t,x)-\P_{i_0}(x_1+x_1^*+c(t+t^*))\to0$ as $t\to-\i$ uniformly in $x\in\R^N$ such that $x_1\geq L+C_\dl+c t^*$.

If $t_k\to-\i$ as $k\to+\i$, it then follows from \eqref{od} that $x_{1,k}\to+\i$ as $k\to+\i$.
For large $k\in \mathbb N$, set $ u_k(t,x)= u(t+t_k,x+x_k)$ for $t\in\R$ and
 $x\in\R^N$ such that $x_1\geq-x_{1,k}/2$.
 For these $x$, one has $x_1+x_{1,k}\geq x_{1,k}/2\geq L$ for large $k$,
 hence $ u_k$ is well defined  for large $k$.
By standard parabolic estimates,  up to extraction of a subsequence,
there is a $C^{1,2}(\R\times\R^N)$ solution $u^{\i}=(u_1^{\i},\cdots,u_n^{\i})$ of \eqref{1.1}  such that $
u_k(t,x)\to u^\i(t,x)$ locally  uniformly in $t\in\R$ and $x\in\mathbb R^N $
 as $k\to+\i$.
 By similar arguments as above, we can arrive at a contradiction.
 The proof of \eqref{hk} is thereby complete.

 Suppose that there exists another entire solution $  v=(v_1,\cdots,v_m)$ satisfying
 $$
 v(t,x)-\P(x_1+c t)\to\bm0\ \text{ uniformly in }x\in\bar\Omega \text{ as }t\to-\i.
 $$
Let $\vp$ be any fixed constant satisfying
\begin{align}\label{vp}
0<\vp<\min\left(\dl,\frac{\dl q_*}{q^*\|\zeta\|_{L^\i(\bar\Omega)}},
\frac{\hat\kappa}{2c M\|\zeta\|_{L^\infty(\bar\Omega)}}\right)(<1).
\end{align}
Then there exists $t_\vp\in\R$ such that
\begin{align}\label{v}
-\vp q_*\cdot\bm1\leq v(t,x)-u(t,x)\leq \vp q_*\cdot\bm1 \ \text{ for all }t\leq t_\vp \text{ and }x\in\bar\Omega.
\end{align}
Define
\begin{align}\label{rho}
\rho=&
{2}{\eta^{-1} \hat\kappa^{-1}}
\left((c M+\eta q^* + \bar DM +\Lambda q^*)\|\zeta\|_{L^\i(\bar\Omega)}
+\bar Dq^*\|\Delta\zeta\|_ {L^\i(\bar\Omega)} +2 \bar DM\|\nabla\zeta\|_{L^\i(\bar\Omega)}
\right),
\end{align}
where $\hat\kappa$ is defined by \eqref{hk}.
Take any $t_0\in(-\i,T_\dl-\rho\vp]$.
For all $t\in[0,T_\dl-t_0-\rho\vp]$ and $x\in\bar\Omega$, define the functions $ w^\pm=(w^\pm_1,\cdots,w^\pm_m)$ as
$$
  w^+(t,x)=\min\left( u(\varrho^+(t),x)+\vp Q(x_1+c\varrho^+(t))\zeta(x)e^{-\eta t},\bm1\right)
$$
and
$$
w^-(t,x)=\max\left(u(\varrho^-(t),x)-\vp Q(x_1+c\varrho^-(t)))\zeta(x)e^{-\eta t},\bm0\right),
$$
 where $\varrho^\pm(t)=t+t_0\pm\rho\vp(1-e^{-\eta t})$.
We shall check the function $w^+$ is a supersolution of the problem satisfied by
 $v(t_0+\cdot,\cdot)$ in $[0,T_\dl-t_0-\rho\vp]\times\bar\Omega$,
 the fact that $w^-$ is a subsolution can be obtained by some similar arguments.
 We first check the initial and boundary conditions.
 For all $x\in\bar\Omega$, by \eqref{2C.2}, \eqref{zeta} and \eqref{v}, one has
$$
w^+(0,x)=\min\left(u(t_0,x)+\vp Q(x_1+c t_0)\zeta(x),\bm1\right)
\geq \min\left(u(t_0,x)+\vp q_*\cdot\bm1,\bm1\right)\geq v(t_0,x).
$$
By \eqref{ck} and \eqref{TD}, there holds $x_1+c\varrho^+(t)\leq L+c(t_0+\vp\rho+t)\leq L+c T_\dl\leq0$ for all $t\in[0,T_\dl-t_0-\rho\vp]$ and $x\in\p\Omega$. Then $ Q(x_1+\varrho^+(t))= Q_0$ for these $t$ and $x$ from \eqref{2C.a}.
Since $\nu\cdot\nabla\zeta=1$ on $\p\Omega$, one has  $\nu\cdot\nabla w_i^+(t,x)=\vp q_i^1>0$
for each $i=1,\cdots,m$ and  for all  $t\in[0,T_\dl-t_0-\rho\vp]$ and $x\in\p\Omega$ such that $w_i^+(t,x)<1$.

 Let $i=1,\cdots,m$ be any fixed index. It suffices to prove that
  $$
  \mathscr L_i[ w^+](t,x)=( w^+_i)_t(t,x)-D_iw^+_i(t,x)-F_i(w^+(t,x))\geq0
  $$
  for $t\in[0,T_\dl-t_0-\rho\vp]$ and $x\in\bar\Omega$ such that  $w_i^+(t,x)<1$.
   By  direct calculations, one gets that
 \begin{align*}
 \mathscr L_i[ w^+](t,x)
 =&\ \vp\rho\eta e^{-\eta t}(u_i)_t(\varrho^+(t),x)+\vp c(1+\vp\rho\eta e^{-\eta t}) q_i'(x_1+c\varrho^+(t))\zeta(x) e^{-\eta t}\\
 &-\vp\eta q_i(x_1+c\varrho^+(t))\zeta(x) e^{-\eta t}-\vp D_i q_i''(x_1+c\varrho^+(t))\zeta(x) e^{-\eta t}\\
 &-\vp D_i q_i(x_1+c\varrho^+(t))\Delta\zeta(x) e^{-\eta t}-2\vp D_i \nabla q_i(x_1+c\varrho^+(t))\cdot \nabla\zeta(x) e^{-\eta t}\\
 &+F_i(  u(\varrho^+(t),x))-F_i( w^+(t,x)).
 \end{align*}
  If $x_1+c\varrho^+(t)\leq -C_\dl$, then $ Q(x_1+c\varrho^+(t))=Q_0$ by \eqref{2C.a}.
  Note that $\varrho^+(t)\leq t+t_0+\vp\rho\leq T_\dl$.
  From \eqref{2C.2}, \eqref{td1} and \eqref{vp}, there holds
  $
   u(\varrho^+(t),x)\leq w^+(t,x)\leq\dl q_*\cdot\bm 1+\vp q^*\|\zeta\|_{L^\i(\bar\Omega)}\cdot\bm 1\leq2\delta q_*\cdot\bm 1\leq 4\vp_0\cdot\bm 1
  $.
  By \eqref{2C.3}, there holds
  \begin{align*}
 F_i(u(\varrho^+(t),x))-F_i( w^+(t,x))
  \geq-\vp\sum_{j=1}^m\mu_{ij}^0 q_j^0\zeta(x)e^{-\eta t}
  \geq\vp \varpi q_i^0\zeta(x)e^{-\eta t}.
  \end{align*}
 By $(u_i)_t>0$,  \eqref{2C.b} and \eqref{2l.a}, one has
\begin{align*}
\mathscr L_i[w^+](t,x)
\geq&-\vp q_i^0\zeta(x)e^{-\eta t}\left(\eta+\bar D\left\|\frac{\Delta\zeta}{\zeta}\right\|_{L^\i(\bar\Omega)}-\varpi\right)
\geq0.
\end{align*}

  If $x_1+c\varrho^+(t)\geq C_\dl$, one infers from \eqref{2C.a} that
  $Q(x_1+c\varrho^+(t))= Q_1$. Since $\varrho^+(t)\leq  T_\dl$,
By similar arguments as above, one gets that $\mathscr L_i[  w^+](t,x)\geq0$.

If $-C_\dl\leq x_1+c\varrho^+(t)\leq C_\dl$ (that is, $x\in\Omega_\dl(\varrho^+(t))$),
then $(u_i)_t(\varrho^+(t),x)\geq\hat\kappa$ by \eqref{hk}.
Furthermore, it follows from \eqref{2C.2} and \eqref{2C.7} that
\begin{align*}
F_i( u(\varrho^+(t),x))-F_i( w^+(t,x))
  \geq-\vp \Lambda q^*\|\zeta\|_{L^\i(\bar\Omega)}e^{-\eta t}.
\end{align*}
 Thus, one infers from \eqref{2C.2}, \eqref{2C.b}, \eqref{vp} and \eqref{rho} that
 \begin{align*}
 \mathscr L_i[ w^+](t,x)
 \geq&\vp\rho\eta\hat\kappa e^{-\eta t}-\vp^2 c\rho\eta M\|\zeta\|_{L^\i(\bar\Omega)} e^{-2\eta t}
 -\vp c M\|\zeta\|_{L^\i(\bar\Omega)} e^{-\eta t}\\
 &-\vp\eta q^* \|\zeta\|_{L^\i(\bar\Omega)}e^{-\eta t}
 -\vp \bar D M\|\zeta\|_{L^\i(\bar\Omega)} e^{-\eta t}
 -\vp\bar D  q^*\|\Delta\zeta\|_ {L^\i(\bar\Omega)} e^{-\eta t}\\
 &-2\vp \bar D M\|\nabla\zeta\|_{L^\i(\bar\Omega)} e^{-\eta t}-\vp \Lambda q^*\|\zeta\|_{L^\i(\bar\Omega)}\\
 \geq&\vp e^{-\eta t}\left({\rho\eta \hat\kappa}-\vp c\rho\eta M\|\zeta\|_{L^\i(\bar\Omega)} e^{-\eta t}-\frac{\rho\eta \hat\kappa}{2}\r)\\
 \geq&0.
 \end{align*}

 By the comparison principle, there holds
 $$
  w^-(t-t_0,x)\leq  v(t,x)\leq w^+(t-t_0,x)\ \text{ for }t\in[t_0,T_\dl-\rho\vp]\text{ and }x\in\bar\Omega.
 $$
 Passing the limit $t_0\to-\i$, one has
 $$
 u(t-\vp\rho,x)\leq v(t,x)\leq u(t+\vp\rho,x)\ \text{ for }t\in(-\i,T_\dl-\rho\vp]\text{ and }x\in\bar\Omega.
 $$
 According to the comparison principle and the arbitrariness of $\vp$,
 there holds $ u(t,x)\equiv v(t,x)$ for all $t\in\R$ and $x\in\bar\Omega$.
 As a conclusion, we obtain the uniqueness of the entire solution obtained in step 1.

{\it Step 4: Convergence of entire solution.} Since $u_t\gg\bm0$ and $\bm0\ll u\ll \bm1$,
 standard parabolic estimates yield that
\begin{align}\label{u-i}
 u(t,x)\to  u_\i(x)\in[\bm0,\bm1]\ \text{ locally uniformly in $x\in\bar\Omega$ as $t\to+\i$,}
 \end{align}
 where $u_\i$ is a classical solution the corresponding stationary equation of \eqref{1.1}.
 Fix an arbitrary small positive constant $ \delta$.
By $\Phi(+\i)=\bm1$, \eqref{1.5} and \eqref{ck},
 there exist  $T_0<0$  and $\xi_0>L$   such that
\begin{align}\label{H}
u(T_0,\cdot)\geq (1-\delta q_*)\cdot\bm1 \ \ \text{  in }\left\{x\in\mathbb R^N: x_1\geq\xi_0\right\}(\subset\Omega).
 \end{align}
 Let $R_1=R_1(\dl)$,  $R_2=R_2(\dl)$, $R_3=R_3(\dl)$ and $\tilde T=\tilde T(\dl)$ with
 $R_1<R_2<R_3$ be defined as in Lemma \ref{key}.
 Take any point $x\in\mathbb R^N$ satisfying $|x|>L+R_3-R_2$. Note that
 \eqref{ck} implies $x\in\Omega$. It is evident that there exist
 an integer $k\geq1$ and $k$ points $x^1,\cdots,x^k$ in $\R^N$ such that
 \begin{align*}
\bc
B(x^1,R_1)\subset\left\{x\in\mathbb R^N: x_1\geq\xi_0\right\},\\
B(x^i, R_3)\subset\Omega&\text{for }1\leq i\leq k,\\
B(x^{i+1}, R_1)\subset B(x^{i}, R_2)&\text{for }1\leq i\leq k-1,\\
x\in B(x^k,R_2).
\ec
 \end{align*}
By  \eqref{H} and Lemma \ref{key}, one has
$u(T_0+\tilde T,\cdot)\geq (1-\delta q_*)\cdot\bm1 $ in $B(x^1,R_2)$.
Since $B(x^2,R_1)\subset B(x^1,R_2)$, one similarly gets that
$ u(T_0+2\tilde T,\cdot)\geq (1-\delta q_*)\cdot\bm1 $ in $B(x^2,R_2)$.
By induction, $u(T_0+k\tilde T,\cdot)\geq (1-\delta q_*)\cdot\bm1 $ in $B(x^k,R_2)$. Hence
 $ u(T_0+k\tilde T,x)\geq (1-\delta q_*)\cdot\bm1 $
for all $x\in\mathbb R^N$ such that $|x|>L+R_3-R_2$.
Together with \eqref{u-i},
one has
$$
u_\i(x)\geq (1-\delta q_*)\cdot\bm1 \ \
\text{ for all }x\in\R^N\text{ such that }|x|>L+R_3-R_2.
$$
Combining with $u_\i\leq\bm1$ and the arbitrariness of $\delta$, one has $ u_\i(x)\to\bm1$ as $|x|\to+\i$.
In addition, owing to the maximum principle and Hopf boundary lemma, there holds $\bm0\ll u_\i\leq\bm1$.
The proof of Theorem \ref{t1} is now complete.
\end{pr}

\section{large time behavior by providing complete propagation}\label{s4}
In this section, we prove that  the entire solution obtained in Theorem \ref{t1} converges to the same planar traveling front
as time goes to infinity by providing complete propagation.
We first analyze the relationship between  the entire solution and the planar traveling front at large time.
Then, we complete the proof of Theorem \ref{t2}.
In what follows, let $a$ and $b$ be the constants given by \eqref{1.a},
let $q_*$, $q^*$, $M$, $\vp_0$, $\varpi$, $\Lambda$, $\bar D$, $\underline D$, $\eta$ and $L$ be
the positive constants defined as in the
beginning of section \ref{s2}, let  $Q$ and $\zeta$ be the functions defined as in \eqref{2C.a} and \eqref{zeta},  respectively.

\subsection{Relationship between  the entire solution and the planar traveling front}
\begin{lem}\label{w-}
For any small $\dl>0$, there exist some constants $\beta>0$, $\rho>0$ and $T^*>0$
such that for any $T\geq T^*$,
 the functions $ w^\mp=(w^\mp_1,\cdots,w^\mp_m)$  defined by
$$
 w^-(t,x)=\max\left(  \Phi(\xi^-(t,x))-\dl  Q(\xi^-(t,x))\zeta(x)e^{-\beta t} ,\bm 0\right)
$$
and
$$
  w^+(t,x)=\min\left(\Phi(\xi^+(t,x))+\dl   Q(\xi^+(t,x))\zeta(x)e^{-\beta t} ,\bm 1\right)
$$
are  sub- and supersolutions of \eqref{1.1} for $t\geq0$ and $x\in\bar\Omega$ respectively, where
$$
\xi^\mp(t,x)=x_1+c(t+T)\mp\rho\dl (1-e^{-\beta t}).
$$
\end{lem}
\begin{pr}
We only show that the function $w^-$ is a subsolution  of \eqref{1.1} in $[0,+\i)\times\bar\Omega$,
the supersolution $w^+$ can be obtained by similar arguments.
We first introduce some notations and parameters.
Choose
\begin{align}\label{db}
0<\dl_1<2\vp_0,\ 0<\dl<\min\e(1,\frac{\dl_1}{q^*\|\zeta\|_{L^\i(\bar\Omega)}}\r)\e(\leq\frac{2\vp_0}{q_*}\r)
\ \text{ and }\
0<\beta<\min\e(\dl,b c,\frac{\varpi}{2}\r).
\end{align}
By virtue of the properties of $\Phi$, there exists $C>1$   such that
\begin{align}\label{pc}
\Phi\leq \dl_1\cdot\bm1\ \text{ in }(-\i,-C] \  \text{ and } \
\Phi\geq (1-\dl_1)\cdot\bm1 \ \text{ in } [C,+\i).
\end{align}
Furthermore, there exists $\kappa>0$ such that
 \begin{align}\label{kapp}
\kappa=\min_{i=1,\cdots, m}\left(\inf_{\xi\in[-C,C]}\P_i'(\xi)\right).
\end{align}
Pick $\rho>0$ large enough such that
\begin{align}\label{rou}
\beta\rho\kappa\geq\left(\beta q^*+ c M+ \Lambda q^*+\bar DM\right)
\|\zeta\|_{L^\i(\bar\Omega)}
+\bar Dq^*\| \Delta \zeta\|_{L^\i(\bar\Omega)}
+2\bar D M\| \nabla \zeta\|_{L^\i(\bar\Omega)}.
\end{align}
 Let $T^*>0$ large enough such that
\begin{align}\label{T}
c T^*\geq 1+\rho +L\ \text{ and }\ a e^{-b(-L+c T^*-\rho)}\leq\dl q_*.
\end{align}
Take any $T\geq T^*$.

 We now turn to check the boundary conditions.
For $t\geq0$ and $x\in\p\Omega$, it follows from \eqref{ck} and  \eqref{T} that
$\xi^-(t,x)\geq -L+c T^*-\rho \geq1$,
then $ Q(\xi^-(t,x))=Q_1$  by \eqref{2C.a}.
From \eqref{1.a}, \eqref{2C.2}, \eqref{ck},  \eqref{db}, \eqref{T} and $\nu\cdot\nabla\zeta=1$ on $\p\Omega$, one infers that
\begin{align*}
\nu\cdot \nabla w_i^-(t,x)
\leq&|\Phi_i'(\xi^-(t,x))|-\dl q_i^1 e^{-\beta t} \\
\leq&a e^{-b(-L+c(t+T^*)-\rho)}-\dl q_*e^{-\beta t}\\
\leq& \dl q_* e^{-b c t}-\dl q_*e^{-\beta t}\\
\leq&0
\end{align*}
for each $i=1,\cdots,m$ and
for all $t\geq0$ and $x\in\p\Omega$ such that $w_i^-(t,x)>0$.

Let $i=1,\cdots,m$ be any fixed index.
It suffices to prove that
$$
\mathscr L_i[ w^-](t,x)=(w_i^-)_t(t,x)-D_i\Delta w_i^-(t,x)-F_i(w^-(t,x))\leq0
$$
for all $t\geq0$ and $x\in\Omega$ such that $w_i^-(t,x)>0$.
By direct calculations, one has
\begin{align*}
\mathscr L_i[ w^-](t,x)
=&F_i(\Phi(\xi^-(t,x)))-F_i(w^-(t,x))-\rho\dl\beta e^{-\beta t}\Phi_i'(\xi^-(t,x))+\dl\beta e^{-\beta t}q_i(\xi^-(t,x))\zeta(x)\\
&-(c-\rho\dl\beta e^{-\beta t})\dl e^{-\beta t}q_i'(\xi^-(t,x))\zeta(x)
+D_i\dl e^{-\beta t}q_i(\xi^-(t,x)) \Delta \zeta(x) \\
&
+2D_i\dl e^{-\beta t}\nabla q_i(\xi^-(t,x)) \nabla \zeta(x)
+D_i\dl e^{-\beta t} \Delta q_i(\xi^-(t,x)) \zeta(x).
\end{align*}

If $\xi^-(t,x)\geq C$, then $ Q(\xi^-(t,x))=  Q_1$ by \eqref{2C.a}.
From \eqref{db} and \eqref{pc}, there holds
$$
\bm1\geq \Phi(\xi^-(t,x))\geq   w^-(t,x)\geq (1-\dl_1)\cdot\bm1-\dl q^*\|\zeta\|_{L^\i(\bar\Omega)}\cdot\bm1\geq(1-4\vp_0)\cdot\bm1.
$$
Hence, it follows from \eqref{2C.3}   that
\begin{align*}
F_i( \Phi(\xi^-(t,x)))-F_i( w^-(t,x))
\leq\sum_{j=1}^m\mu_{ij}^1\dl q_j^1\zeta(x)e^{-\beta t}
\leq-\dl\varpi q_i^1\zeta(x)e^{-\beta t}.
\end{align*}
By a direct calculation, one can deduce from $\Phi_i'>0$, \eqref{2C.b}, \eqref{2l.a}  and \eqref{db}  that
\begin{align*}
\mathscr L_i[  w^-](t,x)
\leq&\dl e^{-\beta t}q_i^1\zeta(x)\left(\beta +\bar D\e\|\frac{\Delta\zeta}{\zeta}\r\|_{L^\i(\bar\Omega)}-\varpi\right)
\leq0.
\end{align*}

If $\xi^-(t,x)\leq- C$, then $  Q(\xi^-(t,x))=  Q_0$ by \eqref{2C.a}.
Similarly,
there holds $\mathscr L_i[ w^-](t,x)\leq0$.

We now consider the case $-C\leq\xi^-(t,x)\leq C$. Thanks to
$\bm0\leq w^-(t,x)\leq \Phi(\xi^-(t,x))\leq \bm1$,
it follows from \eqref{2C.2} and \eqref{2C.7}  that
\begin{align*}
F_i(\Phi(\xi^-(t,x)))-F_i(w^-(t,x))
\leq\dl\Lambda q^*\zeta(x)e^{-\beta t}.
\end{align*}
Hence, one can deduce from \eqref{2C.2},   \eqref{2C.b}, \eqref{2l.a}, \eqref{db}, \eqref{kapp} and \eqref{rou} that
\begin{align*}
\mathscr L_i[ w^-](t,x)
\leq&
-{\rho\dl}\beta \kappa e^{-\beta t}+\dl\beta q^* e^{-\beta t}\|\zeta\|_{L^\i(\bar\Omega)}
+\dl c e^{-\beta t}|q_i'(\xi^-(t,x))|\|\zeta\|_{L^\i(\bar\Omega)}\\
&+\bar D\dl e^{-\beta t}q^*\| \Delta \zeta\|_{L^\i(\bar\Omega)}
+2\bar D\dl e^{-\beta t} |q_i'(\xi^-(t,x))|\| \nabla \zeta\|_{L^\i(\bar\Omega)}\\
&+\bar D\dl e^{-\beta t} |q_i''(\xi^-(t,x))|\| \zeta\|_{L^\i(\bar\Omega)}
+\dl\Lambda q^*\|\zeta\|_{L^\i(\bar\Omega)}e^{-\beta t}\\
\leq&
\dl e^{-\beta t}\left(-\rho\beta \kappa+
\left(\beta q^*+ c M +\Lambda q^*+\bar DM\right)
\|\zeta\|_{L^\i(\bar\Omega)}\r.\\
&\left.+\bar Dq^*\| \Delta \zeta\|_{L^\i(\bar\Omega)}+2\bar D M\| \nabla \zeta\|_{L^\i(\bar\Omega)}
\right)\\
\leq&0.
\end{align*}

As a conclusion, $w^-$ is a subsolution of \eqref{1.1} in $[0,+\i)\times\bar\Omega$.
The proof is complete.
\end{pr}

Let $u$ be the entire solution obtained in Theorem \ref{t1}. In the sequel, we always assume that $ u$ propagates completely in the sense of \eqref{comp}.
\begin{lem}\label{up1}
For any $\vp>0$, there exist some constants $K_1>0$, $K_2>0$ and $\bar T\in\mathbb R$ such that
 \begin{align}\label{K1}
 u(t,x)\geq(1-\vp q_*)\cdot\bm1\ \ \text{ for }t\geq\bar T \text{ and }
 x\in\bar\Omega \text{ such that } x_1+c t\geq K_1
 \end{align}
 and
 \begin{align}\label{K2}
 u(t,x)\leq\vp q_*\cdot\bm1\ \ \text{ for }t\geq\bar T  \text{ and }
 x\in\bar\Omega \text{ such that }x_1+c t\leq- K_2.
 \end{align}
\end{lem}
\begin{pr}
Let $\vp>0$ be any fixed constant. We define
\begin{align}\label{dd}
\dl=\min\e(1,\frac{2\vp}{q^*\|\zeta\|_{L^\i(\bar\Omega)}},\frac{\vp q_*}{2q^*\|\zeta\|_{L^\i(\bar\Omega)}}\r)\e(\leq\frac{\vp}{2}\r).
\end{align}
By the properties of $\Phi$, there exists $A>0$  such that
\begin{align}\label{R}
\bc
\Phi(x_1+c t)\leq\dl q_*\cdot\bm1 &\text{for } t\in\R \text{ and }x\in\R^N \text{ such that }x_1+c t\leq-A+c T^*,\\
\Phi(x_1+c t)\geq(1-\dl q_*)\cdot\bm1 &\text{for } t\in\R \text{ and }x\in\R^N \text{ such that }x_1+c t\geq A+c T^*,
\ec
\end{align}
where $T^*>0$ is defined as in Lemma \ref{w-}.
By \eqref{1.5}, there exists $\hat T<0$ negative enough such that
\begin{align}\label{tt}
-\frac{1}{2}\dl q_*\cdot\bm1\leq u(\hat T,x)-\Phi(x_1+c \hat T)\leq\frac{1}{2}\dl q_*\cdot\bm1\ \ \text{ for all }x\in\bar\Omega.
\end{align}

{\it Step 1: proof of \eqref{K1}}.
By $\Phi(+\i)=\bm1$ and \eqref{tt}, there exists $\tilde A_1>0$ large enough such that
\begin{align}\label{utt}
 u(\hat T,x)\geq \Phi(x_1+c\hat T)-\frac{1}{2}\dl q_*\cdot\bm1\geq (1-\dl q_*)\cdot\bm1
\end{align}
for all $x\in\bar\Omega$ such that $x_1+c\hat T\geq\tilde A_1$.
It is evident that $d_\Omega(y,\{x\in\bar\Omega:x_1+c\hat T\geq\tilde A_1\})<+\i$ for any $y=(y_1,y')\in\bar\Omega$ such that $y_1\geq-A$,
where $y'=(y_2,\cdots,y_N)$.
Recall that $B(\bm\theta,L)\subset K$ from \eqref{ck}.
By Lemma \ref{key}, there are four positive constants
 $R_1=R_1(\dl)$, $R_2=R_2(\dl)$, $R_3=R_3(\dl)$ and $\tilde T=\tilde T(\dl)$
satisfying $R_1<R_2<R_3$, if $B(x_0,R_3)\subset\Omega$ for some
 $x_0\in\bar\Omega$, then $v_{x_0,R_1}^\dl(\tilde T,\cdot)\geq (1-\dl q_*)\cdot\bm1$ in $\bar{B(x_0,R_2)}\subset\Omega$,
 where $v_{x_0,R_1}^\dl$ is the solution of Cauchy problem \eqref{vr} and \eqref{vr0}
 with $R$ replaced by $R_1$.
For any $y\in\bar{\Omega\backslash B(\bm\theta,L+R_3-R_2)}\cap\{x\in\R^N:x_1\geq-A\}$,
there exist $k$ points $x^1,\cdots,x^k\in\R^N$ such that
$$
\bc
B(x^1,R_1)\subset \{x\in\R^N:x_1\geq\tilde A_1-c\hat T\},\\
B(x^1,R_3)\subset\Omega,&\ \ 1\leq i\leq k,\\
B(x^{i+1},R_1)\subset B(x^{i},R_2),&\ \ 1\leq i\leq k-1,\\
y\in B(x^{k},R_2).
\ec
$$

In view of \eqref{utt}, Lemma \ref{key} and the comparison principle, there holds
$$
 u(\tilde T+\hat T,\cdot)\geq  v_{x^1,R_1}^\dl(\tilde T,\cdot)\geq(1-\dl q_*)\cdot\bm1\ \ \text{ in }\ \bar{B(x^1,R_2)}.
$$
Since $B(x^2,R_1)\subset B(x^1,R_2)$, then $  u(\tilde T+\hat T,\cdot)\geq(1-\dl q_*)\cdot\bm1$ in $B(x^2,R_1)$.
By similar arguments, one has  $
 u(2\tilde T+\hat T,\cdot)\geq(1-\dl q_*)\cdot\bm1$
 in $\bar{B(x^2,R_2)}$. By induction, there holds
 $u(k\tilde T+\hat T,\cdot)\geq (1-\dl q_*)\cdot\bm1$
 in $\bar{B(x^k,R_2)}$.
 In particular,
 $$
 u(k\tilde T+\hat T,\cdot)\geq (1-\dl q_*)\cdot\bm1\ \ \text{ in }\
 \bar{\Omega\backslash B(\bm\theta,L+R_3-R_2)}\cap\{x\in\R^N:x_1\geq-A\}.
 $$
On the other hand, thanks to \eqref{comp}, there exists $T'>0$ large enough such that
$$
 u(T',\cdot)\geq (1-\dl q_*)\cdot\bm1\ \ \text{ in }\
 \bar{ B(\bm\theta,L+R_3-R_2)\backslash K}.
$$
Denote
$$
\bar T=\max(k\tilde T+\hat T,T')(\geq \hat T).
$$
Since $ u_t\gg\bm0$ by  Theorem \ref{t1}, there holds
$$
  u(\bar T,\cdot)\geq (1-\dl q_*)\cdot\bm1\ \ \text{ for } x\in\bar\Omega \text{ such that } x_1\geq-A.
$$

Let $w^-$ be the subsolution defined as in Lemma \ref{w-}.
For all  $x\in\bar\Omega$ such that $x_1\geq-A$, it follows from $\Phi\leq\bm1$, \eqref{2C.2} and \eqref{zeta} that
$$
 w^-(0,x)=\max\left(\Phi(x_1+c T^*)-\dl Q(x_1+c T^*)\zeta(x),\bm0\right)
\leq(1-\dl q_*)\cdot\bm1\leq u(\bar T,x).
$$
For all $x\in\bar\Omega$ such that $x_1\leq-A$, it follows from \eqref{2C.2}, \eqref{zeta} and \eqref{R} that
$$
w^-(0,x)
\leq\max\left(\dl q_*\cdot\bm1-\dl q_*\cdot\bm1 ,\bm0\right)=\bm0
\leq u(\bar T,x).
$$
As a result, there holds $ w^-(0,x)\leq  u(\bar T,x)$ for $x\in\bar\Omega$.
Together with $\Phi(+\i)=\bm1$, \eqref{2C.2}, \eqref{dd} and the comparison principle, there exists $K_1>0$ large enough such that
\begin{align*}
  u(t,x)\geq&w^-(t-\bar T,x)\\
 \geq&\max\left(\Phi(x_1+c(t-\bar T+T^*)-\rho\dl (1-e^{-\beta (t-\bar T)}))-\dl q^*\|\zeta\|_{L^\i(\bar\Omega)}\cdot\bm1 ,\bm0\right)
 \geq(1-\vp q_*)\cdot\bm1
\end{align*}
 for all $t\geq \bar T$ and $x\in\bar\Omega$ such that $x_1+c t\geq K_1$.

{\it Step 2: proof of \eqref{K2}.}
By $\Phi(+\i)=\bm1$ and \eqref{tt}, there exists $\tilde A_2>0$ large enough such that
\begin{align*}
 u(\hat T,x)\leq \Phi(x_1+c\hat T)+\frac{1}{2}\dl q_*\cdot\bm1\leq \dl q_*\cdot\bm1
\end{align*}
for all $x\in\bar\Omega$ such that $x_1+c\hat T\leq-\tilde A_2$.
Let $w^+$ be the supersolution defined as in Lemma \ref{w-}.
From \eqref{2C.2}  and \eqref{zeta}, one has
$$
w^+(0,x)=\min\left(\Phi(x_1+c T^*)+\dl Q(x_1+c T^*)\zeta(x),\bm1\right)
\geq
\dl q_*\cdot\bm1\geq u(\hat T,x)
$$
for all  $x\in\bar\Omega$ such that $x_1+c\hat T\leq-\tilde A_2$.
Since $\hat T$ defined as in \eqref{tt} is sufficiently negative, one can assume that
$-\tilde A_2-c\hat T\geq A$.
For all $x\in\bar\Omega$ such that $x_1+c\hat T\geq-\tilde A_2$,
there holds $x_1+c T^*\geq A+c T^*$, hence it follows from \eqref{2C.2}, \eqref{zeta} and \eqref{R} that
\begin{align*}
 w^+(0,x)\geq\min\e((1-\dl q_*)\cdot\bm1+\dl q_*\cdot\bm1,\bm1\r)=\bm1\geq  u(\hat T,x).
\end{align*}
As a result, there holds $ w^+(0,x)\geq u(\hat T,x)$ for $x\in\bar\Omega$.
In view of $\Phi(-\i)=\bm1$, \eqref{2C.2}, \eqref{dd} and the comparison principle,
there exists $K_2>0$ large enough such that
\begin{align*}
 u(t,x)\leq& w^+(t-\hat T,x)\\
 \leq&\min\left(\Phi(x_1+c(t-\hat T+T^*)-\rho\dl (1-e^{-\beta (t-\hat T)})) +\dl q^*\|\zeta\|_{L^\i(\bar\Omega)}\cdot\bm1,\bm1\right)
 \leq\vp q_*\cdot\bm1
\end{align*}
for all  $t\geq \hat T$ and
$x\in\bar\Omega$ such that $x_1+c t\leq- K_2$.
Then  \eqref{K2} follows since $\bar T>\hat T$. The proof is complete.
\end{pr}

\begin{lem}\label{up2}
Let $\bar T>0$ be the constant such that Lemma \ref{up1} holds. For any $\vp>0$, there exists $\bar R_\vp>0$ such that
$$
-\vp q_*\cdot\bm1\leq u(t,x)-\Phi(x_1+c t)\leq\vp q_*\cdot\bm1\
\text{ for all $t\geq\bar T$ and $x\in\bar{\Omega\backslash B(\bm\theta,\bar R_\vp)}$.}
$$
\end{lem}

\begin{pr}
Fix any $r\in\R$. Pick a sequence of points $(x_n)_{n\in\mathbb N}=((x_{1,n},x'_n))_{n\in\mathbb N}\subset\R\times\R^{N-1}$ such that
$x_{1,n}=r$ for all $n\in\mathbb N$
 and $|x_n|\to+\i$ as $n\to+\i$.
Denote
$$
 u_n(t,x)= u(t,x+x_n)=u(t,x_1+x_{1,n},x'+x_n')\ \ \text{ for all }t\in\R
 \text{ and }  x\in\Omega_n=\Omega-x_n,
$$
where $x'=(x_2,\cdots,x_N)$.
Since the obstacle $K$ is compact,  the sets $\Omega_n\to\R^N$ as $n\to+\i$.
By $\bm0\ll  u(t,x)\ll\bm1$ and standard parabolic estimates, up to extraction of a subsequence,
the functions $ u_n(t,x)\to U(t,x)$ locally uniformly in $(t,x)\in\R\times\R^N$ as $n\to+\i$,
where $U=(U_1,\cdots,U_m)$ is a classical solution of
\begin{align}\label{UU}
 U_t(t,x)-D\Delta  U(t,x)- F( U(t,x))=\bm0, \ \ (t,x)\in\R\times\R^N.
\end{align}
It is evident that $\bm0\leq U\leq\bm1 $.
By \eqref{1.5}, there holds
\begin{align}\label{Up1}
U(t,x)\to\Phi(x_1+c t+r)\ \text{ uniformly in }x\in\mathbb R^N\text{ as }t\to-\i.
\end{align}
We claim that
\begin{align}\label{Up}
 U(t,x)=\Phi(x_1+c t+r)\ \ \text{ for all }t\in\R \text{ and }x\in\R^N.
\end{align}
To ensure the coherence of the argument, we postpone the proof of this claim at the end of the proof of Lemma \ref{up2}.

From \eqref{Up}, one has
\begin{align}\label{unp}
u(t,x+x_n)\to \Phi(x_1+c t+r)\ \ \text{ locally uniformly in }(t,x)\in\R\times\R^N \text{ as } n\to+\i.
\end{align}
Fix any $t_*\geq\bar T$ and any $\vp>0$.
Choose a sequence of points $(y_n)_{n\in\mathbb N}=((y_{1,n},y_n'))_{n\in\mathbb N}\subset \R\times\R^{N-1}$ such that
$y_{1,n}=-c t_*$ for all $n\in\mathbb N$  and $ |y_n|\to+\i$  as $ n\to+\i$.
 Since $ \Phi(-\i)=\bm0$ and $ \Phi(+\i)=\bm1$, one can pick two positive constants $K_1$ and $K_2$ (depend on $\vp$) such that
 \begin{align}\label{fz}
 \bc
  \Phi(x_1+c t)\geq(1-\vp q_*)\cdot\bm1&\text{ for } t\geq\bar T \text{ and }x\in\bar\Omega
 \text{ such that }  x_1+c t\geq K_1,\\
  \Phi(x_1+c t)\leq\vp q_*\cdot\bm1& \text{ for }t\geq\bar T \text{ and }x\in\bar\Omega
  \text{ such that }  x_1+c t\leq -K_2.
 \ec
 \end{align}
Then there exists $R_\vp>0$ large enough such that
$$
\left\{x\in\bar\Omega:-K_2\leq x_1+c t_*\leq K_1\right\}\subset \bigcup_{n=1}^\i B(y_n,R_\vp).
$$
Even if it means increasing $ R_\vp$, one can take $K_1$ and $K_2$ sufficiently large.
It follows from \eqref{unp} that there exists $n_0\in\mathbb N$ large enough such that
$$
-\vp q_*\cdot\bm1\leq u(t_*,x+y_n)-\Phi(x_1)\leq\vp q_*\cdot\bm1
\ \
\text{ for }n\geq n_0\ \text{ and }x\in \bar{B(\bm\theta, R_\vp)}.
$$
Hence, there exists a large constant $\bar R_\vp>0$ such that
$
-\vp q_*\cdot\bm1\leq  u(t_*,x)-  \Phi(x_1+c t_*)\leq\vp q_*\cdot\bm1
$
 for all $x\in\bar\Omega$ such that $-K_2\leq x_1+c t_*\leq K_1$  and $|x|\geq\bar R_\vp$.

Since $t_*\geq\bar T$ is arbitrary,  then
$
-\vp q_*\cdot\bm1\leq  u(t,x)-  \Phi(x_1+c t)\leq\vp q_*\cdot\bm1
$
 for $t\geq\bar T$ and $x\in\bar\Omega$ such that $-K_2\leq x_1+c t\leq K_1$  and $|x|\geq\bar R_\vp$.

 Since $\bm0\ll u\ll\bm1$, by increasing $K_1$ and $K_2$ if necessary, one infers from \eqref{fz} and Lemma \ref{up1} that
 $
-\vp q_*\cdot\bm1\leq  u(t,x)- \Phi(x_1+c t)\leq\vp q_*\cdot\bm1
$
 for $t\geq\bar T$ and $x\in\bar\Omega$ such that $ x_1+c t\leq -K_2$ or $ x_1+c t\geq K_1$. As a conclusion,
The statement of Lemma \ref{Up} is thereby true.

\textit{Proof of claim \eqref{Up}.}
Choose any
\begin{align}\label{sw}
0<\sigma<{2\vp_0}\ \text{ and }\ 0<w<\varpi.
\end{align}
Since $\Phi(-\i)=\bm0$ and $\Phi(+\i)=\bm1$, there exists $C>1$   such that
\begin{align}\label{cs}
\Phi(\xi)\leq \sigma \cdot\bm1\ \text{ for }\xi\leq-C \ \text{ and } \
\Phi(\xi)\geq (1-\sigma )\cdot\bm1 \ \text{ for }\xi\geq C.
\end{align}
Fix any $\vp\in(0,\sigma/ q_*)$. By \eqref{Up1},  there exists $T_1<0$ negative enough such that
 \begin{align}\label{Up2}
-\vp q_*\cdot\bm1\leq U(t,x)-\Phi(x_1+c t+r)\leq\vp q_*\cdot\bm1\ \ \text{ for all }t\leq T_1\text{ and }x\in\R^N.
\end{align}
Furthermore, there exists $\kappa>0$ such that
 \begin{align}\label{ka}
\kappa=\min_{i=1,\cdots, m}\left(\inf_{\xi\in[-C,C]}\P_i'(\xi)\right).
\end{align}
Fix any $T_2<T_1$. Let
\begin{align}\label{ro}
\rho=\kappa^{-1}(w q^* +\Lambda q^*
+c M
+\bar DM).
\end{align}

 For all $t\geq T_2$ and $x\in\R^N$, define the functions $ U^\pm=(U^\pm_1,\cdots,U^\pm_m)$ as
$$
U^-(t,x)=\max\e( \Phi(\xi^-(t,x))-\vp Q\e(\xi^-(t,x)\r)e^{-w (t-T_2)},\bm0\r)
$$
and
$$
U^+(t,x)=\min\e( \Phi(\xi^+(t,x))+\vp Q\e( \xi^+(t,x)\r)e^{-w (t-T_2)},\bm1\r),
$$
where
$\xi^\pm(t,x)=x_1+c t+r\pm{\vp\rho}{w^{-1}}\e(1-e^{-w (t-T_2)}\r)$.
We shall prove that the function $U^-$ is a subsolution of \eqref{UU} in $[T_2,+\i)\times\R^N$,
a supersolution $ U^+$ can be obtained by  similar arguments.
Let $i=1\cdots,m$ be any fixed index. It suffices to verify that
$$
\mathscr L_i[ U^-](t,x)=(U_i^-)_t(t,x)-D_i\Delta U_i^-(t,x)-F_i(U^-(t,x))\leq0
$$
for all $t\geq T_2$ and $x\in\R^N$ such that $U_i^-(t,x)>0$.
By  direct calculations, there  holds
\begin{align*}
\mathscr L_i[ U^-](t,x)
=&-\vp \rho e^{-w( t-T_2)}\Phi_i'(\xi^-(t,x))+\vp w  e^{-w( t-T_2)}q_i(\xi^-(t,x))\\
&-\e(c-\vp \rho e^{-w( t-T_2)}\r)\vp e^{-w( t-T_2)}q_i'(\xi^-(t,x))
+D_i\vp  e^{-w( t-T_2)}q_i''(\xi^-(t,x))\\
&+F_i( \Phi(\xi^-(t,x)))-F_i( U^-(t,x)).
\end{align*}

If $\xi^-(t,x)\leq-C$, then  $ Q(\xi^-(t,x))= Q_0$ by \eqref{2C.a}.
Thanks to \eqref{sw} and \eqref{cs}, one has
$
\bm0\leq U^-(t,x)\leq \Phi(\xi^-(t,x))\leq \sigma \cdot\bm1\leq 4\vp_0\cdot\bm1
$.
It then follows from \eqref{2C.3} that
\begin{align*}
F_i(\Phi(\xi^-(t,x)))-F_i(U^-(t,x))
\leq\sum_{j=1}^m\mu_{ij}^0\vp q_j^0e^{-w(t-T_2)}
\leq-\vp\varpi q_i^0e^{-w(t-T_2)}.
\end{align*}
Hence, one derives from $\Phi_i'>0$ and \eqref{sw} that
\begin{align*}
\mathscr L_i[  U^-](t,x)
\leq&\vp  e^{-w( t-T_2)}q_i^0\left(w -\varpi\right)
\leq0.
\end{align*}

If $\xi^-(t,x)\geq C$, then $ Q(\xi^-(t,x))= Q_1$ by \eqref{2C.a}.
Similarly,
$\mathscr L_i[ U^-](t,x)\leq0$.

We now consider the case $-C\leq\xi^-(t,x)\leq C$. Since
$\bm0\leq U^-(t,x)\leq \Phi(\xi^-(t,x))\leq \bm1$, it follows from \eqref{2C.2} and \eqref{2C.7} that
\begin{align*}
F_i(\Phi(\xi^-(t,x)))-F_i( U^-(t,x))
\leq\vp\Lambda q^*e^{-w(t-T_2)}.
\end{align*}
By \eqref{2C.2},  \eqref{2C.b}, \eqref{2l.a}, \eqref{ka} and \eqref{ro}, there holds
\begin{align*}
\mathscr L_i[ U^-](t,x)
\leq\vp e^{-w( t-T_2)}\e(-\rho \kappa+ w q^* +\Lambda q^*
+c M
+\bar DM\r)
\leq0.
\end{align*}

As a result, the function $U^-$ is a subsolution of \eqref{UU} in $[T_2,+\i)\times\R^N$.
From \eqref{2C.2} and \eqref{Up2}, one has
$U^-(T_2,x)\leq U(T_2,x)\leq U^+(T_2,x)$ for $x\in\R^N$.
The comparison principle leads to $ U^-(t,x)\leq  U(t,x)\leq U^+(t,x)$ for $t\geq T_2$ and $x\in\R^N$.
By passing $T_2\to-\i$, one has
$$
 \Phi\e(x_1+c t+r-{\vp\rho}{w^{-1}}\r)\leq  U(t,x)\leq \Phi\e(x_1+c t+r+{\vp\rho}{w^{-1}}\r)
$$
for all $t\in\R$ and $x\in\R^N$. By the arbitrariness of $\vp$,  one concludes that the claim \eqref{Up} holds.
The proof of Lemma \ref{up2} is thereby complete.
\end{pr}

\subsection{Proof of Theorem \ref{t2}}
In this subsolution, we complete the proof of Theorem \ref{t2}. The strategy is to prove that the entire solution
$u(t,x)$ obtained in Theorem \ref{t1} is uniformly bounded from below and from above respectively, as close as we want,
by the planar front $\Phi(x_1+c t)$  for large time. The main idea is to construct a pair of suitable sub- and supersolutions.

We first introduce some notations and parameters.
Since $\Phi(-\i)=\bm0$ and $\Phi(+\i)=\bm1$, there exists $C>2$   such that
\begin{align}\label{C}
\Phi(\xi)\leq\vp_0\cdot\bm1\ \text{ for }\xi\leq -C
\ \ \text{ and }\ \
\Phi(\xi)\geq( 1-\vp_0)\cdot\bm1\ \text{ for }\xi\geq C.
\end{align}
Furthermore, there exists $\tilde\kappa>0$ such that
\begin{align}\label{tkp}
\tilde \kappa:=\min_{i=1,\cdots,m}\left(\inf_{\xi\in[-C,C]}\Phi_i'(\xi)\right).
\end{align}
Even if it means increasing $\bar K_1$ given by \eqref{impo}, we can assume that
\begin{align}\label{bk1}
\bar K_1\geq\max&\left(1,\frac{1}{\tilde \kappa},
\frac{(N-1)L^2}{4c\bar D},\frac{(N-1)L^2}{4\bar D}, \frac{N-1}{8},\frac{N-1}{2c},\right.\nonumber\\
&\left.\tilde\kappa^{-1}
 \left(\max_{i=1,\cdots,m}\|\Phi_i'\|_{L^\i(\R)}+\max_{i=1,\cdots,m}\|\Phi_i''\|_{L^\i(\R)}+M
\right)\right).
\end{align}
By the properties of $\Phi$ and \eqref{1.a}, it follows that there exists $0<\kappa_0<1-\max_{i=1,\cdots,m}\Phi_i(2)$ and $\bar K_2>0$ such that, for all $0<\kappa\leq\kappa_0$,
\begin{align}\label{bk22}
(\Phi(\xi)\geq(1-\kappa)\cdot\bm1)\Rightarrow(\Phi'(\xi)\leq \bar K_2\kappa^{1/2} e^{-b\xi/2}\cdot\bm1).
\end{align}
Set
\begin{align}\label{ome}
\omega=\min\left(\frac{b c}{4},\frac{\varpi q_*  }{16 q^*},\frac{1}{2}\right).
\end{align}

Define
\begin{align}\label{sss}
  \sigma=\min\left(1,\frac{\varpi }{8\bar D  \|\Delta\zeta\|_{L^\i(\bar\Omega)}},\frac{\varpi  }{16 \| \zeta\|_{L^\i(\bar\Omega)}},
  \frac{\varpi }{2\Lambda\|\zeta\|_{L^\i(\bar\Omega)}}\right).
\end{align}
Set
\begin{align}\label{rr}
\rho=\left(2\Lambda q^*+2\sigma\omega q^*\right)\|\zeta\|_{L^\i(\bar\Omega)}+\bar DM+cM
+2q^*\omega +\sigma q^*\bar D\|\Delta\zeta\|_{L^\i(\bar\Omega)}.
\end{align}
Denote
\begin{align}\label{bk5}
\bc
\bar K_3=\bar K_2 \sigma^{-1} q_*^{-1}e^{b(L+1)}\left(1+\frac{(N-1)L}{4\bar D}\right),\\
\bar K_4=(2+\omega)\omega\tilde\kappa(12+5\omega)^{-1}\rho^{-1}\left(\max_{i=1,\cdots,m}\|\Phi_i'\|_{L^\i(\R)}\right)^{-1},\\
\bar K_5=2\omega\tilde\kappa\rho^{-1}\left(\max_{i=1,\cdots,m}\|\Phi_i'\|_{L^\i(\R)}\right)^{-1}.
\ec
\end{align}
It is evident that $\bar K_4\leq\bar K_5$.
Fix an arbitrary real number $\vp$ satisfying
\begin{align}\label{vpp}
0<\vp<\min\left(\min_{i=1,\cdots,m}\|\Phi_i'\|_{L^\i(\R)},\frac{1}{\bar K_4},
\frac{\bar K_3\kappa_0^{1/2}}{\bar K_4},\frac{\bar K_3^2}{\bar K_4},
\frac{\tilde\kappa}{2M\bar K_4},
\frac{\vp_0}{\bar K_4q^*\|\zeta\|_{L^\i(\bar\Omega)}} \right).
\end{align}
Denote
\begin{align}\label{muuu}
\mu=\min\left(1,\bar K_4\vp,
\bar K_5\vp,\bar K_3\kappa_0^{1/2},
\bar K_3^2 ,\frac{\tilde\kappa}{2M},\frac{\vp_0}{q^*\|\zeta\|_{L^\i(\bar\Omega)}}
\right)=\bar K_4\vp\leq\frac{\vp_0}{q^*}
\end{align}
and
\begin{align}\label{lamm}
\lambda&=\min\left(\kappa_0,\mu  ,4\mu^2\bar K_3^{-2}(2+\omega)^{-2}\right)
=4\mu^2\bar K_3^{-2}(2+\omega)^{-2}=4\bar K_3^{-2}(2+\omega)^{-2}\bar K_4^2\vp^2.
\end{align}

By $u_t\gg\bm0$ and  \eqref{comp}, there exists $t_{\lambda}> \max(\bar T,(L+1)/c)$ such that
\begin{align}\label{tl}
  u(t_\lambda-1+t,x)\geq u(t_\lambda,x)\geq (1-\lambda)\cdot\bm1\  \text{ for all }t\geq 1\text{ and }x\in\p \Omega,
\end{align}
where $\bar T>0$ is the constant such that Lemma \ref{up1} holds.
By Lemma \ref{up2}, there is $\bar R_{\lambda}>0$
such that
\begin{align}\label{rl}
-\lambda q_*\cdot\bm1\leq  u(t_{\lambda},x)-\Phi(x_1+c t_{\lambda})\leq  \lambda q_*\cdot\bm1
\ \text{ for }x\in \bar{\Omega\backslash B(\bm\theta,\bar R_{\lambda})}.
\end{align}
By Theorem \ref{t1},  $\bm0\ll u \ll \bm1$ in $\R\times\bar\Omega$.
Then $\min_{\bar{\Omega}\cap\bar{ B(\bm\theta,\bar R_{\lambda})}} u(t_\lambda,\cdot)\gg\bm0$ and $\max_{\bar{\Omega}\cap\bar{ B(\bm\theta,\bar R_{\lambda})}} u(t_\lambda,x)\ll\bm1$.
Furthermore, by $\Phi(-\i)=\bm0$ and $\Phi(+\i)=\bm1$, there exists $\beta>0$ such that
\begin{align}\label{ul2}
\Phi(x_1+c t_\lambda-\beta e^{-|x'|^2})\leq u(t_\lambda,x)\leq\Phi(x_1+c t_\lambda+\beta e^{-|x'|^2})
\end{align}
 for all $x=(x_1,x')\in \bar{\Omega}\cap\bar{ B(\bm\theta,\tilde R_{\lambda})}$,
 where $x'=(x_2,\cdots,x_N)$ and $|x'|^2=x_1^2+\cdots+x_N^2$.
Since $\Phi$ is increasing, $\beta$ can be chosen large enough such that
\begin{align}\label{bt1}
\bc
\beta\geq\max\left(1,\frac{\gamma_0}{8\bar K_1 \bar D}, \frac{N-1}{4},\frac{1}{8\bar D}\right),\\
c(t_{\lambda}+2+2\omega^{-1})-\beta(3+2\omega^{-1})^{-1}e^{-\frac{(N-1)L^2}{3+2\omega^{-1}}}+L<0,\\
a\left(1+\frac{(N-1)L}{4\bar D}\right)e^{-b(-L+c t_{\lambda}+\beta e^{-(N-1)L^2})}
\leq\sigma\mu q_*,
\ec\end{align}
where $\gamma_0>0$  such that
\begin{align}\label{gm}
e^{-{s^{-1}(N-1)L^2}}\geq1-{2s^{-1}(N-1)L^2}\ \text{ for all }s\geq\gamma_0.
\end{align}
Define
\begin{align}\label{ag}
  \bc
  \alpha=\frac{\underline D(N-1)}{8\bar K_1\beta\bar D}\leq\frac{N-1}{4\bar K_1\beta},\\
  \gamma=\max\left(\gamma_0,\gamma_1,2\beta (N-1)L^2c^{-1},2\beta (N-1)L^2,8\bar K_1\beta\bar D\right)=8\bar K_1\beta\bar D\geq 8\beta\bar D
  \geq4\bar D,
  \ec
\end{align}
where
$\gamma_1=8 \beta\bar D\tilde\kappa^{-1}\left(\max_{i=1,\cdots,m}\|\Phi_i'\|_{L^\i(\R)}+\max_{i=1,\cdots,m}\|\Phi_i''\|_{L^\i(\R)}+M\right)$.
From \eqref{bk1} and \eqref{bt1}, one has $0<\alpha<1$ and $\gamma>1$.

For all $t\geq1$, we define  the function
\begin{align}\label{g}
  g(t)=L +c(t-1+t_{\lambda})-\beta t^{-\alpha}e^{-\frac{(N-1)L^2}{\gamma t}}.
\end{align}
From \eqref{ag}, one obtains that for all $t\geq1$,
\begin{align}\label{gp}
  g'(t)\geq
  c-\beta (N-1)L^2\gamma^{-1}t^{-\alpha-2}e^{-\frac{(N-1)L^2}{\gamma t}}
  \geq c-\beta(N-1)L^2\gamma^{-1}\geq\frac{c}{2}.
\end{align}
Since $\gamma>1$, $0<\alpha<1$, $\beta>0$,  $3+2\omega^{-1}\geq1$,
one gets from \eqref{bt1} that
$$
g(3+2\omega^{-1})
\leq c(t_{\lambda}+2+2\omega^{-1})
-\beta(3+2\omega^{-1})^{-1} e^{-\frac{(N-1)L^2}{3+2\omega^{-1}}}+L<0.
$$
On the other hand, since $g$ is continuous in $[1,+\i)$, there exists a unique constant $t_1$ such that
\begin{align}\label{gt1}
t_1>3+2\omega^{-1}\ \text{ and }\ g(t_1)=0.
\end{align}
Define
\begin{align*}
t_2=t_1-2-
\frac{2\omega^{-1}\bar K_3 \lambda^{1/2}}
{\bar K_3\lambda^{1/2}+\mu e^{2\omega+2-\omega(t_1-1)}}
\in(t_1-2-2\omega^{-1},t_1-2)
\end{align*}
and
\begin{align*}
\bc
\sigma_1=\omega^2(\bar K_3\lambda^{1/2}+\mu e^{2\omega+2-\omega(t_1-1)})>0,\\
\sigma_2=\bar K_3\lambda^{1/2}\mu e^{2\omega+2-\omega(t_1-1)}>0.
\ec
\end{align*}

Let $v$ be a function  defined by
\begin{align*}
v(t)=
\bc
\mu e^{-\omega(t-1)}
&\text{ for }t\in[1,t_1-2-2\omega^{-1}],\\
\bar K_3\lambda^{1/2}
+\frac{\sigma_1}{4}[(t-t_2)^2-(t_1-2-t_2)^2]
&\text{ for }t\in[t_1-2-2\omega^{-1},t_1-2],\\
\bar K_3\lambda^{1/2}+
\frac{\bar K_3\omega \lambda ^{1/2}}{2}
 [1-(t-t_1+1)^2 ]
&\text{ for }t\in[t_1-2,t_1],\\
\bar K_3\lambda^{1/2}e^{-\omega(t-t_1)}
&\text{ for }t\geq t_1.
\ec
\end{align*}
It is straightforward to check that the function $v$ is of class $C^1$ on $[1,+\i]$, and that
it is decreasing on $[1,t_2]\cup[t_1-1,+\i)$ and increasing on $[t_2,t_1-1]$.
Since $v>0$ on $[t_1,+\i]$  and
$$
\min_{t\in[1,t_1]}v(t)=v(t_2)=\bar K_3\lambda^{1/2}-\frac{\sigma_1}{4}(t_1-2-t_2)^2=\omega^2\sigma_2\sigma_1^{-1}>0.
$$
Therefore,  $v>0$ on $[1,+\i)$.
By \eqref{lamm}, there holds
 $v(t_1-1)=\bar K_3\lambda^{1/2}(1+\omega/2)=\mu=v(1)$.
 Hence
\begin{align}\label{mv}
\max_{t\geq1}v(t)=\mu.
\end{align}
By the same arguments as in \cite[Lemma 7.2]{BHM2009}, one can prove that
\begin{align}\label{vo}
-v'(t)\leq2\omega v(t)\ \text{ for all }t\geq1.
\end{align}

Let $V$ be the function defined by
\begin{align}\label{V}
V(t)=\rho\tilde\kappa^{-1}\int_t^{+\i}v(\tau)d\tau.
\end{align}
Then $V$ is decreasing in $[1,+\i)$ and $V(+\i)=0$.
By the definition of $v$, \eqref{ome}, \eqref{bk5}, \eqref{vpp}, \eqref{muuu}, \eqref{lamm} and \eqref{mv}, one has
\begin{align}\label{V1}
V(1)=
&\rho \tilde\kappa^{-1}
\left(\int_1^{t_1-2-2\omega^{-1}}v(\tau)d\tau
+\int_{t_1-2-2\omega^{-1}}^{t_1}v(\tau)d\tau
+\int_{t_1}^{+\i}v(\tau)d\tau
\right)\nonumber\\
\leq
&\rho \tilde\kappa^{-1}
\left(\mu\omega^{-1}
+(2+2\omega^{-1})\mu
+\bar K_3\lambda^{1/2}\omega^{-1}
\right)\nonumber\\
\leq&\rho\tilde\kappa^{-1}\omega^{-1}
\left(5\mu
+\bar K_3\lambda^{1/2}
\right)\nonumber\\
\leq&\rho\tilde\kappa^{-1}\omega^{-1}
\left(5\bar K_4+2(2+\omega)^{-1}\bar K_4
\right)\vp\nonumber\\
=&\left(\max_{i=1,\cdots,m}\|\Phi_i'\|_{L^\i(\R)}\right)^{-1}\vp\ (\leq1)
.
\end{align}

For all $t\geq1$, we define the functions
\begin{align*}
\tilde v(t)=\mu e^{-2\omega(t-1)}
\ \text{ and }\
 \tilde V(t)=\frac{1}{2}\rho\tilde\kappa^{-1}\omega^{-1}\mu e^{-2\omega(t-1)} .
\end{align*}
By direct calculations, there holds
\begin{align}\label{fs0}
 \tilde v'(t)=-2\omega v(t)\ \text{ for all }t\geq1.
\end{align}
By \eqref{bk5}, \eqref{vpp} and \eqref{muuu}, one gets that
\begin{align}\label{fs1}
0>\tilde V'(t)=-\rho\tilde\kappa^{-1}\mu e^{-2\omega(t-1)}
=-\rho\tilde\kappa^{-1}\tilde v(t)\ \text{ for all }t\geq1
\end{align}
and
\begin{align}\label{fs2}
\tilde V(t)\leq\frac{1}{2}\rho\tilde\kappa^{-1}\omega^{-1}\bar K_5\vp e^{-2\omega(t-1)} =
\left(\max_{i=1,\cdots,m}\|\Phi_i'\|_{L^\i(\R)}\right)^{-1}\vp e^{-2\omega(t-1)}\leq1\ \text{ for all }t\geq1.
\end{align}

\begin{lem}\label{ups}
The function $\underline u=(\underline u_1,\cdots,\underline u_m)$ defined by
$$
\underline u(t,x)=\max\left(\Phi(\xi(t,x))-v(t)Q(\chi(t,x))-\sigma v(t)q_*\zeta(x)\cdot\bm1,\bm0\right)
$$
is a subsolution of the equation satisfied by $ u(t_{\lambda}-1+t,x)$ for $t\geq1$ and $x\in\bar\Omega $,
where
$\xi(t,x)=x_1+c(t-1+t_{\lambda})-\beta t^{-\alpha}e^{-\frac{|x'|^2}{\gamma t}}+V(t)-V(1)$ and
$\chi(t,x)=\xi(t,x)-V(t)$.
\end{lem}
\begin{pr}
Let us first check the initial conditions.
For all $x\in\bar\Omega\cap\bar{B(\bm\theta,\bar R_{\lambda})}$, since $\gamma>1$ and $\Phi'\gg\bm0$, one gets from \eqref{ul2} that
\begin{align*}
  \underline u(1,x)\leq \Phi(x_1+c t_{\lambda}-\beta e^{-{|x'|^2}/\gamma} )
   \leq u(t_\lambda,x).
\end{align*}
For all $x\in\bar{\Omega\backslash B(\bm\theta,\bar R_{\lambda})}$, it follows from
 $\Phi'\gg\bm0$, $u\gg\bm0$, \eqref{2C.2}, \eqref{lamm} and \eqref{rl} that
\begin{align*}
\underline u(1,x)\leq\max\left(\Phi(x_1+c t_{\lambda})-\mu q_*\cdot\bm1,\bm0\r)
\leq \max\e( \Phi(x_1+c t_{\lambda})  -\lambda q_*\cdot\bm1,\bm0\r)\leq u(t_{\lambda},x).
\end{align*}
As a consequence,
$\underline u(1,x)\leq u(t_{\lambda},x)$  for all $x\in\bar\Omega$.

Next, we verify the boundary conditions on $\p\Omega$.
By \eqref{tl}, for all $t\geq1$ and $x\in\p \Omega$ such that $\underline u(t,x)\leq (1-\lambda)\cdot\bm1$, there holds
$\underline u(t,x)\leq u(t_{\lambda}-1+t,x)$.
We claim that
\begin{align}\label{clm}
 \nu(x)\cdot \nabla\underline u(t,x)\leq\bm0\
\text{ for all $t\geq1$ and $x\in\p \Omega$  such that $\underline u(t,x)\geq(1-\lambda)\cdot\bm1$.}
\end{align}

{\it Proof of claim \eqref{clm}.} Take any $t\geq1$ and $x\in\p \Omega$  such that $\underline u(t,x)\geq(1-\lambda)\cdot\bm1$.
From  \eqref{lamm} and the condition of the claim, one has
$$
\Phi(\xi(t,x))=\underline u(t,x)+v(t)Q(\chi(t,x))+\sigma v(t)q_*\zeta(x)\cdot\bm1\geq(1-\lambda)\cdot\bm1
\geq(1-\kappa_0)\cdot\bm1.
$$
Then it follows from \eqref{bk22} that
\begin{align}\label{PK}
  \Phi'(\xi(t,x))\leq \bar K_2\lambda^{1/2} e^{-b\xi(t,x)/2}\cdot\bm1.
\end{align}
  By  $0<\kappa_0<1-\max_{i=1,\cdots,m}\Phi_i(2)$, there holds
  $\Phi(\xi(t,x))\geq \max_{i=1,\cdots,m}\Phi_i(2)\cdot\bm1\geq\Phi(2)$.
  Since $\Phi$ is increasing, one gets that
  $\xi(t,x)\geq2$.
  On the other hand, it follows from \eqref{ck} and $V(t)\leq V(1)$ that
  $$
  2\leq\xi(t,x)\leq L+c(t-1+t_{\lambda})-\beta t^{-\alpha}e^{-\frac{(N-1)L^2}{\gamma t}} =g(t),
  $$
  where $g$ is defined as in \eqref{g}.
  By \eqref{gp} and \eqref{gt1}, one has $t\geq t_1$.

  For all $\tau\geq1$, we define the function
  $$
  l(\tau)=-L+c(\tau-1+t_{\lambda})-\beta \tau^{-\alpha}-V(1).
  $$
 Since
  $l'(\tau)=c+\alpha\beta \tau^{-\alpha -1} \geq c$,
  one gets from \eqref{ck} that
  \begin{align}\label{xl}
    \xi(t,x)\geq
    l(t)\geq l(t_1)+{c}(t-t_1).
  \end{align}
 Since $t_1\geq1$, one infers from \eqref{gm}, \eqref{ag},   \eqref{gt1} and \eqref{V1} that
  \begin{align*}
  l(t_1)=&\ l(t_1)-g(t_1)\\
  =&-2L-V(1)+\beta t_1^{-\alpha}\e(e^{-\frac{(N-1)L^2}{\gamma t_1}}-1\r)\\
  \geq &-2L -V(1)-2\beta t_1^{-\alpha}(N-1)L^2\gamma^{-1}\\
  \geq &-2L -V(1)-1\\
  \geq &-2L -2.
  \end{align*}
  Together with \eqref{PK} and \eqref{xl}, one obtains
  \begin{align}\label{pp}
   \Phi'(\xi(t,x))\leq\bar K_2\lambda^{1/2} e^{b(L+1)}e^{-b c(t-t_1)/2}\cdot\bm1.
  \end{align}
  Since $\xi(t,x)\geq2$, one has $\chi(t,x)\geq 2-  V(t)\geq1$. By \eqref{2C.a}, one has $Q(\chi(t,x))=Q_1$.
  Note that $\nu\cdot \nabla\zeta=1$, where $\nu=(\nu_1,\cdots,\nu_N)$ denotes the
  outward unit normal on $\p\Omega$.
   Since $\Phi'\gg\bm0$, $t\geq t_1\geq1$, it follows from \eqref{ck}, \eqref{ome}, \eqref{bk5}, \eqref{ag}, \eqref{pp} and the definition of $v(t)$
   that
  \begin{align*}
    \nu\cdot\nabla \underline u_i(t,x)
    =&\left(\nu_1+2\sum_{j=2}^N\nu_i\beta t^{-\alpha-1}x_i\gamma^{-1}e^{-\frac{|x'|^2}{\gamma t}}\right)\Phi_i'(\xi(t,x))
    -\sigma v(t)q_*\\
    \leq&\e(1+2\beta\gamma^{-1}(N-1)L\r)\Phi_i'(\xi(t,x))-\sigma v(t)q_*\\
    \leq& \e(1+\frac{(N-1)L}{4\bar D}\r)\bar K_2\lambda^{1/2}e^{b(L+1)}e^{-b c(t-t_1)/2}-\sigma v(t)q_*\\
    \leq&\sigma\bar K_3q_*\lambda^{1/2}e^{-b c(t-t_1)/2}
    -\sigma\bar K_3q_*\lambda^{1/2}e^{-\omega (t-t_1)}\\
    \leq &0
  \end{align*}
  for all $t\geq1$ and $x\in\p K$  such that $\underline u(t,x)\geq(1-\lambda)\cdot\bm1$ and
   for each $i=1,\cdots,m$.
  This completes the proof of claim \eqref{clm}.

 Fix any $i=1,\cdots,m$. We now prove
  $
  \mathscr L_i[\underline u](t,x)=(\underline u_i)_t(t,x)-D_i\Delta \underline u_i(t,x)-F_i(\underline u(t,x))\leq0
  $
  for  $t\geq1$ and $x\in\bar\Omega$ such that $\underline u_i(t,x)\geq0$.
  By direct calculations, one has
  \begin{align*}
  \mathscr L_i[\underline u](t,x)
  =&F_i(\Phi(\xi(t,x)))-F_i(\underline u(t,x))+V'(t)\Phi'(\xi(t,x))-q_i(\chi(t,x))v'(t)\\
  &-cv(t)q_i'(\chi(t,x))+\sigma D_iv(t)q_*\Delta\zeta(x)+D_iv(t)q_i''(\chi(t,x)) -\sigma v'(t) q_*\zeta(x)\\
  &+ \beta t^{-\alpha-1}e^{-\frac{|x'|^2}{\gamma t}}\left(\alpha-2D_i(N-1)\gamma^{-1}\right)
  \left(\Phi'_i(\xi(t,x))-v(t)q_i'(\chi(t,x))\right)\\
  &+ \beta \gamma^{-1}t^{-\alpha-2}e^{-\frac{|x'|^2}{\gamma t}}|x'|^2\left(4D_i\gamma^{-1}-1\right)
  \left(\Phi'_i(\xi(t,x))-v(t)q_i'(\chi(t,x))\right)\\
  &-4D_i \beta^2 \gamma^{-2}t^{-2\alpha-2}|x'|^2 e^{-\frac{2|x'|^2}{\gamma t}}
  \left( \Phi_i''(\xi(t,x))-v(t)q_i''(\chi(t,x))\right).
  \end{align*}
By \eqref{2C.b} and \eqref{ag}, one gets that $\alpha-2D_i(N-1)\gamma^{-1}\leq0$ and $4D_i\gamma^{-1}-1\leq0$.

  If $\xi(t,x)\geq C$, it then follows from $V\leq1$ and $C\geq2$ that $\chi(t,x)\geq C-V(t)\geq1$. By \eqref{2C.a}, one has $Q(\chi(t,x))=Q_1$.
 From \eqref{2C.2}, \eqref{C}, \eqref{sss}, \eqref{muuu} and \eqref{mv}, there holds
  $$
  \bm1\geq\Phi(\xi(t,x))\geq\underline u(t,x)\geq (1-\vp_0)\cdot\bm1-\mu q^*\cdot\bm1-\sigma\mu q_*\|\zeta\|_{L^\i(\bar\Omega)}\cdot\bm1
  \geq(1-4\vp_0)\cdot\bm1.
  $$
  By \eqref{2C.2}, \eqref{2C.3}, \eqref{2C.7} and \eqref{sss},  one gets that
  \begin{align*}
  F_i(\Phi(\xi(t,x)))-F_i(\underline u(t,x))
  \leq&\sum_{j=1}^m\mu_{ij}^1v(t)q_j^1+\sigma \Lambda v(t)q_*\|\zeta\|_{L^\i(\bar\Omega)}\\
  \leq&-\varpi q_i^1v(t)+\varpi q_*v(t)/2 \leq-\varpi q_*v(t)/2.
  \end{align*}
  Since $\Phi_i'>0$, $V'<0$ and $t\geq1$, it follows from \eqref{impo}, \eqref{2C.2}, \eqref{2C.b}, \eqref{bk1},
  \eqref{ome}, \eqref{sss}, \eqref{bt1},  \eqref{ag}
  and \eqref{vo} that
  \begin{align*}
  \mathscr L_i[\underline u](t,x)
  \leq
  &\beta \gamma^{-1}\e(4D_i\gamma^{-1}-1+4D_i\beta\gamma^{-1}\bar K_1\r)t^{-\alpha-2}e^{-\frac{|x'|^2}{\gamma t}}|x'|^2\Phi_i'(\xi(t,x))\\
  &+2\omega v(t)q_i^1-\varpi q_*v(t)/2+2\sigma\omega v(t)q_*\zeta(x)+\sigma\bar D q_*v(t)\Delta\zeta(x)\\
   \leq
  &v(t)\left(2\omega q^*+2\sigma\omega q_*\|\zeta\|_{L^\i(\bar\Omega)}+\sigma\bar D q_*\|\Delta\zeta\|_{L^\i(\bar\Omega)}-\varpi q_*/2 \right)\\
  \leq& 0.
  \end{align*}

  If $\xi(t,x)\leq- C$, it then follows from  $ V\geq0$ that   $\chi(t,x)\leq- C- V(t)\leq0$. By \eqref{2C.a}, one has $Q(\xi(t,x))=Q_0$.
 By  similar arguments as above, one obtains that
  $\mathscr L_i[\underline u](t,x)\leq0$.

  If $-C\leq\xi(t,x)\leq C$, it then follows from \eqref{2C.2}, \eqref{2C.7} and \eqref{zeta}  that
  \begin{align*}
  F_i(\Phi(\xi(t,x)))-F_i(\underline u(t,x))
  \leq\Lambda v(t)q_i(\chi(t,x))+\sigma\Lambda v(t)q_*\zeta(x)
  \leq2\Lambda v(t)q^*\|\zeta\|_{L^\i(\bar\Omega)} .
  \end{align*}
  Since $\Phi_i'>0$, $V'<0$ and $t\geq1$, it follows from \eqref{impo}, \eqref{2C.2}, \eqref{2C.b}, \eqref{zeta},
  \eqref{tkp}, \eqref{bk1}, \eqref{rr}, \eqref{muuu}, \eqref{bt1}, \eqref{ag},  \eqref{mv}, \eqref{vo} and \eqref{V}
   that
  \begin{align*}
  \mathscr L_i[\underline u](t,x)
  \leq
  &\beta \gamma^{-1}t^{-\alpha-2}e^{-\frac{|x'|^2}{\gamma t}}|x'|^2
  \left(4\bar D\gamma^{-1}\left(\|\Phi'_i\|_{L^\i(\R)}+\beta\|\Phi''_i\|_{L^\i(\R)}+M\beta\right)
  -\tilde\kappa\right)\\
  &+2\Lambda v(t)q^*\|\zeta\|_{L^\i(\bar\Omega)}+\tilde\kappa V'(t)+2\omega v(t)q_i(\chi(t,x))+cv(t)M \\
  &+\bar Dv(t)M+2\sigma\omega v(t)q_*\|\zeta\|_{L^\i(\bar\Omega)}+\sigma\bar Dq_*v(t)\|\Delta\zeta\|_{L^\i(\bar\Omega)}\\
  \leq
  &\tilde\kappa V'(t)+v(t)\rho\\
  \leq& 0.
  \end{align*}

As a result,  $\underline u$ is a subsolution of the equation satisfied by $u(t_\lambda-1+\cdot,\cdot)$ in $[1,+\i)\times\bar\Omega $. The proof is complete.
\end{pr}

\begin{lem}\label{sbs}
The function $\bar u=(\bar u_1,\cdots,\bar u_m)$ defined by
$$
\bar u(t,x)=\min\left(\Phi(\tilde\xi(t,x))+\tilde v(t) Q(\tilde\chi(t,x)) +\sigma\tilde v(t) q_*\zeta(x)\cdot\bm1,\bm1\right)
$$
is a supersolution of the problem satisfied by
 $u(t_{\lambda}-1+t,x)$ for $t\geq1$ and $x\in\bar\Omega$,
where
$\tilde\xi(t,x)=x_1+c(t-1+t_{\lambda})+\beta t^{-\alpha}e^{-\frac{|x'|^2}{\gamma t}}-\tilde V(t)+\tilde V(1)$ and
$\tilde\chi(t,x)=\tilde\xi(t,x)+\tilde V(t)$.
\end{lem}
\begin{pr}
Let us first check the initial conditions.
For $x\in\bar\Omega\cap\bar{B(\bm\theta,\bar R_{\lambda})}$,
since $\Phi'\gg\bm0$ and $\gamma>1$, one gets from \eqref{ul2} that
\begin{align*}
\bar u(1,x)&\geq\Phi\left(x_1+ct_{\lambda}+\beta e^{-{|x'|^2}/{\gamma}}\right)
\geq u(t_{\lambda},x).
\end{align*}
For $x\in\bar {\Omega\backslash B(\bm\theta,\bar R_{\lambda})}$,
since $\Phi'\gg\bm0$ and $u\ll\bm1$, it follows from \eqref{2C.2}, \eqref{lamm} and \eqref{rl} that
$$
\bar u(1,x)\geq\min\left(\Phi(x_1+c t_{\lambda})+\mu q_*\cdot\bm1,\bm1\right)
\geq\min\left(\Phi(x_1+ct_{\lambda})+\lambda q_*\cdot\bm1,\bm1\right) \geq u(t_{\lambda},x).
$$
As a consequence, one obtains
$\bar u(1,x)\geq u(t_{\lambda},x)$  for all $x\in\bar\Omega$.

We then verify the boundary conditions.
For all $t\geq1$, we define
$$
k(t)=-L+c(t-1+t_{\lambda})+\beta t^{-\alpha}e^{-\frac{(N-1)L^2}{\gamma t}}.
$$
By \eqref{2C.b}, \eqref{bk1} and \eqref{ag}, one has
$$
k'(t)\geq c-\alpha\beta t^{-\alpha-1}e^{-\frac{(N-1)L^2}{\gamma t}}
\geq c-\alpha\beta\geq c-\frac{N-1}{4\bar K_1}
\geq\frac{c}{2}\ \ \text{ for all }t\geq1.
$$
For all $t\geq1$ and $x\in\p \Omega$, it follows from \eqref{ck}, \eqref{fs1}, $\gamma\geq1$ and $t_{\lambda}\geq(L+1)/c$ that
\begin{align*}
\tilde\xi(t,x)\geq k(t)\geq k(1)+\frac{c}{2}(t-1)
\geq-L+c t_{\lambda}+\beta e^{-(N-1)L^2}+\frac{c}{2}(t-1)\geq-L+c t_{\lambda} \geq1.
\end{align*}
On the other hand,
 one infers from $\tilde V(t)>0$  that
$
\tilde\chi(t,x)\geq\tilde\xi(t,x)\geq1$,
hence $Q(\tilde\xi(t,x))=Q_1$ by \eqref{2C.a}.
Since $\nu\cdot\nabla\zeta=1$ on $\p\Omega$,  it can be deduced from \eqref{1.a}, \eqref{ck}, \eqref{bk1}, \eqref{ome}, \eqref{bt1}, \eqref{ag}
and the definition of $\tilde v(t)$ that

\begin{align*}
\nu\cdot\nabla\bar u_i(t,x)
&\geq\left|\nabla\P_i(\tilde\xi(t,x))\right|+\sigma\tilde v(t) q_*\\
&\geq\Phi_i'(\tilde\xi(t,x))\left(1+4(N-1)\beta^2t^{-2\alpha-2}e^{-\frac{2|x'|^2}{\gamma t}}\gamma^{-2}|x'|^2\right)^{1/2}
+\sigma\tilde v(t) q_*\\
&\geq -a\left(1+4(N-1)L^2\beta^2\gamma^{-2}\right)^{1/2}e^{-b(-L+c t_{\lambda}+\beta e^{-(N-1)L^2}+\frac{c}{2}(t-1))}
+\sigma\tilde v(t) q_*\\
&\geq-a\left(1+\frac{(N-1)L^2}{16\bar K_1^2\bar D^2}\right)^{1/2}e^{-b(-L+c t_{\lambda}+\beta e^{-(N-1)L^2})-\frac{b c}{2}(t-1)}
+\sigma\tilde v(t) q_*\\
&\geq-a\left(1+\frac{(N-1)L}{4\bar D}\right)e^{-b(-L+c t_{\lambda}+\beta e^{-(N-1)L^2})}e^{-2\omega(t-1)}
+\sigma\mu q_*e^{-2\omega(t-1)}\\
&\geq0
\end{align*}
for $t\geq1$ and $x\in\p\Omega$ and for each $i=1,\cdots,m$.

Fix any $i=1,\cdots,m$, it suffices to prove that
$
\mathscr L_i[\bar u](t,x)
=(\bar u_i)_t(t,x)-D_i\Delta \bar u_i(t,x)-F_i(\bar u(t,x))\geq0
$
for $t\geq1$ and $x\in\bar\Omega$ such that $\bar u_i(t,x)<1$.
By direct calculations, one has
\begin{align*}
\mathscr L_i[\bar u](t,x)
=&F_i(\Phi(\tilde\xi(t,x)))-F_i(\bar u(t,x))-\tilde V'(t)\Phi_i'(\tilde\xi(t,x))-D_i\tilde v(t)q_i''(\tilde\chi(t,x))\\
&+\sigma q_*\tilde v'(t)\zeta(x)-\sigma D_iq_*\tilde v(t)\Delta\zeta(x)
+c\tilde v(t)q_i'(\tilde\chi(t,x))+\tilde v'(t)q_i(\tilde\chi(t,x))\\
&
+\beta t^{-\alpha-1}e^{-\frac{|x'|^2}{\gamma t}}\left(2D_i(N-1)\gamma^{-1}-\alpha\right)
\left(\Phi_i'(\tilde\xi(t,x))+\tilde v(t)q_i'(\tilde\chi(t,x))\right)\\
&+\beta\gamma^{-1}t^{-\alpha-2}e^{-\frac{|x'|^2}{\gamma t}}|x'|^2\left(1-4D_i\gamma^{-1}\right)
\left(\Phi_i'(\tilde\xi(t,x))+\tilde v(t)q_i'(\tilde\chi(t,x))\right)\\
&-4D_i\beta^2\gamma^{-2}t^{-2\alpha-2}e^{-\frac{2|x'|^2}{\gamma t}}|x'|^2
\left(\Phi_i''(\tilde\xi(t,x))+\tilde v(t)q_i''(\tilde\xi(t,x))\right)
.
\end{align*}
By \eqref{2C.b} and \eqref{ag},   $2D_i(N-1)\gamma^{-1}-\alpha\geq0$ and $1-4D_i\gamma^{-1}\geq0$.

If $\tilde\xi(t,x)\geq C$, it then follows from  $\tilde V(t)\geq0$ that
$\tilde\chi(t,x)\geq C>1$.
By \eqref{2C.a}, one has $Q(\tilde\chi(t,x))=Q_1$.
From \eqref{C}, there holds
$
\bm1\geq\bar u(t,x)\geq \Phi(\tilde\xi(t,x))\geq(1-\vp_0)\cdot\bm1$.
Then, one infers from \eqref{2C.2}, \eqref{2C.3}, \eqref{2C.7} and \eqref{sss} that
\begin{align*}
F_i(\Phi(\tilde\xi(t,x)))-F_i(\bar u(t,x))
\geq&-\sum_{j=1}^m\mu_{ij}^1\tilde v(t)q_j^1 -\sigma\Lambda q_*\tilde v(t)\zeta(x)\\
\geq&\varpi\tilde v(t)q_i^1 -\sigma\Lambda q_*\tilde v(t)\|\zeta\|_{L^\i(\bar\Omega)}\geq
\frac{1}{2}\varpi\tilde v(t)q_* .
\end{align*}
Together with $\Phi_i'>0$, $\tilde V'(t)<0$, $t\geq1$, one gets from  \eqref{impo}, \eqref{2C.b}, \eqref{ome}, \eqref{ag}
and \eqref{fs0} that
\begin{align*}
\mathscr L_i[\bar u](t,x)
\geq&\tilde v(t)\left(\frac{1}{2}\varpi q_*-2\omega q^*-2\sigma \omega q_*\|\zeta\|_{L^\i(\bar\Omega)} -\sigma \bar D q_*\|\Delta\zeta\|_{L^\i(\bar\Omega)}\right)
\geq0.
\end{align*}

If $\tilde\xi(t,x)\leq -C$, then
$\tilde\chi(t,x)\leq- C+\tilde V(t)\leq0$ by \eqref{fs2} and $C>2$.
By \eqref{2C.a},  $Q(\tilde\chi(t,x))=Q_0$.
Similarly,
$\mathscr L_i[\bar u](t,x)\geq0$.

If $-C\leq\tilde\xi(t,x)\leq C$, it then follows from \eqref{2C.2}, \eqref{2C.7} and \eqref{zeta} that
$$
F_i(\Phi(\tilde\xi(t,x)))-F_i(\bar u(t,x))
\geq-2\Lambda\tilde v(t)q^*\|\zeta\|_{L^\i(\bar\Omega)}.
$$
Since $\Phi_i'>0$, $\tilde V'(t)<0$, $t\geq1$, it follows from \eqref{impo}, \eqref{2C.2}, \eqref{2C.b}, \eqref{tkp},
\eqref{rr}, \eqref{muuu}, \eqref{ag}, \eqref{fs0} and \eqref{fs1}
that
\begin{align*}
\mathscr L_i[\bar u](t,x)
\geq&-2\Lambda\tilde v(t)q^*\|\zeta\|_{L^\i(\bar\Omega)}-\tilde V'(t)\tilde\kappa
-2\sigma\omega q_*\tilde v(t)\zeta(x)-\bar D\tilde v(t)M-\sigma \bar Dq_*\tilde v(t)\|\Delta\zeta\|_{L^\i(\bar\Omega)}\\
&-c\tilde v(t)M -2\omega \tilde v(t)q_i(\chi(t,x))
+\beta t^{-\alpha-1}e^{-\frac{|x'|^2}{\gamma t}}\left(2D_i(N-1)\gamma^{-1}-\alpha\right)
\left(\tilde\kappa-\mu M\right)\\
&+\beta\gamma^{-1}t^{-\alpha-2}e^{-\frac{|x'|^2}{\gamma t}}|x'|^2
\left(\tilde\kappa-\mu M -4\bar D\gamma^{-1}\left(\|\Phi'_i\|_{L^\i(\R)}+\beta M+\beta\|\Phi''_i\|_{L^\i(\R)}\right)\right)\\
\geq&-\tilde V'(t)\tilde\kappa-\rho\tilde v(t)\\
\geq&\ 0.
\end{align*}

As a result,
$\bar u$ is a supersolution of the problem satisfied by $u(t_{\lambda}-1+\cdot,\cdot)$ in $[1,+\i)\times\bar\Omega$. The proof is complete.
\end{pr}
\vspace{0.2cm}

\begin{pr}[Proof of Theorem \ref{t2}]
We divide the proof into two steps.

{\it Step 1: proof of \eqref{l+}}.
By Lemma \ref{ups} and the comparison principle, one concludes that
$\underline u(t,x)\leq u(t+t_{\lambda}-1,x)$ for all
$t\geq1\text{ and }x\in\bar\Omega$.
By \eqref{2C.2}, for each $i=1,\cdots,m$ and for all $t\geq t_{\lambda}$,
\begin{align*}
&\inf_{x\in\bar \Omega}\left(u_i(t,x)-\Phi_i(x_1+c t)\right)\\
\geq&\inf_{x\in\bar \Omega}\left(\underline u_i(t+1-t_{\lambda},x)-\Phi_i(x_1+c t)\right)\\
\geq&\inf_{x\in\bar \Omega}\left(\Phi_i(\xi(t+1-t_{\lambda},x))
-v(t+1-t_{\lambda})\e(q^*+\sigma q_*\|\zeta\|_{L^\i(\bar\Omega)}\r)
-\Phi_i(x_1+c t)\right)\\
\geq&-\left(\beta(t+1-t_{\lambda})^{-\alpha}+V(1)-V(t+1-t_{\lambda})\right)
\|\Phi_i'\|_{L^\i(\R)}\\
&-v(t+1-t_{\lambda})\e(q^*+\sigma q_*\|\zeta\|_{L^\i(\bar\Omega)}\r).
\end{align*}
By \eqref{V1}, one gets that the right-hand side converges to
$-V(1)\|\Phi_i'\|_{L^\i(\R)}\geq-\vp$ as $t\to+\i$.
Since $\vp>0$ is arbitrary, one has
$$
\liminf_{t\to+\i}\left(\inf_{x\in\bar \Omega}\left(u(t,x)-\Phi(x_1+c t)\right)\right)\geq\bm0.
$$

On the other hand, by Lemma \ref{sbs} and the comparison  principle, one has
$\bar u(t,x)\geq u(t-1+t_{\lambda})$ for all $t\geq1$ and $x\in\bar\Omega$.
By \eqref{2C.2}, for each $i=1,\cdots,m$ and  for all $t\geq t_{\lambda}$,
\begin{align*}
&\sup_{x\in\bar\Omega}\left(u_i(t,x)-\Phi_i(x_1+c t)\right) \\
\leq&\sup_{x\in\bar\Omega}\left(\underline u_i(t+1-t_{\lambda},x)-\Phi_i(x_1+c t)\right) \\
\leq&\sup_{x\in\bar\Omega}\left(\Phi_i(\tilde\xi(t+1-t_{\lambda},x))
+\tilde v(t+1-t_{\lambda})\left(q^*+\sigma q_*\|\zeta\|_{L^\i(\bar\Omega)}\right)-\Phi_i(x_1+c t)\right) \\
\leq&\left(\beta(t+1-t_{\lambda})^{-\alpha}+\tilde V(1)-\tilde V(t+1-t_{\lambda})\right)\|\Phi_i'\|_{L^\i(\R)}\\
&+\tilde v(t+1-t_{\lambda})\left(q^*+\sigma q_*\|\zeta\|_{L^\i(\bar\Omega)}\right).
\end{align*}
 By \eqref{fs2},  one gets that the right-hand side converges to $\tilde V(1)\|\Phi_i'\|_{L^\i(\R)}\leq\vp$ as $t\to+\i$.
Since $\vp>0$ is arbitrary, one has
$$
\limsup_{t\to+\i}\left(\sup_{x\in\bar \Omega}\left(u(t,x)-\Phi(x_1+c t)\right)\right)\leq\bm0.
$$
As a conclusion, the formula \eqref{l+} follows.

{\it Step 2: proof of \eqref{UP}.}
Choose
$$
  \rho_1=\tilde\kappa^{-1}\left((2cM+\bar D M+\Lambda q^*+q*)\|\zeta\|_{L^\i(\bar\Omega)}+2\bar DM\|\nabla\zeta\|_{L^\i(\bar\Omega)}
  +\bar D q^*\|\Delta\zeta\|_{L^\i(\bar\Omega)}\right).
$$
Let $\vp$ be an arbitrary positive real number satisfying
$$
0<\vp<\min\e(1,\frac{\vp_0}{q^*\|\zeta\|_{L^\i(\Omega)}},\frac{\varpi}{2},\frac{c}{\rho_1}\r).
$$
By \eqref{1.5} and \eqref{l+}, there exist $t_\vp>0$ such that
\begin{align}\label{upp}
-\vp q_*\cdot\bm1\leq u(t,x)-\Phi(x_1+c t)\leq\vp q_*\cdot\bm1\ \text{ for all }|t|\geq t_\vp\text{ and }x\in\bar\Omega.
\end{align}
Let $\rho_2=L+c t_\vp+1$.
For $t\geq -t_\vp$ and $x\in\bar\Omega$, we define functions $u^\pm=(u^\pm_1,\cdots,u^\pm_m)$ by
$$\bc
u^+(t,x)=\min\left(\Phi(\xi^+(t,x))+\vp Q(\xi^+(t,x))\zeta(x)e^{-\vp (t+t_\vp)},\bm1\right),\\
u^-(t,x)=\max\left(\Phi(\xi^-(t,x))-\vp Q(\xi^-(t,x))\zeta(x)e^{-\vp (t+t_\vp)},\bm0\right),
\ec$$
where $\xi^\pm(t,x)=x_1+c t\pm\rho_1(1-e^{-\vp (t+t_\vp)})\pm\rho_2$.
We shall prove that the function $u^+$ is a supersolution of \eqref{1.1} in $[-t_\vp,+\i)\times\bar\Omega$.
The fact that $\underline u$ is a subsolution can be obtained by  similar arguments.
Let us first verify the initial and boundary conditions.
By \eqref{2C.2}, \eqref{zeta}, \eqref{upp} and $\Phi'\gg\bm0$, there holds
$$
u^+(-t_\vp,x)\geq\min\left(\Phi(x_1-c t_\vp)+\vp q_*\cdot\bm1,\bm1\right)\geq u(-t_\vp,x)\ \text{ for }x\in\bar\Omega.
$$
For $t\geq-t_\vp$ and $x\in\p\Omega$,
it follows from \eqref{ck}  that $\xi^+(t,x)\geq -L-c t_\vp+\rho_2=1$,
hence  $Q(\xi^+(t,x))=Q_1$ by \eqref{2C.a}.
Since $\nu\cdot\nabla\zeta=1$, one infers that for each $i=1,\cdots,m$,
$\nu\cdot\nabla u_i^+(t,x)=\vp Q_1 e^{-\vp(t+t_\vp)}\geq0$ for all $t\geq-t_\vp$ and $x\in\p\Omega$
such that $u_i^+(t,x)<1$.
 By similar arguments as in Lemma \ref{w-},
 one can obtain that
 $\mathscr L_i[u^+](t,x)=(u_i^+)_t(t,x)-D_i\Delta u_i^+(t,x)-F_i(u^+(t,x))\geq0$
 for $t\geq-t_\vp$ and $x\in\bar\Omega$ such that
 $u^+_i(t,x)<1$ and for each $i=1,\cdots,m$.

Since $\Phi(-\i)=\bm0$ and $\Phi(+\i)=\bm1$, it follows from the comparison principle  that
 there exist $K_1>0$ and $K_2>0$ large enough such that
$$
u(t,x)\geq  u^-(t,x)\geq \min(\Phi(x_1+c t-\rho_1-\rho_2)-\vp q_*\cdot\bm1,\bm1)\geq(1- 2\vp q_*)\cdot\bm1
$$
for all $t\geq-t_\vp$ and $x\in\bar\Omega$ such that $x_1+c t\geq K_1$, and
$$
u(t,x)\leq u^+(t,x)\leq \max(\Phi(x_1+c t+\rho_1+\rho_2)+\vp q_*\cdot\bm1,\bm0)\leq 2\vp q_*\cdot\bm1
$$
for all $t\geq-t_\vp$ and $x\in\bar\Omega$ such that $x_1+c t\leq- K_2$.
Therefore, by similar arguments as in Lemma \ref{up2},
one can get that there exists $\bar R_\vp>0$ large enough such that
\begin{align}\label{vppp}
-2\vp q_*\cdot\bm1\leq u(t,x)-\Phi(x_1+c t)\leq2\vp q_*\cdot\bm1
\ \text{ for }t\geq-t_\vp \text{ and }x\in\bar\Omega\text{ such that }|x|\geq\bar R_\vp.
\end{align}
Since $\vp$ is arbitrary,
one deduces from \eqref{upp} and \eqref{vppp} that \eqref{UP} holds.
That completes the proof of Theorem \ref{t2}.
\end{pr}

\section{Some geometrical conditions for complete propagation}\label{s5}
This section is concerned with sufficient conditions to ensure the complete propagation.
For this purpose, we prove two Liouville-type results of stationary solutions $u_\i$ of \eqref{ui} by sliding methods, that is, $u_\i\equiv\bm1$ in $\bar\Omega$, when the obstacle $K$ is star-shaped or directionally convex.
We first show  the following auxiliary lemma.
\begin{lem}\label{Uxi}
There exists a function $U=(U_1,\cdots,U_m)\in C^2([0,+\i),[\bm0,\bm1])$ such that
\begin{align}\label{Ueq}
\bc
DU''(\xi)+F(U(\xi))=\bm0,&\ \xi>0,\\
U'(\xi)\gg\bm0,&\ \xi\geq0,\\
U(0)=\bm0,\ U(+\i)=\bm1.
\ec
\end{align}
\end{lem}
\begin{pr}
 We divide the proof into three  steps.

{\it Step 1: choice of parameters and notations.}
Let $q_*$, $q^*$, $\vp_0$, $\varpi$, $\Lambda$, $\bar D$, $\underline D$ and $L$ be
positive constants defined as in the
beginning of section \ref{s2}.
Fix an arbitrary constant $\dl\in(0,\vp_0/q{^*})$.
By the properties of $\Phi$, there exists a constant $C>1$ such that
\begin{align}\label{defc}
\Phi(\xi)\leq \dl q_*\cdot\bm1\ \text{ for }\xi\leq-C \ \text{ and } \
\Phi(\xi)\geq (1-\dl q_*)\cdot\bm1 \ \text{ for }\xi\geq C.
\end{align}
Furthermore, there exists $\kappa>0$ such that
 \begin{align}\label{kap}
\kappa=\min_{i=1,\cdots,m}\left(\inf_{\xi\in[-C,C]}\P_i'(\xi)\right).
\end{align}
Denote
 \begin{align}\label{d1}
\dl_1=\min\left(\dl,\frac{1}{q^*},\frac{c\kappa}{\e(2\Lambda q^*+\bar D M\r)}\right)
\ \text{ and }
\ \dl_2=\min\left(\frac{\dl_1 \varpi q_*}{\Lambda q^*},\frac{\dl_1}{2},\frac{\dl_1q_*}{q^*},
\frac{\varpi\dl_1 q_*}{2(\bar D+2\Lambda q_*)}\r)
.
\end{align}
Since $\Phi(-\i)=\bm0$ and $\Phi(+\i)=\bm1$,
there exists $S_1\geq C$ such that
\begin{align}\label{S1}
\Phi(\xi)\leq \dl_2 q_*\cdot\bm1\ \text{ for }\xi\leq-S_1 \ \text{ and } \
\Phi(\xi)\geq (1-\dl_2 q_*)\cdot\bm1 \ \text{ for }\xi\geq S_1.
\end{align}
Let $\hat h:\R\to[0,1]$ be a nondecreasing $C^2$ function such that
 \begin{align}\label{hh}
\hat h(s)=0\ \text{ for }s\leq0\ \ \text{ and }\ \
\hat h(s)=1\ \text{ for }s\geq S_2,
\end{align}
where $S_2>0$ can be chosen sufficiently large such that
\begin{align}\label{S2}
\|\hat h''\|_{L^\i(\R)}+2\|\hat h'\|_{L^\i(\R)}\max_{i=1,\cdots,m}\|\Phi_i'\|_{L^\i(\R)}\leq \dl_2.
\end{align}
Take $S_0=S_1+S_2+C$, $S\geq S_0$ and $h(s)=\hat h(s-S+S_0)$.

{\it Step 2: construction of a subsolution.} For all $\xi\geq0$,
we define the function $\underline U=(\underline U_1,\cdots,\underline U_m)$ as
\begin{align*}
\underline U(\xi)=\max\e(h(\xi)\Phi(\xi-S+C)+\dl_2 Q_0-\dl_1 Q(\xi-S+C),\bm0\r).
\end{align*}
We shall prove that the function $\underline U$ is a subsolution of the elliptic boundary problem
\begin{align}\label{Uueq}
DU''(\xi)+F(U(\xi))=\bm0,\ \
U(0)=\bm0\ \text{ and  }\ U(+\i)=\bm1+\dl_2Q_0-\dl_1Q_1(\leq\bm1)
\end{align}
on $[0,+\i)$.
 Note that $\bm0\leq\underline U(\xi)\leq\bm1+\dl_2 q^*-\dl_1q_*\leq\bm1$ for all $\xi\geq0$
by \eqref{2C.2} and \eqref{d1}.
 We first verify the boundary conditions.
 On the one hand, since $S_0-S\leq0$ and $-S+C\leq -S_0+C\leq -(S_1+S_2)\leq0$,
  it  follows from \eqref{2C.a}, \eqref{d1} and \eqref{hh} that
$$
\underline U(0)=\max\e(\dl_2 Q_0-\dl_1 Q_0,\bm0\r)=\bm0.
$$
On the other hand, by \eqref{2C.a}, \eqref{d1} and $\P(+\i)=\bm1$, one gets that
$$
\underline U(+\i)=\max\e(\bm1+\dl_2 Q_0-\dl_1 Q_1,\bm0\r)=\bm1+\dl_2 Q_0-\dl_1 Q_1.
$$

Fix any $i=1,\cdots,m$. It is suffices to check that
$
\mathscr L_i[\underline U](\xi)=D_i\underline U_i''(\xi)+F_i(\underline U(\xi))\geq0
$
for all $\xi\geq0$ such that $\underline U_i(\xi)>0$.
If $0\leq\xi\leq S-S_0$, then $\xi-S+S_0\leq0$ and $\xi-S+C\leq0$. By \eqref{2C.a}, \eqref{2C.2},  \eqref{d1} and \eqref{hh}, one has
$\bm0\leq\underline U(\xi)=\dl_2 Q_0-\dl_1 Q_0\leq\bm0$, that is, $\underline U(\xi)=\bm0$.

If $S-S_0\leq\xi\leq S-S_0+S_2$, then $0\leq\xi-S+S_0\leq S_2$ and
$\xi-S+C\leq-S_0+S_2+C=-S_1<0$.
From \eqref{2C.a} and  \eqref{hh}, one gets that $h(\xi)\in[0,1]$ and  $\underline U(\xi)=h(\xi)\Phi(\xi-S+C)+\dl_2 Q_0-\dl_1 Q_0$.
By $\Phi_i<1$  and \eqref{2C.b}, there holds
\begin{align*}
\mathscr L_i[\underline U](\xi)
=&D_i h''(\xi)\Phi_i(\xi-S+C)+2D_i h'(\xi)\Phi_i'(\xi-S+C)\\
 &+D_ih(\xi)\Phi_i''(\xi-S+C)+F_i(\underline U(\xi))\\
\geq& -\bar D \left(\|\hat h''\|_{L^\i(\R)}+2\|\hat h'\|_{L^\i(\R)}\max_{i=1,\cdots,m}\|\Phi_i'\|_{L^\i(\R)}\right)\\
&+c  h(\xi)\Phi_i'(\xi-S+C)-  h(\xi)F_i(\Phi(\xi-S+C))+F_i(\underline U(\xi)).
\end{align*}
 By \eqref{2C.7}, \eqref{S1}, $F_i(\bm0)=0$ and the mean value theorem, there exist $\tau\in(0,1)$ such that
  \begin{align*}
F_i(\Phi(\xi-S+C))
=\sum_{j=1}^m\frac{\p F_i}{\p u_j}\left(\tau\Phi(\xi-S+C)\right)\Phi_j(\xi-S+C)
\leq\dl_2\Lambda q_*.
\end{align*}
It follows from \eqref{2C.2}, \eqref{d1}, \eqref{S1} and $h\leq1$ that
$
\bm0\leq \underline U(\xi)\leq\dl_2 q_*\cdot\bm1+\dl_2 q^*\cdot\bm1\leq2\dl q^*\cdot\bm1\leq 4\vp_0\cdot\bm1$.
Then, by $F_i(\bm0)=0$, $0\leq h\leq1$, \eqref{2C.3},   \eqref{2C.7}, \eqref{d1}, \eqref{S1}  and the mean value theorem  that there exists $\tau\in(0,1)$ such that
  \begin{align*}
F_i(\underline U(\xi))
=&\sum_{j=1}^m\frac{\p F_i}{\p u_j}\left(\tau\e(h(\xi)\Phi(\xi-S+C)+\dl_2 Q_0-\dl_1 Q_0\r)\right)\e(h(\xi)\Phi_j(\xi-S+C)-(\dl_1-\dl_2) q_j^0\r)\\
\geq&-\dl_2 \Lambda q_*-\sum_{j=1}^m\mu_{ij}^0(\dl_1-\dl_2)q_j^0\\
\geq&-\dl_2 \Lambda q_*+\frac{1}{2}\varpi\dl_1 q_*.
\end{align*}
Together with $\Phi_i'>0$, $0\leq h\leq1$, \eqref{d1} and \eqref{S2} that
\begin{align*}
\mathscr L_i[\underline U](\xi)
\geq& -\dl_2\bar D -\dl_2\Lambda q_*
-\dl_2 \Lambda q_*+\frac{1}{2}\varpi\dl q_*\geq0.
\end{align*}

If $S-S_0+S_2\leq\xi\leq S-2C$, then $\xi-S+S_0\geq S_2$ and $\xi-S+C\leq -C$.
By \eqref{2C.a}, \eqref{S1} and  \eqref{hh}, there holds
  $$
  \bm0\leq\Phi(\xi-S+C)\leq\underline U(\xi)=\Phi(\xi-S+C)+\dl_2 Q_0-\dl_1 Q_0\leq \dl_2 q_*\cdot\bm1+\dl_2 q^*\cdot\bm1\leq 4\vp_0\cdot\bm1.
  $$
Since $\Phi_i'>0$, it follows from  \eqref{2C.3} and  \eqref{d1} that
\begin{align*}
\mathscr L_i[\underline U]
=& c\Phi_i'(\xi-S+C)
-F_i(\Phi(\xi-S+C))+F_i(\underline U(\xi))
\geq\varpi\e(\dl_1-\dl_2\r) q_i^0
\geq0.
\end{align*}

If $S-2C\leq\xi\leq S$, then $\xi-S+S_0\geq S_1+S_2-C\geq S_2 $ and $-C\leq\xi-S+C\leq C$.
By \eqref{2C.2}, \eqref{d1} and \eqref{hh}, one has
$$
\bm0\leq\underline U(\xi)=\Phi(\xi-S+C)+\dl_2 Q_0-\dl_1 Q(\xi-S+C)\leq\bm1+\dl_2 q^*\cdot\bm1-\dl_1 q_*\cdot\bm1\leq\bm1.
$$
From \eqref{2C.2}, \eqref{2C.7}, \eqref{2C.b}, \eqref{kap}  and \eqref{d1},
 there holds
\begin{align*}
\mathscr L_i[\underline U](\xi)
=&\ c\Phi_i'(\xi-S+C)-D_i\dl_1q_i''(\xi-S+C)
-F_i(\Phi(\xi-S+C))+F_i(\underline U(\xi))\\
\geq&\ c\kappa-\bar D\dl_1 M-\dl_2\Lambda q^*-\dl_1\Lambda q^*\\
\geq&\ 0.
\end{align*}

If $\xi\geq S$, then $\xi-S+S_0\geq S_0\geq S_2 $ and $\xi-S+C\geq C\geq1$. By \eqref{2C.a}, \eqref{defc} and \eqref{hh},  there holds $$
(1-4\vp_0)\cdot\bm1\leq(1-\dl q_*)\cdot\bm1\leq\underline U(\xi)=\Phi(\xi-S+C)+\dl_2 Q_0-\dl_1 Q_1\leq\bm1+\dl q^*\cdot\bm1\leq\bm1+4\vp_0\cdot\bm1.
$$
By $\Phi_i'>0$, \eqref{2C.3}, \eqref{2C.7}, \eqref{d1} and the mean value theorem,
there exists $\tau\in(0,1)$ such that
\begin{align*}
\mathscr L_i[\underline U]
=&\ c\Phi_i'(\xi-S+C)
-F_i(\Phi(\xi-S+C))+F_i(\underline U(\xi))\\
\geq&\ \sum_{j=1}^m\frac{\p F_i}{\p u_j}\left(\Phi(\xi-S+C)+\tau\e(\dl_2 Q_0-\dl_1 Q_1\r)\right)\e(\dl_2 q_j^0-\dl_1 q_j^1\r)\\
\geq&\ -\delta_2\Lambda q^*+\dl_1 \varpi q_*\\
\geq&\ 0.
\end{align*}
As a conclusion, the function $\underline U$ is a subsolution of \eqref{Uueq} on $[0,+\i)$.

{\it Step 3: the existence of increasing solution.} Let $U(t,\xi)=(U_1(t,\xi),\cdots,U_m(t,\xi))$ be the classical solution of
the following parabolic initial-boundary problem
 \begin{align}\label{utx}
\bc
U_t(t,\xi)=DU_{\xi\xi}(t,\xi)+F(U(t,\xi)), &t>0,\ \xi>0,\\
U(t,0)=\bm0,\ U(t,+\i)=\bm1+\dl_2 Q_0-\dl_1 Q_1,&t\geq0,\\
U(0,\xi)=\underline U(\xi),& \xi>0.
\ec
\end{align}
In terms of $\bm0\leq\underline U(\xi)\leq\bm1$ for all $\xi\geq0$,
it follows from the comparison principle that $\bm0\leq U(t,\xi)\leq\bm1$ for all $t\geq0$ and $\xi\geq0$. Since $\underline U(\xi)$ is a subsolution of the corresponding elliptic problem of \eqref{utx}, another application of the comparison principle yields that
$U(t,\xi)$ is nondecreasing in $t$ for $\xi\geq0$. By virtue of standard parabolic estimates, there holds
$$
U(t,\xi)\to U(\xi) \ \text{ uniformly in }\xi\in\R \text{ as }t\to+\i,
$$
where $U=(U_1,\cdots,U_m)\in C^2(\mathbb R^+,[\bm0,\bm1])$ is a classical solution of \eqref{Uueq}.

We now turn to prove that $U'(\xi)\gg\bm0$ for all $\xi\geq0$.
 By \eqref{2C.2} and $\Phi'\ll\bm0$, one has $\underline U'(\xi)\geq\bm0$ for $\xi\geq\bm0$.
 Then, for any $\xi_1>\xi_2\geq0$, then holds $\underline U(\xi_1)\geq \underline U(\xi_2)$.
  Applying the comparison principle, one gets that $U(t,\xi_1)\geq U(t,\xi_2)$ for all $t\geq0$,
   that is, $U(t,\xi)$ is nondecreasing in $\xi\geq0$ for all $t\geq0$.
Therefore, $U(\xi)$ is also nondecreasing in $\xi\geq0$.
 Assume to the contrary that there exists $i\in\{1,\cdots,m\}$ and $0\leq\xi_1<\xi_2$ such that $U_i(\xi_1)=U_i(\xi_2)$.
  Let $s=\xi_2-\xi_1>0$ and $W(\xi)=U(\xi+s)-U(\xi)$ for all $\xi\geq0$. Then $W\geq\bm0$ on $[0,+\i)$
  and the $i$-th component $W_i(\xi_1)=U_i(\xi_2)-U_i(\xi_1)=0$.
  According to the maximum principle, there holds $W_i(\xi)\equiv0$ for all $\xi\geq0$,
  it implies that $U_i(\xi+s)=U_i(\xi)$ for all $\xi\geq0$.
  By induction, one has $U_i(\xi+ks)=U_i(\xi)$ for all $\xi\geq0$ and $k\in \mathbb N$. Set $\xi=0$. By passing $k\to+\i$,  one reaches a contradiction with $U(0)=\bm0$ and $U(+\i)=\bm1+\dl_2Q_0-\dl_1Q_1$. Thus, $U'(\xi)\gg\bm0$ for all $\xi\geq0$.
 Since $\dl$ is arbitrary, one concludes that $U$ is a classical solution of \eqref{Ueq}. The proof is thereby complete.
\end{pr}

\begin{pr}[Proof of Theorem \ref{lcomp}]
Let $u(t,x)$ be the entire solution obtained in Theorem \ref{t1}.
We already have shown in Theorem \ref{t1} that
$u(t,x)\to u_\i(x)$ as $t\to+\i$ locally uniformly in $x\in\bar\Omega$ and $u_\i(x)\to\bm1$ as $|x|\to+\i$,
where  $u_\i=(u_\i^1,\cdots,u_\i^m)$ is  classical solution of \eqref{ui}.
It suffices to prove that $u_\i\equiv\bm1$ in $\bar\Omega$ under the conditions of Theorem \ref{lcomp}.
We divided the proof into two steps.

{\it Step 1: star-shape obstacles.} Up to a shift of the origin $\bm\theta$ in $\R^N$, one can assume that the obstacle $K$ (if not empty) is star-shaped with respect to $\bm\theta$.
Take any
\begin{align}\label{dll}
0<\dl< \frac{4\vp_0q_*}{q^*}.
\end{align}
Take $r_0>0$ large enough such that $K\subset B(\bm\theta,r_0)$ and
\begin{align}\label{udl}
u_\i(x)\geq(1-\dl q_*)\cdot\bm1\ \ \text{ for all }|x|\geq r_0.
\end{align}
Let $U$ be the function obtained in Lemma \ref{Uxi}. We claim that
\begin{align}\label{claU}
u_\i(x)\geq U(|x|-r_0)\ \  \text{ for all }|x|\geq r_0.
\end{align}

{\it Proof of claim \eqref{claU}.} Define
\begin{align*}
\vp_*=\inf\left\{\vp>0:u^\vp_\i(x)=u_\i(x)+\vp \dl Q_1\geq U(|x|-r_0)\  \text{ for all }|x|\geq r_0\right\}.
\end{align*}
Clearly, $\vp_*\geq0$.
Assume to the contrary that $\vp_*>0$.
In terms of \eqref{2C.2} and \eqref{udl}, one has
 $u^{\vp}_\i(x)\geq (1-\dl q_*)\cdot\bm1+\vp\dl q_*\cdot\bm1\geq\bm1\geq U(|x|-r_0)$ for $|x|\geq r_0$ if $\vp\geq1$.
 Hence $0<\hat\vp_*<1$.
By the definition of $\vp_*$, there exists an index $s\in\{1,\cdots,m\}$ and sequences  $(\vp_n)_{n\in\mathbb N}\subset\R_+$ and $(x_n)_{n\in\mathbb N}\subset\R^N$ such that
$\vp_n\to\vp_*$ as $n\to+\i$, $|x_n|\geq r_0$ and
\begin{align*}
u^{\vp_n,s}_\i(x_n)=u_\i^s(x_n)+\vp_n \dl q^1_s< U_s(|x_n|-r_0)\  \text{ for all }n\in \mathbb N.
\end{align*}
Since $u_\i(x)\to\bm1$ and $U(|x|-r_0)\to\bm1$ as $|x|\to+\i$,
it follows that $x_n$ is bounded for all $n\in\mathbb N$.
Up to extraction of a subsequence,
one can assume that $x_n\to\bar x$ as $n\to+\i$.
It is easy to see that $|\bar x|\geq r_0$.
Hence,  $u^{\vp_*,s}_\i(\bar x)=u_\i^s(\bar x)+\vp_* \dl q_s^1\leq U_s(|\bar x|-r_0)$. By the definition of $\vp_*$, one has
\begin{align}\label{barx}
u^{\vp_*,s}_\i(\bar x)
= U_s(|\bar x|-r_0).
\end{align}

By \eqref{2C.2}, \eqref{dll}, \eqref{udl} and $\vp_*\leq1$, there holds
$$
(1-4\vp_0)\cdot\bm1 \leq(1-\dl q_*)\cdot\bm1\leq
u^{\vp_*}_\i(x)=u_\i( x)+\vp_* \dl Q_1
\leq(1+\dl q^*)\cdot\bm1\leq(1+4\vp_0)\cdot\bm1
$$
for all $|x|\geq r_0$.
From \eqref{2C.3},  one gets that
\begin{align}\label{uvpi}
D_i \Delta u_\i^{\vp_*,i}(x)+F_i(u_\i^{\vp_*}(x))
=-F_i(u_\i(x))+F_i(u_\i^{\vp_*}(x))
\leq\sum_{j=1}^m\mu_{ij}^1\vp_*\dl q_j^1
\leq-\vp_*\dl \varpi q_i^1
<0
\end{align}
for all $|x|\geq r_0$ and for each $i=1,\cdots,m$.
Define the function $\tilde U=(\tilde U_1,\cdots,\tilde U_m)$ by $\tilde U(x)=U(|x|-r_0)$.
Thanks to Lemma \ref{Uxi},
\begin{align}\label{Uvpi}
D_i\Delta \tilde U_i(x)+F_i(\tilde U(x))
=D_i U_i''(|x|-r_0)+D_i\frac{N-1}{|x|}U_i'(|x|-r_0)+F_i(U(|x|-r_0))>0
\end{align}
 for all $x\in \R^N\backslash B(\bm\theta,r_0)$ and  for each $i=1,\cdots,m$.
 It is easy to see that $U(|\cdot|-r)$ is a strict subsolution of \eqref{1.1} in $\{x\in\mathbb R^N:|x|>\max(r,0)\}$ for all $r\in\R$.
 Define the function $z=(z_1,\cdots,z_m)$ by $z=u_\i^{\vp_*}-\tilde U$ in $\R^N\backslash B(\bm\theta,r_0)$.
By the definition of $\vp_*$,
 one has $z(x)\geq\bm 0$ for $|x|\geq r_0$.
From {(A4)}, \eqref{uvpi}, \eqref{Uvpi} and the mean value theorem, there exists $\tau\in(0,1)$ such that
 \begin{align*}
 D_s\Delta z_s(x)\leq F_s(\tilde U(x))-F_s(u_\i^{\vp^*}(x))
 &=-
 \sum_{j=1}^m \frac{\p F_s}{\p u_j}\e(u_\i^{\vp_*}(x)-\tau z(x)\r)z_j(x)\\
& \leq-\frac{\p F_s}{\p u_s}\e(u_\i^{\vp_*}(x)-\tau z(x)\r)z_s(x)
 \end{align*}
for  $|x|\geq r_0$. It follows from \eqref{barx} that
  $z_s(\bar x)=0$.
  Applying the strong elliptic maximum principle, one has $z_s(x)\equiv 0$ for $|x|\geq r_0$.
   Hence $u_\i^{\vp_*,s}(x)=\tilde U_s(x)$ for all $ |x|\geq r_0$. It is impossible for $|x|=r_0$ since $U_s(0)$=0. The proof of claim \eqref{claU} is   complete.

 By \eqref{claU} and  $U'\gg\bm0$ in $[0,+\i)$, one obtains that
  \begin{align*}
  u_\i(x)\geq U(|x|-r)\ \ \text{ for all  }r\geq r_0 \text{ and  }x\in\bar\Omega\text{ such that } |x|\geq r
.
\end{align*}
Define
$$
r^*=\inf\e\{r\in\mathbb R: u_\i(x)\geq U(|x|-r)\text{ for all }x\in\bar\Omega
\text{ such that }|x|\geq r\r\}.
$$
By \eqref{claU}, one has $r^*\leq r_0$.
We now turn to prove $r^*=-\i$.
Assume that  $r^*>-\i$. By continuity, there holds
\begin{align*}
 u_\i(x)\geq U(|x|-r^*)\ \ \text{ for all }x\in\bar\Omega
\text{ such that }|x|\geq r^*.
\end{align*}
 Two cases may occur.

{\it Case 1: $\inf\left\{u_\i(x)-U(|x|-r^*):x\in\bar\Omega, \ r^*\leq |x|\leq r_0\right\}\gg\bm0$.}
In this case, by continuity of $ u_\i$ and $U$, there exists $r_*<r^*$ such that
$$
 u_\i(x)\geq U(|x|-r^*)\ \ \text{ for all }x\in\bar\Omega
\text{ such that }r_*\leq|x|\leq r_0.
$$
In particular, $u_\i(x)\geq U(|x|-r^*)$ for all $|x|=r_0$.
By the same arguments as the proof of \eqref{claU}, we get that
$$
u_\i(x)\geq U(|x|-r_*)\ \ \text{ for all }\ \ |x|\geq r_0.
$$
Hence, there holds
$u_\i(x)\geq U(|x|-r^*)$ for all $x\in\bar\Omega$
 such that $|x|\geq r_*$, which contradicts the minimality of $r^*$. Thus, case 1 is ruled out.

{\it Case 2: $\inf\left\{u_\i^i(x)-U_i(|x|-r^*):x\in\bar\Omega,\  r^*\leq |x|\leq r_0\right\}=0$ for some $ i\in\{1,\cdots,m\}$.} By continuity, there exists $\bar x\in\bar\Omega$ such that
$$
u_\i^i(\bar x)=U_i(|\bar x|-r^*)\ \ \text{ and }
\ \ r^*\leq|\bar x|\leq r_0.
$$
Since $U(0)=\bm0$ and $u_\i\gg\bm0$ in $\bar\Omega$, one further gets that $|\bar x|>r^*$.

 Assume first that $\bar x\in\Omega$ and $|\bar x|>0$.
 Note that the set $\Omega\cap\{x\in\R^N:|x|\geq r^*\}$ is connected.
 For $x\in\bar\Omega$, define the function $z=(z_1,\cdots,z_m)$ by $z(x)=u_\i(x)-U(|x|-r^*)$.
 Then $z(x)\geq\bm0$ for $x\in\bar\Omega$ such that $|x|\geq r^*$ and $z_i(\bar x)=0$.
 Applying (A4)  and the strong elliptic maximum principle, one gets that $z_i(x)\equiv0$ for $x\in\bar\Omega$ such that   $|x|\geq r^*$,
 that is, $u_\i^i(x)\equiv U_i(|x|-r^*)$. But $U(|x|-r^*)$ is a strict subsolution of \eqref{1.1} for $x\in\bar\Omega$ such that
  $|x|\geq r^*$, a contradiction.

 Therefore, either  $|\bar x|=0$ or $\bar x\in\p\Omega$.
 In the first case, one gets that $r^*<0$ and $\Omega=\R^N$ since the obstacle $K$ is star-shaped with respect to $\bm\theta$ if it is not empty.
 Since $u_\i^i\in C^1(\mathbb R^N)$ satisfies $u_\i^i(x)\geq U_i(|x|-r^*)$ for $x\in\R^N$ and $U_i'>0$ in $[0,+\i)$, it follows that
 $$
 u_\i^i(\bm\theta)=U_i(-r^*)<U_i(|x|-r^*)\leq u_\i^i(x)\ \text{ for all $|x|>0$}.
 $$
 Hence $\nabla u_\i^i(\bm\theta)=\bm0$.
 But since $U_i'(-r^*)>0$, there exists $\sigma>0$ such that
 $u_\i^i(x)<U_i(|x|-r^*)$ for $x\in\R^N$ such that $0<|x|\leq\sigma$, a contradiction.

 Therefore, $\bar x\in\p\Omega$ (then $K\neq\emptyset$ and $|\bar x|>0$) and
 $$
 u_\i^i( x)>U_i(|x|-r^*)\ \ \text{ for all }
x\in\Omega \text{ such that }|x|\geq r^*.
 $$
 For all $x\in\bar\Omega$, define the function $z=(z_1,\cdots,z_m)$ by $z(x)=u_\i(x)-U(|x|-r^*)$.
 Then $z(x)\geq\bm0$ for $x\in\bar\Omega$ such that $|x|\geq r^*$.
  According to $z_i(\bar x)=0$ and the Hopf boundary lemma, there holds
 $ \nu(\bar x)\cdot \nabla z_i(\bar x)<0$.
 Hence,
 $$
 0=\nu(\bar x)\cdot \nabla u_\i^i(\bar x) <\nu(\bar x)\cdot \nabla (U_i(|\bar x|-r^*))=\e( \nu(\bar x)\cdot\frac{\bar x}{|\bar x|}\r)\times U_i'(|\bar x|-r^*).
 $$
 However, since the obstacle $K$ is star-shaped with respect to $\bm\theta$, one has $\nu(\bar x)\cdot\bar x=-\nu_K(\bar x)\cdot\bar x\leq0$. Together with $U_i'>0$, one has $\nu(\bar x)\cdot \nabla (U_i(|\bar x|-r^*))\leq0$, a contradiction.

 As a conclusion, $r^*=-\i$, which implies that $u_\i(x)\geq U(|x|-r)$ for all $r\in\R$ and $x\in\bar\Omega$ with $|x|\geq r$.
 Since $U(+\i)=\bm1$, by taking $r\to-\i$, there holds $u_\i(x)\geq \bm1$ for all $x\in\bar\Omega$, hence $u_\i\equiv\bm1$ in $\bar\Omega$.

{\it Step 2: directionally convex obstacles}. Assume that the obstacle $K$ is directionally convex with respect to a hyperplane $P=\{x\in\R^N:x\cdot \tilde e=l\}$, where $\tilde e\in\mathbb S^{N-1}$  and $l\in\R$. Since $K$ is compact and $u_\i(x)\to\bm1$ as $|x|\to+\i$, there exists $r_0>0$ large enough such that
$K\subset\{x\in\R^N:|x\cdot \tilde e-l|\leq r_0\}$ and $u_\i(x)\geq (1-\dl q_*)\cdot\bm1$ for all
$x\in\R^N$ such that $|x\cdot \tilde e-l|\geq r_0$. It is easy to check that the function
$x\mapsto U(|x\cdot\tilde e-l|-r_0)$ is a stationary solution of \eqref{1.1} in the domain
$\{x\in\R^N:|x\cdot\tilde e-l|>r_0\}$.  By some similar arguments as  the proof of claim \eqref{claU}, one gets that
$$
u_\i(x)\geq U(|x\cdot\tilde e-l|-r_0)\ \ \text{ for all  }x\in\R^N\text{ such that }|x\cdot \tilde e-l|\geq r_0.
$$

Define
$$
r^*=\inf\e\{r\in\mathbb R: u_\i(x)\geq U(|x\cdot\tilde e-l|-r)\text{ for all }x\in\bar\Omega
\text{ such that }|x\cdot\tilde e-l|\geq r\r\}.
$$
Clearly,  $r^*\leq r_0$. We now turn to prove $r^*=-\i$.
Assume that  $r^*>-\i$. By continuity, there holds
\begin{align*}
 u_\i(x)\geq U(|x\cdot\tilde e-l|-r^*)\ \ \text{ for all }x\in\bar\Omega
\text{ such that }|x\cdot\tilde e-l|\geq r^*.
\end{align*}
 Two cases may occur.

{\it Case 1: $\inf\left\{u_\i(x)-U(|x\cdot\tilde e-l|-r^*):x\in\bar\Omega, \ r^*\leq |x\cdot\tilde e-l|\leq r_0\right\}>0$.} Similar to Case 1 in Step 1, there exists $r_*<r^*$ such that $u_\i(x)\geq U(|x\cdot\tilde e-l|-r_*)$ for all $x\in\bar\Omega$
 such that $|x\cdot\tilde e-l|\geq r_*$, a contradiction with the minimality of $r^*$. Hence, case 1 is impossible.

 {\it Case 2: $\inf\left\{u_\i^i(x)-U_i(|x\cdot\tilde e-l|-r^*):x\in\bar\Omega, \ r^*\leq |x\cdot\tilde e-l|\leq r_0\right\}=0$ for some
 $i=1,\cdots,m$.}
 Then there exists a sequence $(x_n)_{n\in \mathbb N}\subset\bar\Omega$ such that
 $$
 u_\i^i(x_n)-U_i(|x_n\cdot\tilde e-l|-r^*)\to0\text{  as }n\to+\i
 $$
 and
 $$
 r^*\leq |x_n\cdot\tilde e-l|\leq r_0\ \ \text{ for all }n\in\mathbb N.
 $$
 Up to extraction of a subsequence, two cases may occur, either $x_n\to\bar x\in\bar\Omega$ or $|x_n|\to+\i$ as $n\to+\i$.

{\it Subcase 1: $x_n\to\bar x\in\bar\Omega$ as $n\to+\i$.} In this case, one has
$$
 u_\i^i(\bar x)= U_i(|\bar x\cdot\tilde e-l|-r^*)\ \text{  and }\
  r^*\leq |\bar x\cdot\tilde e-l|\leq r_0.
 $$
 Since $U(0)=\bm0$ and $u_\i\gg\bm0$, one has $|\bar x\cdot\tilde e-l|>r^*$.

 Assume first that $\bar x\in\Omega$ and $|\bar x\cdot\tilde e-l|>0$. Then $\bar x\notin P$.
 Note that the sets $\Omega_\pm=\{x\in\Omega:\pm (x\cdot\tilde e-l)>\max(r^*,0)\}$ are connected since $K$ is directional convex with respect to the hyperplane $P$.
 Let $\tilde\Omega\in\{\Omega_-,\Omega_+\}$ be the connected set containing $\bar x$.
 Denote $z(x)= u_\i( x)-U(|x\cdot\tilde e-l|-r^*)$ in $\tilde\Omega$.
 Since $z(x)\geq\bm0$ for $x\in\tilde\Omega$ and $z_i(\bar x)=0$,
  the strong elliptic maximum principle and (A4) give rise to $z_i\equiv0$ in $\tilde\Omega$,
  that is, $u_\i^i( x)=U_i(|x\cdot\tilde e-l|-r^*)$  for $x\in\tilde\Omega$.
 In particular, choose a sequence of points $(\tilde x_n)_{n\in\mathbb N}\subset\tilde\Omega$
 satisfying $|\tilde x_n|\to+\i$ as $n\to+\i$, and $|\tilde x_n\cdot \tilde e-l|=\max(r^*+1,1)$.
  By the property of $u_\i$, it follows  that $u_\i^i(\tilde x_n)\to1$ as  $n\to+\i$.
  But since $U'\gg\bm0$ and $U(+\i)=\bm1$, one gets that
  $U_i(|\tilde x_n\cdot\tilde e-l|-r^*)\to U_i(\max(r^*+1,1)-r^*)<1$ as $n\to+\i$, a contradiction.

 Therefore,  either $|\bar x\cdot\tilde e-l|=0$ or $\bar x\in\p \Omega$ and $|\bar x\cdot\tilde e-l|>0$.
 In the first case, there holds $r^*<0$.
 Note that $\bar x+\mathbb R \tilde e\in\bar\Omega$ since $K$ is directionally convex with respect to the hyperplane $P$.
 Since $u_\i^i\in C^1(\bar\Omega)$ satisfies $u_\i^i(x)\geq U_i(|x\cdot\tilde e-l|-r^*)$ for $x\in\bar\Omega$,
 and $U'(-r^*)>0$, it follows that
 $$
 u_\i^i(\bar x)=U_i(-r^*)<U_i(|x\cdot\tilde e-l|-r^*)\leq u_\i^i(x)\ \text{ for }x\in\bar\Omega \text{ such that }
 |x\cdot\tilde e-l|>0.
 $$
Hence, $\nabla u_\i^i(\bar x)=\bm\theta$.
But since  $U'(-r^*)>0$, there exists $\sigma_0>0$ such that
 $$
 u_\i^i(\bar x+\sigma \tilde e)< U _i(|(\bar x+\sigma\tilde e)\cdot\tilde e-l|-r^*)
 \ \text{ for all }0<\sigma<\sigma_0.
 $$
It is a contradiction.

Therefore, $\bar x\in\p \Omega$ (then $K\neq\emptyset$) and $|\bar x\cdot\tilde e-l|>0$. Then,
 $ u_\i^i( x)> U_i(| x\cdot\tilde e-l|-r^*)$  for all $x\in\Omega$  such that
   $| x\cdot\tilde e-l|\geq r^*$.
Define the function $z=(z_1,\cdots,z_m)$ by $z(x)=u_\i(x)-U(|x\cdot\tilde e-l|-r^*)\geq0$.
 Since $z_i(\bar x)=0$ and  $z_i>0$ in $\Omega$, one infers from the Hopf boundary lemma that
 $ \nu(\bar x)\cdot \nabla z_i(\bar x)<0$.
 Hence,
 $$
 0<-\nu(\bar x)\cdot \nabla z_i(\bar x)=\nu(\bar x)\cdot \nabla (U_i(|\bar x\cdot\tilde e-l|-r^*))=\e( \nu(\bar x)\cdot\tilde e\r)\times\frac{\bar x\cdot\tilde e-l}{|\bar x\cdot\tilde e-l|}U_i'(|\bar x\cdot\tilde e-l|-r^*).
 $$
 However, since the obstacle $K$ is directionally convex with respect to the hyperplane $P$, one has  $\nu(\bar x)\cdot\tilde e\geq0$ when $\bar x\cdot\tilde e-l\leq0$ and  $\nu(\bar x)\cdot\tilde e\leq0$ when $\bar x\cdot\tilde e-l>0$. Together with $U'>0$, one has $\nu(\bar x)\cdot \nabla (U_i(|\bar x|-r^*))\leq0$, a contradiction. Hence, the subcase 1 is ruled out.

 {\it Subcase 2: $|x_n|\to+\i$ as $n\to+\i$.} By $u_\i^i(x)\to1$ as $|x|\to+\i$, there holds $u_\i^i(x_n)\to1$ as  $n\to+\i$.
 On the other hand, $\limsup_{n\to+\i} U(|x_n\cdot\tilde e-l|-r^*)<1$ since $r^*\leq|x_n\cdot\tilde e-l|\leq r_0$. That is a contradiction.

 As a result, $r^*\to-\i$.  Hence, $u_\i(x)\geq U(|x\cdot\tilde e-l|-r)$ for all $r\in\R$ and $x\in\bar\Omega$ with $|x\cdot\tilde e-l|\geq r$.
 Since $U(+\i)=\bm1$, by taking $r\to-\i$, there holds $u_\i(x)\geq \bm1$ for all $x\in\bar\Omega$, hence $u_\i\equiv\bm1$ in $\bar\Omega$.

 {\it  Step 3: complete propagation.} When $K$ is a star-shaped or a directionally convex obstacle, one concludes from Theorem \ref{t1}, step 1 and step 2 that $u(t,x)\to\bm1$ locally uniformly in $x\in\bar\Omega$ as $t\to+\i$, that is, $u$ propagates completely in the sense of \eqref{comp}.
 The proof is complete.
\end{pr}

\section{Application}\label{s6}
The main results of this paper can be applied to a wide range of reaction-diffusion systems, such as a buffering bistable system,
a Lotka-Volterra competition-diffusion model with spatiotemporal delays and a system of $n$ obligate mutualist species.
For specific details, please refer to \cite[Section 5]{wangzhicheng2012} and \cite[Section 4]{SW18}.
In this section, we provide an illustrative application using the Lotka-Volterra competition-diffusion  system as an example.

Consider the following Lotka-Volterra competition-diffusion system with two-components
\begin{align}\label{clv}
\begin{cases}
  (u_1)_t=\Delta u_1 +u_1 (1-u_1 -k_1u_2 ) &\text{ in } \R\times\Omega, \\
(u_2)_t=d\Delta u_2 +ru_2 (1-u_2 -k_2u_1 )  &\text{ in } \R\times\Omega, \\
\nu\cdot\nabla u_1=0,\ \nu\cdot\nabla u_2=0  &\text{ on } \R\times\p\Omega,
\end{cases}
\end{align}
where  $k_1$, $k_2$, $r$ and $d$ are positive constants. The variables $u(t,x)$ and $v(t,x)$
stand for the population density of two competing species. We assume that
$$
k_1>1\text{ and }k_2>1.
$$
It is easy to check that system \eqref{clv} has two stable equilibria $(0, 1)$ and $(1, 0)$.
By \cite[Theorem 2.1]{Kan-on}, system \eqref{clv} admits a  planar traveling front $\Phi(\xi)=(\Phi_1(\xi),\Phi_2(\xi))$
connecting $(0,1)$ and $(1,0)$, where $\xi=x\cdot e+c t$, $c\in\R$,  $e\in \mathbb S^{N-1}$.
In particular, the planar traveling front $\Phi(\xi)$ is unique (up to shifts)
and satisfies $\Phi_1'>0$ and $\Phi_2'<0$ in $\R$.

Assume that $c>0$. In fact, following \cite[Theorems 4.2 and 4.3]{MHO} and \cite[Theorems 3.1 and 3.3]{MZYO}, we know that the wave speed $c$ is positive if one of the following conditions is satisfied:

\noindent (P1) $1<\frac{dk_1}{d-r(k_2-1)}<\frac{2(k_2-1)}{k_2}$;

\noindent (P2) $\frac{r+d(k_1-1)}{k_2r}<3-2k_1$;

\noindent (P3) $k_2>2$, there exists an integer $n\geq 2$ such that $1<k_1<1+\frac{n}{(n-1)(2n-1)}$, and either
$\frac{2d(k_1-1)(n-1)^2}{k_2n^2}<r<\frac{d(k_1-1)(n-1)^2}{n^2}$
or
$\frac{d(k_1-1)(n-1)^2}{k_2n^2}<r<\frac{2d(k_1-1)(n-1)^2}{n^2}$;

\noindent (P4) $\frac{2r+4d(k_1-1)}{rk_2}<3-k_1$.

Set $u^*_1=u_1$ and $u_2^*=1-u_2$, the system \eqref{clv} reduces to (for the sake
of simplicity, we drop the symbol $*$)
\begin{align}\label{clv1}
\begin{cases}
  (u_1)_t  =\Delta u_1  +u_1  (1-k_1-u_1  +k_1u_2  )&\text{ in } \R\times\Omega, \\
(u_2)_t =d\Delta u_2  +r(1-u_2  )(k_2 u_1  -u_2  )&\text{ in } \R\times\Omega, \\
\nu\cdot \nabla u_1  =0,\ \nu\cdot u_2  =0&\text{ on } \R\times\Omega.
\end{cases}
\end{align}
It is easy to see that the equilibria $(0,1)$ and $(1,0)$ change into $(0,0)$ and $(1,1)$, respectively.
Let $\Psi_1(\xi)=\Phi_1(\xi)$ and $\Psi_2(\xi)=1-\Phi_2(\xi)$.
Then $\Psi(\xi)=(\Psi_1(\xi),\Psi_2(\xi))$ is a traveling wave front of \eqref{clv1} connecting $(0,0)$ and $(1,1)$.
In particular, $\Psi_1'>0$ and $\Psi_2'>0$ in $\R$.
It is not difficult to verify that (A1)-(A4) hold for system \eqref{clv1}
with
$$
R_0=(1,2k_2),\ R_1=(2k_1,1),\ \lambda_0=\frac{1}{2}\min\left(\frac{r}{2},k_1-1\right)\text{ and }
\lambda_1=\frac{1}{2}\min\left(\frac{1}{2},r(k_2-1)\right).
$$
Furthermore, it follows from \cite[Lemmas 2.1 and 2.2]{mt} that there exist constants $\lambda>0$, $\mu<0$ and vectors $\alpha^\pm=(\alpha^\pm_1,\alpha^\pm_2)\in \R^2_+$ such that
\begin{align*}
&\lim_{\xi\to-\i}(\Psi_1(\xi),\Psi_2(\xi)) e^{-\mu\xi}=(\alpha_1^-,\alpha_2^-),\
\lim_{\xi\to+\i}(1-\Psi_1(\xi),1-\Psi_2(\xi)) e^{-\lambda\xi}=(\alpha_1^+,\alpha_2^+),\\
&\lim_{\xi\to-\i}(\Psi_1'(\xi),\Psi_2'(\xi)) e^{-\mu\xi}=\mu(\alpha_1^-,\alpha_2^-),\
\lim_{\xi\to+\i}(-\Psi_1'(\xi),-\Psi_2'(\xi)) e^{-\lambda\xi}=\lambda(\alpha_1^+,\alpha_2^+),\\
&\lim_{\xi\to-\i}(\Psi_1''(\xi),\Psi_2''(\xi)) e^{-\mu\xi}=\mu^2(\alpha_1^-,\alpha_2^-),\
\lim_{\xi\to+\i}(-\Psi_1''(\xi),-\Psi_2''(\xi)) e^{-\lambda\xi}=\lambda^2(\alpha_1^+,\alpha_2^+).
\end{align*}
Hence, there exists $k>0$ and $C>0$ such that
\begin{align*}
-k\Psi_1'(\xi)\leq\Psi_1''(\xi)\leq k\Psi_1'(\xi) \text{ and }
-k\Psi_2'(\xi)\leq\Psi_2''(\xi)\leq k\Psi_2'(\xi)\ \
\text{ for }\xi\in\R
\end{align*}
and
\begin{align*}
\Psi_1''(\xi)\leq 0\text{ and }\Psi_2''(\xi)\leq 0\ \ \text{ for }\xi\geq C.
\end{align*}

By Theorem \ref{t1}, system \eqref{clv1} admits a unique time-increasing entire solution
$v=(v_1,v_2)$ satisfying
$v(t,x)-\Psi(x\cdot e+ct)\to(0,0)$  uniformly in $x\in\bar\Omega$ as $t\to-\i$.
Theorems \ref{t2} and \ref{lcomp} imply that, if the obstacle $K$ is star-shaped or directionally convex, then $v(t,x)\to(1,1)$ locally uniformly in $x\in\bar\Omega$ as $t\to+\i$, and
$v(t,x)-\Psi(x\cdot e+ct)\to(0,0)$ uniformly in $x\in\bar\Omega$  as $t\to+\i$.
Consequently, system \eqref{clv} admits a unique entire solution
$u=(u_1,u_2)$ satisfying
\begin{align*}
(u_1(t,x),u_2(t,x))-(\Phi_1(x\cdot e+ct),\Phi_2(x\cdot e+ct))\to(0,0)\ \text{ uniformly in }x\in\bar\Omega \text{ as }t\to-\i,
\end{align*}
where $u_1$ is time-increasing and $u_2$ is time-decreasing.
Furthermore, if the obstacle $K$ is star-shaped or directionally convex, then
$(u_1(t,x),u_2(t,x))\to(1,0)$  locally uniformly in
$x\in\bar\Omega$ as $t\to+\i$,
and
\begin{align*}
(u_1(t,x),u_2(t,x))-(\Phi_1(x\cdot e+ct),\Phi_2(x\cdot e+ct))\ \text{ uniformly in }x\in\bar\Omega \text{ as }t\to+\i.
\end{align*}

Based on  above conclusions, we will further study the dynamic behavior of the Lotka-Volterra competition-diffusion system \eqref{clv} in the exterior domains in future research.

\section*{Acknowledgments}
The first author would like to give her sincere thanks to China
Scholarship Council for a 18-month visit of Aix Marseille University.
Her work was also  partially supported by NSF
of China (12171120).
The second author's work was partially
supported by NSF of China (12171120).

\section*{Date availability statements}
We do not analyse or generate any datasets, because our work proceeds within a theoretical and mathematical approach.

\section*{Conflict of interest}
There is no conflict of interest to declare.

\end{document}